\documentclass[11pt]{article}
\usepackage[margin=1.1in]{geometry}
\usepackage{amsmath,amssymb,amsthm,mathtools}
\usepackage{enumitem}
\usepackage{booktabs,array,tabularx}
\usepackage{ltablex}
\keepXColumns
\usepackage[protrusion=true,expansion=false]{microtype}
\usepackage[colorlinks=true,linkcolor=black,citecolor=black,urlcolor=black,
  hypertexnames=false]{hyperref}

\newtheorem{theorem}{Theorem}[section]
\newtheorem{proposition}[theorem]{Proposition}
\newtheorem{lemma}[theorem]{Lemma}
\newtheorem{corollary}[theorem]{Corollary}
\newtheorem{quotedthm}[theorem]{Theorem}
\theoremstyle{definition}
\newtheorem{definition}[theorem]{Definition}
\newtheorem{openprob}[theorem]{Problem}
\theoremstyle{remark}
\newtheorem{remark}[theorem]{Remark}

\DeclareMathOperator{\Tr}{Tr}
\DeclareMathOperator{\sinc}{sinc}
\DeclareMathOperator{\Rey}{Re}
\DeclareMathOperator*{\Res}{Res}
\newcommand{\R}{\mathbb{R}}
\newcommand{\Z}{\mathbb{Z}}
\newcommand{\C}{\mathbb{C}}
\newcommand{\F}{\mathcal{F}}
\newcommand{\eps}{\varepsilon}
\newcommand{\Ntil}{\widetilde N}
\newcommand{\BG}{\mathsf{G}}
\newcommand{\Lb}{\bar L}
\newcommand{\dd}{\,d}

\title{Uniform sine-kernel determinant asymptotics,
tail-side quantiles,\\
and prolate eigenvalue bounds}
\author{Ahmadreza Azimifard}
\date{August 2026 \quad (revised version)}

\hypersetup{%
  pdftitle={Uniform sine-kernel determinant asymptotics, tail-side
quantiles, and prolate eigenvalue bounds},
  pdfauthor={Ahmadreza Azimifard},
  pdfsubject={Spectral theory of time-band limiting operators;
Riemann-Hilbert analysis of sine-kernel Fredholm determinants},
  pdfkeywords={prolate spheroidal wave functions, time-band limiting
operator, sine kernel, Fredholm determinant, Riemann-Hilbert problem,
Barnes G-function, plunge region, eigenvalue counting function, quantile asymptotics}}

\begin{document}
\maketitle

\begin{abstract}

  Let $S_c=P_{(0,c)}QP_{(0,c)}$ be the one-dimensional sinc-kernel concentration operator,
   let $N_a(c)=\#\{n:\lambda_n(c)>a\}$, and set $\bar L=\log((1-\delta)/\delta)$. We prove,
    uniformly for each fixed $A>0$, the tail-side quantile formula $N_\delta(c)=c+\pi^{-2}\bar L\log(4\pi^2c/\bar L)+O_A(\log c+\bar L)$ for $6\le\bar L\le A\log c$.
     It yields corresponding additive formulas for the lower half and full plunge, with main terms respectively $\pi^{-2}\bar L\log(4\pi^2c/\bar L)$ and twice this quantity. 
     An exact one-tail-coordinate selection gives, for fixed $A>0$, $d\ge1$, and $q\in(1/2,1)$, the one-sided tensor-product
      bound $\Lambda_\delta(c;d)\ge\pi^{-2}d c^{d-1}\bar L\log(4\pi^2c/\bar L)-O_{A,d,q}(c^{d-1}(\log c+\bar L))$ for $L_{d,q}\le\bar L\le A\log c$,
       where $L_{d,q}=\log(q^{-(d-1)}(e^6+1)-1)$; the tensor content is nontrivial for $d\ge2$.

The analytic input is a signed growing-parameter sine-kernel determinant asymptotic: uniformly for
 $0\le\omega\le A\log s$, $\log\det(I+(e^{2\omega}-1)K_s)=4\omega s/\pi+2\pi^{-2}\omega^2\log(4s)+2\log|G(1+i\omega/\pi)|^2+O_A((1+\omega)^4\log^2s/s)$, where $G$ is
  the Barnes $G$-function. We prove this negative-coupling counterpart of the Bothner--Deift--Its--Krasovsky theorem by direct IIKS steepest descent. We also retain 
  the uniform head-side results and use a two-way determinant reduction to obtain the moving-depth lower-half bridge bound with constant $1/(32\pi^2)$; extending it
   to the deeper range uses Kulikov--Dam Larsen and may require a smaller constant. These counting formulas are additive. Their errors become uniformly relative
    when $\bar L$ tends uniformly to infinity; fixed thresholds are covered separately by Landau--Widom. A Lambert-$W_{-1}$ formula is recorded only for the
     continuous main term, not for individual eigenvalues.


\end{abstract}

\tableofcontents

\section{Introduction and statement of results}\label{sec:intro}

\subsection{The problem}\label{sec:problem}

Fix the unitary Fourier transform
$\F f(\xi)=\int_\R f(x)e^{-2\pi ix\xi}\,dx$, let $Q$ be the orthogonal
projection of $L^2(\R)$ onto the Paley--Wiener space of functions with
Fourier support in $(-\tfrac12,\tfrac12)$, and for $c>0$ put
\begin{equation}\label{eq:op}
S_c \;=\; P\,Q\,P \quad\text{on } L^2(0,c),\qquad P:=P_{(0,c)},
\end{equation}
with eigenvalues $1>\lambda_0(c)\ge\lambda_1(c)\ge\cdots>0$; they are
in fact simple and strictly decreasing \cite{SlepianPollak}. About $c$ of
them cluster at $1$, the rest decay to $0$, and the transition region ---
the \emph{plunge} --- has logarithmic width. For $0<a<1$ set
\[
N_a(c)=\#\{n:\lambda_n(c)>a\},
\]
and for $0<\delta<\tfrac12$ define the half-window counts
\begin{equation}\label{eq:counts}
\Lambda^+_\delta(c)=\#\bigl\{n:\tfrac12<\lambda_n(c)\le1-\delta\bigr\},
\qquad
D(\delta,c)=\Lambda^-_\delta(c)
   =\#\bigl\{n:\delta<\lambda_n(c)\le\tfrac12\bigr\},
\end{equation}
so that, exactly,
\begin{equation}\label{eq:identities}
N_{1-\delta}=N_{1/2}-\Lambda^+_\delta,\qquad
N_{\delta}=N_{1/2}+D(\delta,c),\qquad
\Lambda_\delta:=\Lambda^+_\delta+D(\delta,c)
  =\#\{\delta<\lambda_n\le1-\delta\}.
\end{equation}
(The closed-top conventions in \eqref{eq:counts} are chosen so that
\eqref{eq:identities} are identities; since the spectrum is simple, they
differ from any other convention by at most $1$ at the measure-zero set
of thresholds meeting an eigenvalue, and every estimate below carries
$O(\ln c)$ slack.) Throughout, $\log=\ln$ is natural,
\[
L=L(\delta)=\ln\tfrac1\delta,\qquad
\Lb=\Lb(\delta)=\ln\tfrac{1-\delta}\delta\in[L-\ln2,\;L],\qquad
R_\alpha=R_\alpha(\delta,c)=\ln\frac{\alpha c}{L}.
\]

We formulate the bridge problem as follows.

\begin{openprob}[uniform lower-half density on the bridge]\label{op:bridge}
Do there exist constants $\alpha\ge4$, $\delta_0\in(0,\tfrac12)$,
$\kappa>0$, $C<\infty$, $c_0<\infty$ such that
\begin{equation}\label{eq:bridge}
D(\delta,c)\;\ge\;\kappa\,L\,\ln\!\Bigl(\frac{\alpha c}{L}\Bigr)\;-\;C
\tag{BRIDGE}
\end{equation}
for every $c\ge c_0$ and every $\delta$ with
$\alpha^{-c}<\delta\le\delta_0$?
\end{openprob}

Before the present work, three of the four threshold regimes were
settled. At fixed $\delta$, Landau--Widom \cite{LandauWidom} give the
asymptotic equality $D=\pi^{-2}\ln\frac{1-\delta}\delta\,\ln c+o(\ln c)$,
with no uniformity as $\delta$ moves with $c$. On the deep polynomial
range $\alpha_1^{-c}<\delta<c^{-\alpha_1}$, Kulikov and Larsen
\cite{KDL} prove the two-sided order $D\asymp L\,R_{\alpha_1}$, and below
$\alpha_1^{-c}$ they determine the order outright. For the lower-bound
problem just stated, what remained open was exactly the moving-threshold
range $c^{-\alpha}\le\delta\le\delta_0$ --- the \emph{bridge} --- where the
obstruction is already visible in \cite{KDL}'s ranges. Of the four
one-dimensional counting inputs their method assembles, the upper bound on
$N_{1-\eps}$ and the lower bound on $N_\eps$ are available only for
$\eps<c^{-A}$. The pointwise formulas of
\cite{Kulikov26,BonamiKarouiDecay}, in the form extracted in
\cite[\S2]{KDL}, yield the requisite two-sided exponential-rate estimates
only outside an index window of width $\asymp\log^2c$ around $n=c$; on the
decay side, the $O(\log n)$ error in the Bonami--Karoui exponent is what
forces this restriction.

This paper proves a block of head-side theorems, an exact two-way
reduction of \eqref{eq:bridge} to a tail-side determinant bound, and the
required signed growing-parameter determinant estimate.  The last input
closes the moving bridge; the deep-range theorem of \cite{KDL} covers
the complementary interval.
\subsection{Main results: the tail-side quantile theorem and its
consequences}\label{sec:results}

Everything below is a statement about one object. Put
\[
F_c(v)\;=\;\ln\det\bigl(I-\gamma(v)S_c\bigr),\qquad
\gamma(v)=1-e^{-2v}\qquad(v\in\R),
\]
a single function of a single real parameter. Its \emph{head half}
$v>0$ weighs the eigenvalues near $1$ and its \emph{tail half} $v<0$
weighs those near $0$: the $v$-derivative of $F_c$ is (up to a factor
$-2$) a sigmoid-smoothed counting function of the spectrum at threshold
$\theta(v)=(1+e^{-2v})^{-1}$, so the two halves of the plunge are the
two signs of one parameter in one determinant. The head half follows from
Theorem~\ref{qt:BDIK}.  The
logarithmic tail window is treated here by a signed Riemann--Hilbert
argument in Section~\ref{sec:tail}.  The lower-half bridge estimate
\eqref{eq:bridge} is exactly a consequence of that tail half through
Theorem~\ref{thm:red}.

This revision adds the tail-side quantile theorem, which is stated
first.  It also corrects the identification of the origin-sector
multipliers in the signed Riemann--Hilbert proof and supplies a complete
verification of the third-ray jump.  The head-side quantile theorem, the
two-way reduction, the bridge estimate, and the signed determinant
theorem from the first version are collected in
\S\ref{sec:legacyresults}.  Throughout, for
$0<\delta<\tfrac12$ we abbreviate
\begin{equation}\label{eq:Phiintro}
 \Lb=\ln\frac{1-\delta}{\delta},\qquad
 \Phi_c(x)\;=\;\frac{x}{\pi^2}\,\ln\frac{4\pi^2c}{x}
 \qquad(x>0,\ c>0).
\end{equation}
The statements of this subsection are proved in
Section~\ref{sec:tailquant}.  They are consequences of the signed
determinant theorem, Theorem~\ref{thm:signedintro} below, together with
the head-side machinery of
Sections~\ref{sec:barnes}--\ref{sec:thmA}.

\begin{theorem}[tail-side quantile theorem; proved in
\S\ref{sec:tq-quantile}]\label{thm:tail}
For every fixed $A>0$ there are constants $c_A,C_A<\infty$, depending
only on $A$, such that for every $c\ge c_A$ and every
$\delta\in(0,\tfrac12)$ with
\begin{equation}\label{eq:tailrange}
 6\;\le\;\Lb\;=\;\ln\frac{1-\delta}{\delta}\;\le\;A\ln c ,
\end{equation}
one has
\begin{equation}\label{eq:TQ}
 \Bigl|\,N_\delta(c)\;-\;c\;-\;\frac{\Lb}{\pi^2}
 \ln\frac{4\pi^2c}{\Lb}\,\Bigr|\;\le\;C_A\,(\ln c+\Lb).
\end{equation}
\end{theorem}

The order of the quantifiers is part of the statement: $A$ is fixed
first, and $c_A,C_A$ may then depend on $A$, but on nothing else --- in
particular not on $c$, on $\delta$, or on $\Lb$ within
\eqref{eq:tailrange}.  Nothing here is uniform in $A$, and the upper
endpoint $A\ln c$ of \eqref{eq:tailrange} is genuinely shorter than the
head-side endpoint $c^{1/3}$ of the head-side quantile theorem stated
below (Theorem~\ref{thm:A}): it is inherited from the logarithmic window
of the signed determinant theorem, Theorem~\ref{thm:signedintro} below.

\begin{corollary}[lower-half and full-plunge formulas; proved in
\S\ref{sec:tq-quantile}]\label{cor:plunge}
Fix $A>0$.  Uniformly on the range \eqref{eq:tailrange},
\begin{align}
 D(\delta,c)&=\Phi_c(\Lb)+O_A(\ln c+\Lb),
 \label{eq:Dformula}\\
 \Lambda_\delta(c)&=2\,\Phi_c(\Lb)+O_A(\ln c+\Lb).
 \label{eq:Pformula}
\end{align}
\end{corollary}

\begin{remark}[additive estimate and relative subranges]
\label{rem:relativelimit}
Theorem~\ref{thm:tail} and Corollary~\ref{cor:plunge} are \emph{additive}
statements.  On \eqref{eq:tailrange} one has
$\Phi_c(\Lb)\ge\tfrac23\pi^{-2}\Lb\ln c$ for $c$ large, so the ratio of
the displayed error to the main term is
$O_A\bigl(\Lb^{-1}+(\ln c)^{-1}\bigr)$.  In particular, it tends to $0$
uniformly on every subrange
\begin{equation}\label{eq:relrange}
 L_0(c)\le\Lb\le A\ln c,\qquad L_0(c)\to\infty .
\end{equation}
On such subranges, \eqref{eq:TQ}, \eqref{eq:Dformula}
and \eqref{eq:Pformula} yield the relative asymptotics
$D(\delta,c)=(1+o_A(1))\Phi_c(\Lb)$ and
$\Lambda_\delta(c)=(1+o_A(1))\,2\Phi_c(\Lb)$.  If $\Lb$ remains bounded,
this additive estimate alone gives no relative conclusion; for each fixed
threshold the separate Landau--Widom theorem does give a relative
asymptotic for the one-dimensional counts.  No bounded-depth relative
interpretation of the displayed tensor coefficient is asserted.  The
analogous limitation of the joint-range upper bound is recorded below in
Corollary~\ref{cor:upperD}.
\end{remark}

For the cube pair $A_d=(0,c)^d$, $B_d=(-\tfrac12,\tfrac12)^d$, write
$\mu_{\mathbf n}=\prod_{r=1}^d\lambda_{n_r}(c)$ for the eigenvalue of the
tensor slot $\mathbf n=(n_1,\dots,n_d)$ (Lemma~\ref{lem:tensor}) and
\begin{equation}\label{eq:tensorcount}
 \Lambda_\delta(c;d)=\#\bigl\{\mathbf n:\delta<\mu_{\mathbf n}
 \le1-\delta\bigr\},
\end{equation}
counted with slot multiplicity, consistently with the
$d$-dimensional block stated below in Corollary~\ref{cor:tensor}.  The
next two statements select \emph{one}
tail coordinate rather than $d$ of them; this is what converts the
one-dimensional lower-half formula \eqref{eq:Dformula} into a bound of
surface order $c^{d-1}$.

\begin{theorem}[exact one-tail-coordinate tensor block; proved in
\S\ref{sec:tq-tensor}]\label{thm:tensorblock}
Let $d\ge1$ be an integer, let $q\in(\tfrac12,1)$, let
$0<\delta<\tfrac12$, and put $\delta'=\delta\,q^{-(d-1)}$.  If
$\delta'<\tfrac12$, then
\begin{equation}\label{eq:tensorblock}
 \Lambda_\delta(c;d)\;\ge\;d\,N_q(c)^{\,d-1}\,D(\delta',c).
\end{equation}
For $d=1$ we use the empty-product convention $N_q(c)^0=1$.
This is an exact inequality, valid for every $c>0$.
\end{theorem}

\begin{theorem}[one-sided tensor surface bound; proved in
\S\ref{sec:tq-tensor}]\label{thm:surface}
Fix $A>0$, an integer $d\ge1$, and $q\in(\tfrac12,1)$.  Put
\begin{equation}\label{eq:Ldq}
 \beta=q^{-(d-1)},\qquad
 L_{d,q}=\ln\bigl(\beta(e^6+1)-1\bigr).
\end{equation}
There are constants $c_{A,d,q},C_{A,d,q}<\infty$ such that whenever
\begin{equation}\label{eq:surfrange}
 c\ge c_{A,d,q},\qquad L_{d,q}\le\Lb\le A\ln c ,
\end{equation}
one has
\begin{equation}\label{eq:surface}
 \Lambda_\delta(c;d)\;\ge\;
 \frac{d}{\pi^2}\,c^{\,d-1}\,\Lb\,\ln\frac{4\pi^2c}{\Lb}
 \;-\;C_{A,d,q}\,c^{\,d-1}(\ln c+\Lb).
\end{equation}
Consequently, on any subrange \eqref{eq:relrange} with
$L_0(c)\ge L_{d,q}$ and $L_0(c)\to\infty$,
\begin{equation}\label{eq:surfacerel}
 \Lambda_\delta(c;d)\;\ge\;
 \Bigl(\frac{d}{\pi^2}-o_{A,d,q}(1)\Bigr)\,
 c^{\,d-1}\,\Lb\,\ln\frac{4\pi^2c}{\Lb},
\end{equation}
uniformly.  No matching upper bound and no matching leading coefficient
is claimed.
\end{theorem}

On the common growing-depth range, this lower bound has the same surface
order as the product-box upper bound in
\cite[Corollary~1.3]{AzimifardProduct}; no matching leading coefficient
or two-sided tensor asymptotic is asserted.

At $d=1$, \eqref{eq:tensorblock} is the trivial inequality
$\Lambda_\delta(c;1)\ge D(\delta,c)$ and \eqref{eq:surface} is
\eqref{eq:Dformula} read as a lower bound; the content is at $d\ge2$,
where Theorem~\ref{thm:surface} is a growing-threshold refinement of the
fixed-window lower bound $\Lambda_{\delta_\ast}(c;d)\gtrsim c^{d-1}\ln c$
discussed below, after Corollary~\ref{cor:tensor}.  Unlike that
corollary, which selects all $d$ coordinates from the upper half and
therefore produces a $d$-th power of a logarithmic quantity,
Theorem~\ref{thm:surface} keeps $d-1$ coordinates in the cluster
$\{\lambda_n>q\}$ and pays for them with the factor $c^{d-1}$.

Finally we record the exact scalar inversion of the main term
$\Phi_c$, and what it does \emph{not} say.

\begin{lemma}[scalar inverse of the continuous main term; proved in
\S\ref{sec:tq-lambert}]\label{lem:lambert}
Let $c>0$ and $0<m<4c/e$.  Then $m=\Phi_c(x)$ has exactly one solution
$x$ in $(0,4\pi^2c/e)$, namely
\begin{equation}\label{eq:lambert}
 x\;=\;-\,\frac{\pi^2m}{W_{-1}\bigl(-m/(4c)\bigr)} ,
\end{equation}
where $W_{-1}$ is the branch of the Lambert $W$ function with
$W_{-1}(y)\le-1$ for $-1/e<y<0$.
\end{lemma}

Lemma~\ref{lem:lambert} inverts the \emph{continuous} main term of
\eqref{eq:TQ} and nothing else.  It is not an individual-eigenvalue
theorem, and none is claimed anywhere in this paper: converting
\eqref{eq:TQ} into a statement about $\lambda_n(c)$ for an individual
index $n$ would require a generalized inverse of a step function, an
explicit admissible index range, a convention for integer jumps, and the
propagation of the $O_A(\ln c+\Lb)$ counting error through that inverse.
See Remark~\ref{rem:nolambertthm}.

\subsection{Supporting results: the head side, the two-way reduction,
and the signed determinant theorem}\label{sec:legacyresults}

The supporting results in this subsection were already stated in the
first version of the paper.  They are included here because
Theorem~\ref{thm:tail} and its consequences rest on them: the signed
determinant theorem,
Theorem~\ref{thm:signedintro}, is the analytic input to the tail-side
proof of \S\ref{sec:tailquant}, and the head-side quantile theorem,
Theorem~\ref{thm:A}, supplies the counts $N_{1/2}(c)$ and
$N_{1-\delta}(c)$ used in Corollary~\ref{cor:plunge}.  Their conclusions
are retained; the signed-determinant proof in
Section~\ref{sec:tail} has been expanded and corrected in this revision.

The engine on the head side is a published theorem which, to our
knowledge, has not previously been applied
to prolate counting: the \emph{uniform} sine-kernel determinant
asymptotics of
Bothner, Deift, Its and Krasovsky \cite{BDIK2}, which extend the
classical Basor--Widom formula \cite{BasorWidom83,Charlier21} from a
fixed generating parameter $\gamma\in(0,1)$ to $\gamma=1-e^{-2v}$ with
$v$ as large as $s^{1/3}$ ($s=$ half the rescaled interval length), with
an explicit error $O((v+v^3)/s)$. In the eigenvalue picture the
determinant's logarithmic $v$-derivative is a sigmoid-smoothed counting
function of the spectrum at threshold $1-\eps$,
$\ln\frac{1-\eps}\eps=2v$;
the theorem therefore controls head-side counts at depths up to
$\Lb\asymp c^{1/3}$ --- far beyond the $\log^2 c$ window, and covering
every bridge depth $\Lb\le\alpha\ln c$ with room to spare.

\begin{theorem}[head-side quantile theorem; proved in
\S\ref{sec:thmA}]\label{thm:A}
There exist absolute constants $C_{\mathrm h}<\infty$ and $c_{\mathrm h}<\infty$ such that
for all $c\ge c_{\mathrm h}$ and all $\eps$ with
$6\le\Lb=\ln\frac{1-\eps}{\eps}\le c^{1/3}$,
\[
\Bigl|\,N_{1-\eps}(c)\;-\;c\;+\;\frac{\Lb}{\pi^2}
 \ln\frac{4\pi^2c}{\Lb}\,\Bigr|\;\le\;C_{\mathrm h}\,(\ln c+\Lb).
\]
Moreover $|N_{1/2}(c)-c|\le C_{\mathrm h}\ln c$.
\end{theorem}

The subscript $\mathrm h$ marks these as \emph{head-side} constants.
They are absolute: no parameter is quantified in Theorem~\ref{thm:A},
whose range $6\le\Lb\le c^{1/3}$ involves none.  They are unrelated to
the fixed-$A$ constants $C_A,s_A$ of Theorem~\ref{thm:signedintro} and
of \S\ref{sec:tail}, and to the fixed-$A$ constants of
Theorem~\ref{thm:tail}, all of which do depend on a parameter fixed in
advance.

In the notation \eqref{eq:Phiintro} the main display of
Theorem~\ref{thm:A} reads
$N_{1-\eps}(c)=c-\Phi_c(\Lb)+O(\ln c+\Lb)$, the exact head-side mirror
of \eqref{eq:TQ}.

Since $\Lb\ln(4\pi^2c/\Lb)\ge\tfrac23\,\Lb\ln c$ on the stated range, the
error is dominated by the main term as soon as $\Lb\ge C$ ---
i.e.\ the theorem is a genuine two-sided asymptotic
$N_{1-\eps}=c-\pi^{-2}\Lb\ln(4\pi^2c/\Lb)(1+o(1))$ whenever
$\Lb\to\infty$, uniformly.

It is convenient to record once the form of Theorem~\ref{thm:A} that the
next two results consume. By \eqref{eq:identities},
$\Lambda^+_\delta=N_{1/2}-N_{1-\delta}$, so both parts of
Theorem~\ref{thm:A}, together with $L-\ln2\le\Lb\le L$ (the passage from
$\Lb$ to $L$ costing a further $O(\ln c)$, since
$\frac{d}{dx}\frac{x}{\pi^2}\ln\frac{4\pi^2c}{x}
=\frac{1}{\pi^2}(\ln\frac{4\pi^2c}{x}-1)=O(\ln c)$ on the range), supply
an absolute constant $C_\star<\infty$ with
\begin{equation}\label{eq:Cstar}
\Bigl|\,\Lambda^+_\delta(c)-\frac{L}{\pi^2}\ln\frac{4\pi^2c}{L}\,\Bigr|
\;\le\;C_\star\,\bigl(\ln c+L\bigr)
\qquad\text{for }c\ge c_{\mathrm h}\text{ and }6\le\Lb\le c^{1/3}.
\end{equation}

\begin{theorem}[upper-half density on the bridge; proved in
\S\ref{sec:thmB}]\label{thm:B}
For every $\eta\in(0,1)$ there exist $L_0=L_0(\eta)$ and $c_0=c_0(\eta)$
such that for all $c\ge c_0$ and all $\delta$ with
$L_0\le L=\ln\frac1\delta\le c^{1/3}$,
\[
(1-\eta)\,\frac{L}{\pi^2}\,\ln\frac{4\pi^2c}{L}
\;\le\;\Lambda^+_\delta(c)\;\le\;
(1+\eta)\,\frac{L}{\pi^2}\,\ln\frac{4\pi^2c}{L}.
\]
In particular, for every $\alpha\ge4$ the exact mirror of
\eqref{eq:bridge} holds: with $\kappa=\tfrac1{2\pi^2}$, $C=0$,
$\delta_0=e^{-L_0}$ and $c\ge c_0(\alpha)$,
\[
\Lambda^+_\delta(c)\;\ge\;\frac{1}{2\pi^2}\,L\,
\ln\!\Bigl(\frac{\alpha c}{L}\Bigr)
\qquad\text{for all }c^{-\alpha}\le\delta\le\delta_0 .
\]
\end{theorem}

We are not aware of either statement being available previously on a
moving range $\delta\ge c^{-A}$: the published upper bounds on
$N_{1-\eps}$ we know of stop at $\eps<c^{-A}$ \cite[\S2]{KDL}, and the
pointwise route of \cite{Kulikov26} requires $c-n\gtrsim\log^2c$.
When $L\to\infty$, Theorem~\ref{thm:B} identifies the upper-half leading
coefficient as $\pi^{-2}$.  No corresponding sharpness statement is
made here for bounded $L$.

Two corollaries follow. Tensorization (restated and reproved as
Lemma~\ref{lem:tensor}) gives a threshold-dependent lower block in the
$d$-dimensional plunge count:

\begin{corollary}[$d$-dimensional bridge block; proved in
\S\ref{sec:thmB}]\label{cor:tensor}
Let $d\ge2$ and let $\Lambda_\eps(c;d)$ denote the plunge count
$\#\{\eps<\mu\le1-\eps\}$ of the concentration operator of the cube pair
$A=(0,c)^d$, $B=(-\tfrac12,\tfrac12)^d$. For every $\eta\in(0,1)$ there
are $L_0,c_0$ such that for $c\ge c_0$ and
$L_0\le\ln\frac1\delta\le c^{1/3}$ with $\delta\le2^{-d}$,
\[
\Lambda_\delta(c;d)\;\ge\;
\Bigl[(1-\eta)\,\frac{L}{\pi^2}\,\ln\frac{4\pi^2c}{L}\Bigr]^{d}.
\]
\end{corollary}

This block is not the first bridge-uniform lower bound in dimension
$d\ge2$, and it is not the strongest one in its $c$-dependence.  For a
fixed $\delta_\ast\in(0,\tfrac12)$, fixed-window asymptotics for the cube
pair give $\Lambda_{\delta_\ast}(c;d)\asymp c^{d-1}\ln c$; monotonicity in
the threshold consequently gives the same lower bound uniformly for
$0<\delta\le\delta_\ast$.  This follows, for example, by applying the
fixed-window trace asymptotics in \cite[Thm.~1.11, eq.~(1.15)]{KDL}
to an indicator supported inside $(0,1)$.  That lower bound is stronger
in $c$ for fixed $d\ge2$, whereas Corollary~\ref{cor:tensor} records the
different, explicit dependence on a growing threshold depth $L$ obtained
directly from the one-dimensional upper half.  We make no sharpness or
priority claim for it.  Combining Theorem~\ref{thm:B} with the fully
explicit Karnik--Romberg--Davenport bound \cite{KRD} gives the following
joint-range upper bound on the lower half:

\begin{corollary}[joint-range upper bound on the lower half; proved in
\S\ref{sec:thmB}]\label{cor:upperD}
There are absolute constants $L_0,c_0,C',C_2$ --- with
\[
C_2\;=\;\frac{2\ln51-\ln4\pi^2}{\pi^2}\;+\;C_\star
\;=\;0.4243\ldots+C_\star ,
\]
$C_\star$ as in \eqref{eq:Cstar} --- such that for $c\ge c_0$ and
$L_0\le L\le c^{1/3}$,
\begin{equation}\label{eq:upperD-a}
D(\delta,c)\;\le\;\frac{L}{\pi^2}\,\bigl(\ln c+\ln L\bigr)
\;+\;C_2\,L\;+\;C'\ln c .
\end{equation}
Consequently, for every $\eta\in(0,1)$ there is $c_0(\eta)$ such that
for $c\ge c_0(\eta)$ and $L_0\le L\le c^{1/3}$,
\[
D(\delta,c)\;\le\;(1+\eta)\,\frac{L}{\pi^2}\,\ln c\;+\;C_\eta\ln c
\quad\text{for }L\le c^{\,\eta/2}.
\]
Consequently, this joint-range bound identifies the coefficient
$\pi^{-2}$ along subregimes with $L=L(c)\to\infty$; when $L$ is bounded,
the additive $C_\eta\ln c$ term has the same order as the displayed
nominal leading term, so this bound alone does not identify the
coefficient.
\end{corollary}

\begin{remark}\label{rem:formb}
Since $\pi^2C_2=\ln\bigl(51^2/(4\pi^2)\bigr)+\pi^2C_\star$, the term
$C_2L$ of \eqref{eq:upperD-a} folds exactly back into the logarithm:
$D(\delta,c)\le\frac{L}{\pi^2}\ln\bigl(K\,c\,L\bigr)+C'\ln c$ with the
absolute constant
$K=e^{\pi^2C_2}=\frac{51^2}{4\pi^2}e^{\pi^2C_\star}
=65.884\ldots\cdot e^{\pi^2C_\star}$.
\end{remark}

The $\Theta(L)$ term in \eqref{eq:upperD-a} is not removable by the
present method: it enters twice, once as $(2\ln51-\ln4\pi^2)L/\pi^2$ from
the explicit constants inside the logarithms of
Theorem~\ref{qt:KRD} and Theorem~\ref{thm:A}, and once as $C_\star L$
from Theorem~\ref{thm:A}'s own error term, which is $\Theta(\Lb)$ by
Steps 3--4 of \S\ref{sec:thmA}. On the bridge, where
$L\le\alpha\ln c$, it is absorbed by the $O(\ln c)$ slack and the clean
form survives; at the deep end $L\asymp c^{1/3}$ it is not.

The coefficient $\pi^{-2}$ agrees with the Landau--Widom fixed-threshold
value, but the upper bound above does not establish that coefficient
uniformly when $L$ remains bounded.  The required lower bound comes from
the signed determinant theorem below.

The third result identifies exactly what a lower bound must come from.
Write $\beta=e^u-1$, $u>0$, so that
$\det(I+\beta S_c)=\prod_n(1+\beta\lambda_n(c))$ is the moment
generating function $\mathbb{E}\,e^{uN}$ of the sine point process
counting function on $(0,c)$, evaluated in the \emph{upper} tail.

\begin{definition}[the tail-side determinant bound]\label{def:T}
Let $\alpha\ge4$. Say that $(T_\alpha)$ holds if there exist
$\kappa_T>0$, $L_T<\infty$, $c_T<\infty$ such that for all $c\ge c_T$
and all $u\in[L_T,\;\alpha\ln c]$,
\begin{equation}\label{eq:T}
\ln\det\bigl(I+(e^u-1)S_c\bigr)\;\ge\;u\,c\;+\;\kappa_T\,u^2\ln c .
\tag{$T_\alpha$}
\end{equation}
\end{definition}

In the variable of $F_c$ this reads $F_c(-u/2)\ge uc+\kappa_Tu^2\ln c$.
The signed determinant theorem of Section~\ref{sec:tail} supplies this
lower bound for every fixed logarithmic window.

\begin{theorem}[signed determinant asymptotic; established in
Section~\ref{sec:tail}, Theorem~\ref{thm:signedmain}]
\label{thm:signedintro}
For every $A>0$ there are constants $s_A\ge5$ and $C_A>0$ such that, uniformly for
$s\ge s_A$ and $0\le\omega\le A\ln s$,
\begin{align}
 \ln\det\bigl(I+(e^{2\omega}-1)K_s\bigr)
 ={}&\frac{4\omega s}{\pi}+\frac{2\omega^2}{\pi^2}\ln(4s)\notag\\
 &+2\ln\!\left[
 \BG(1+i\omega/\pi)\BG(1-i\omega/\pi)\right]
 +\mathcal R_A(s,\omega),\label{eq:signedintro}
\end{align}
where
\[
 |\mathcal R_A(s,\omega)|
 \le C_A\frac{(1+\omega)^4\ln^2s}{s}.
\]
\end{theorem}

\begin{remark}[what is and is not new in Theorem~\ref{thm:signedintro}]
\label{rem:whatisnew}
The displayed formula is formally the formula of
\cite[Thm.~1.2]{BDIK2} read at $v=-\omega<0$.  It is \emph{not} obtained
from that theorem: \cite[Thm.~1.2]{BDIK2} is stated only for
$\gamma\in[0,1)$, i.e.\ $v\ge0$, and nothing in \cite{BDIK1,BDIK2}
asserts the negative-parameter case.  The fixed-$\omega$ case is
classical \cite{BasorWidom83,BudylinBuslaev} and is available with an
explicit fixed-parameter $O_\omega(s^{-1})$ error from
\cite[eq.~(1.4)]{Charlier21}.  The compact-uniform version in
\cite[Thm.~1.1]{Charlier21} has the weaker error $O(\ln s/s)$, and
neither statement supplies growing-$\omega$ information.  What is
established here is the
uniformity on the logarithmic window $\omega\le A\ln s$, on the side
$\gamma<0$.  The two Stokes-ray jumps of the local model are derived in
Lemma~\ref{lem:modeljumps}; Lemma~\ref{qmi:third} derives the third-ray
jump and the full origin behaviour.  The
error $O_A((1+\omega)^4\ln^2s/s)$ is weaker than the
$O((v+v^3)/s)$ of \cite[Thm.~1.2]{BDIK2}, and the range $\omega\le A\ln s$
is much shorter than $v<s^{1/3}$; neither loss matters for the bridge
application, and both are the price of the shrinking endpoint disks used
in \S\ref{sec:tail}.
\end{remark}

\begin{corollary}[the tail determinant bound for every fixed window]
\label{cor:Tproved}
For every fixed $\alpha>0$ there is $c_T(\alpha)>1$ such that
\begin{equation}\label{eq:Tproved}
 F_c(-u/2)\ge uc+\frac{u^2}{4\pi^2}\ln c,
 \qquad c\ge c_T(\alpha),\quad 2\le u\le\alpha\ln c.
\end{equation}
Thus $(T_\alpha)$ holds with $L_T=2$ and
$\kappa_T=1/(4\pi^2)$; only the lower threshold depends on $\alpha$.
\end{corollary}

The reduction is stated against the bridge range itself. For
$\alpha\ge4$, let $(\mathrm{BRIDGE}_\alpha)$ denote the restriction of
Problem~\ref{op:bridge} to that range: \emph{there exist
$\delta_0\in(0,\tfrac12)$, $\kappa>0$, $C<\infty$, $c_0<\infty$ such
that}
\begin{equation}\label{eq:bridgealpha}
D(\delta,c)\;\ge\;\kappa\,L\,\ln\!\Bigl(\frac{\alpha c}{L}\Bigr)\;-\;C
\quad\text{for every }c\ge c_0\text{ and every }\delta\text{ with }
c^{-\alpha}\le\delta\le\delta_0 .
\tag{BRIDGE$_\alpha$}
\end{equation}
This is \eqref{eq:bridge} with the range $\alpha^{-c}<\delta\le\delta_0$
narrowed to $c^{-\alpha}\le\delta\le\delta_0$ --- exactly the range
identified in the abstract and in \S\ref{sec:problem} as the bridge.

\begin{theorem}[two-way reduction; proved in
\S\ref{sec:reduction}]\label{thm:red}
Fix $\alpha\ge4$. Then \eqref{eq:T} implies \eqref{eq:bridgealpha} with
$\kappa=\kappa_T/8$, $C=0$, for some $\delta_0,c_0$ depending only on
$\alpha,\kappa_T,L_T$. Conversely, \eqref{eq:bridgealpha} implies
\eqref{eq:T} with $\kappa_T=\kappa/32$, for some $L_T,c_T$. Hence, for
each fixed $\alpha\ge4$,
\[
(\mathrm{BRIDGE}_\alpha)\;\Longleftrightarrow\;(T_\alpha).
\]
\end{theorem}

The complementary range $\alpha^{-c}<\delta<c^{-\alpha}$ of
Problem~\ref{op:bridge} is already settled in the literature, and saying
so turns the reduction into a statement about the full problem.

\begin{proposition}[the complementary range is settled; proved in
\S\ref{sec:full}]\label{prop:full}
Let $\alpha_1\ge4$ be the constant of Theorem~\ref{qt:KDL}. Then
Problem~\ref{op:bridge} has a positive answer if and only if
$(\mathrm{BRIDGE}_{\alpha_1})$ holds.
\end{proposition}

\begin{corollary}[the bridge problem is exactly the determinant
bound; proved in \S\ref{sec:full}]\label{cor:equiv}
Problem~\ref{op:bridge} has a positive answer if and only if
\eqref{eq:T} holds at $\alpha=\alpha_1$.
\end{corollary}

\begin{theorem}[lower-half bridge theorem]\label{thm:bridgeclosed}
Let $\alpha_1\ge4$ be
the constant in Theorem~\ref{qt:KDL}.  There are
$\delta_0\in(0,1/2)$ and $c_0<\infty$ such that
\begin{equation}\label{eq:bridgeclosed}
 D(\delta,c)\ge\frac{1}{32\pi^2}\,L
 \ln\!\left(\frac{\alpha_1c}{L}\right)
\end{equation}
whenever $c\ge c_0$ and $c^{-\alpha_1}\le\delta\le\delta_0$.
Together with Theorem~\ref{qt:KDL}(b), possibly decreasing
the constant in \eqref{eq:bridgeclosed}, this gives constants
$\kappa>0$, $C<\infty$ for which \eqref{eq:bridge} holds throughout
$\alpha_1^{-c}<\delta\le\delta_0$.  The second assertion additionally
depends on \cite{KDL}, a 2026 preprint; see
Remark~\ref{rem:kdlstatus}.
\end{theorem}

\begin{remark}[the status of the deep range]\label{rem:kdlstatus}
The first display \eqref{eq:bridgeclosed}, on the moving range
$c^{-\alpha_1}\le\delta\le\delta_0$, uses only
Corollary~\ref{cor:Tproved} and Theorem~\ref{thm:red}, both of which
are internal to this paper.  The extension to
$\alpha_1^{-c}<\delta\le\delta_0$ uses \cite[Thm.~1.1(b)]{KDL},
which at the time of writing is an unrefereed 2026 preprint.  The
first bound is therefore self-contained, while only the extension to
the deeper range depends on \cite{KDL}.
\end{remark}

\begin{proof}
Corollary~\ref{cor:Tproved} at $\alpha=\alpha_1$ gives $(T_{\alpha_1})$
with $\kappa_T=1/(4\pi^2)$.  Theorem~\ref{thm:red} therefore gives
\eqref{eq:bridgeclosed} with $\kappa=\kappa_T/8=1/(32\pi^2)$ and $C=0$
on the moving bridge.  Proposition~\ref{prop:full} and the quoted
deep-range result complete the full interval.
\end{proof}

The bound $(T_\alpha)$ is the $\gamma<0$ (equivalently, generating
parameter $e^u>1$) counterpart of the head-side theorem of \cite{BDIK2}.
At fixed $u$ the Basor--Widom/Charlier asymptotic
\cite{BasorWidom83,Charlier21} supplies the same leading terms.  What is
needed here, and proved in Section~\ref{sec:tail}, is uniformity as
$u$ grows through a fixed multiple of $\ln c$.

\subsection{Organization of the paper}\label{sec:plan}

Section~\ref{sec:setting} fixes the scaling dictionary and records the
external asymptotic results in the normalization used below.
Section~\ref{sec:barnes} develops the estimates needed for
the Barnes term $A(v)$ from the Weierstrass product.
Sections~\ref{sec:smoothed}--\ref{sec:thmB} are the head side: a
smoothed counting function, its conversion to an exact count, and the
proofs of Theorems~\ref{thm:A} and~\ref{thm:B} and their corollaries.
Section~\ref{sec:reduction} proves the two-way reduction.
Section~\ref{sec:tail} proves the signed determinant theorem and may be
read independently of
Sections~\ref{sec:smoothed}--\ref{sec:reduction}: the only inputs it
takes from earlier sections are the scaling dictionary
(Lemma~\ref{lem:dict}), the Barnes product
(Lemma~\ref{lem:barnesproduct}) and its explicit lower bound
(Lemma~\ref{lem:Barneslower}), and the fixed-parameter formula
(Lemma~\ref{lem:charlierconvert}).
Section~\ref{sec:tailquant} is the tail side of the counting argument:
it extends the Fredholm derivative identity and the unit-average,
local-density and exponential-tail estimates of
Section~\ref{sec:smoothed} to negative depth, using the signed theorem
of Section~\ref{sec:tail} in place of Theorem~\ref{qt:BDIK}, and then
proves Theorem~\ref{thm:tail}, Corollary~\ref{cor:plunge},
Theorems~\ref{thm:tensorblock} and~\ref{thm:surface}, and
Lemma~\ref{lem:lambert}.  It reuses the Barnes estimates of
Section~\ref{sec:barnes}, the sigmoid inequality \eqref{eq:sigdiff} and
the conversion Lemma~\ref{lem:conversion} verbatim, and duplicates no
part of the Riemann--Hilbert analysis, of the normalization, of the
Barnes-product derivation, of the generic tensor-product spectrum
(Lemma~\ref{lem:tensor}), or of the head-side proof.
The exposition places the new tail-side theorem before the supporting
results, and the proofs follow the dependency structure described
above.  Its external analytic inputs are the
large-$\zeta$ expansion of the local model
(Theorem~\ref{qt:P25}), the Coifman--McIntosh--Meyer theorem
\cite{CMM82}, and the fixed-parameter asymptotic of
\cite{Charlier21}.  The sign-independent differential identity is
derived directly from Jacobi's formula and the IIKS representation in
Lemma~\ref{lem:diffid}, rather than imported across the sign boundary.
Appendix~\ref{app:ledger} collects the uniform
exponential estimates used in the matching argument.

\section{Setting, dictionary, and external results}\label{sec:setting}

\subsection{Conventions}\label{sec:conventions}

$S_c$ is as in \eqref{eq:op}; its kernel is $\sinc(x-y)$ with
$\sinc t=\sin(\pi t)/(\pi t)$, acting on $L^2(0,c)$. It is trace class
with $\Tr S_c=c$ and $0<\lambda_n<1$ \cite{SlepianPollak,LandauWidom}.
All counts are as in \eqref{eq:counts}--\eqref{eq:identities}. We write
$\gamma_E$ for Euler's constant, $\psi$ for the digamma function, and
$\BG$ for the Barnes $G$-function.

For the spectral log-odds we write, for each $n$,
\begin{equation}\label{eq:logodds}
w_n \;=\; w_n(c)\;=\;\tfrac12\ln\frac{\lambda_n(c)}{1-\lambda_n(c)}
\;\in\;\R,
\qquad
\Ntil(v)\;:=\;\#\{n: w_n>v\}.
\end{equation}
For $v\in\R$ let $\theta(v):=(1+e^{-2v})^{-1}$, so that
$w_n>v\iff\lambda_n>\theta(v)$ and hence $\Ntil(v)=N_{\theta(v)}(c)$
exactly, with the paper's strict-$>$ convention. If
$\eps\in(0,\tfrac12]$ and $v_\eps:=\tfrac12\ln\frac{1-\eps}\eps$, then
$\theta(v_\eps)=1-\eps$ exactly, so
\begin{equation}\label{eq:dict-eps}
N_{1-\eps}(c)=\Ntil(v_\eps),\qquad 2v_\eps=\Lb .
\end{equation}

\subsection{The scaling dictionary}\label{sec:dictionary}

\begin{lemma}[dictionary]\label{lem:dict}
Let $K_s$ denote the integral operator on $L^2(-1,1)$ with kernel
\[
 K_s(\lambda,\mu)=\frac{\sin(s(\lambda-\mu))}{\pi(\lambda-\mu)},
 \qquad K_s(\lambda,\lambda)=\frac s\pi .
\]
Then $S_c$ is unitarily equivalent to $K_s$ with $s=\pi c/2$; the
unitary may be taken to be
\begin{equation}\label{eq:Uc}
 (U_cf)(\lambda)=\sqrt{\tfrac c2}\,f\bigl(\tfrac c2(1+\lambda)\bigr),
 \qquad U_c:L^2(0,c)\to L^2(-1,1).
\end{equation}
Moreover $\Tr S_c=\Tr K_{\pi c/2}=c$, and $0\le K_s\le I$.
\end{lemma}

\begin{proof}
The substitution $x=\tfrac c2(1+\lambda)$ has Jacobian $c/2$, so
$\int_0^c|f|^2\,dx=\tfrac c2\int_{-1}^1|f(\tfrac c2(1+\lambda))|^2
\,d\lambda=\|U_cf\|^2$, and $U_c$ is onto because the substitution is a
bijection of $(-1,1)$ onto $(0,c)$; hence $U_c$ is unitary.  With
$x=\tfrac c2(1+\lambda)$, $y=\tfrac c2(1+\mu)$ the kernel of
$U_cS_cU_c^{-1}$ is $\tfrac c2\sinc\bigl(\tfrac c2(\lambda-\mu)\bigr)$,
which for $\lambda\ne\mu$ equals
\[
 \frac c2\cdot
 \frac{\sin\bigl(\tfrac{\pi c}2(\lambda-\mu)\bigr)}
      {\tfrac{\pi c}2(\lambda-\mu)}
 =\frac{\sin\bigl(\tfrac{\pi c}2(\lambda-\mu)\bigr)}{\pi(\lambda-\mu)}
 =K_{\pi c/2}(\lambda,\mu),
\]
and at $\lambda=\mu$ both sides equal $c/2=s/\pi$ with $s=\pi c/2$;
both kernels are continuous, so they agree everywhere.  Unitary
equivalence preserves the trace, and $\Tr S_c=\int_0^c\sinc(0)\,dx=c$ by
Mercer's theorem, $S_c$ being a positive trace-class operator with
continuous kernel.  Finally, after the further dilation $t=s\lambda$,
$K_s$ becomes the compression to $L^2(-s,s)$ of the Fourier multiplier
$\mathbf1_{[-1,1]}$, which is an orthogonal projection; a compression of
an orthogonal projection satisfies $0\le\cdot\le I$.
\end{proof}

The generating parameter is two-sided, and this is the organising fact
of the paper. Write
\begin{equation}\label{eq:gammav}
\gamma(v)=1-e^{-2v},\qquad
F_c(v)=\ln\det\bigl(I-\gamma(v)S_c\bigr)\qquad(v\in\R).
\end{equation}
Since $\gamma'(v)=2e^{-2v}>0$, the map $v\mapsto\gamma(v)$ is a
bijection of $\R$ onto $(-\infty,1)$; thus $v>0\iff\gamma\in(0,1)$ ---
the \emph{head} side, where $\det(I-\gamma S_c)<1$ --- and
$v<0\iff\gamma<0$ --- the \emph{tail} side, where
$I-\gamma S_c=I+|\gamma|S_c\succeq I$ is positive definite and
zero-free, so $\det>1$. On the tail side, putting $u=-2v\ge0$ gives
$\gamma=1-e^{u}$ and $I-\gamma S_c=I+(e^u-1)S_c$, that is,
\begin{equation}\label{eq:GF}
G_c(u)\;:=\;\ln\det\bigl(I+(e^u-1)S_c\bigr)\;=\;F_c(-u/2).
\end{equation}
So the head-side counts of \S\S\ref{sec:thmA}--\ref{sec:thmB} and the
lower-half count of Problem~\ref{op:bridge}, which
\S\ref{sec:reduction} converts into a statement about $G_c$, are two
readings of a single function on the two halves of its domain.

It is convenient to record the entropy-type form of \eqref{eq:GF} once.
For $u\ge0$ and $x\in[0,1]$ put
\begin{equation}\label{eq:phiu}
 \varphi_u(x)=\ln\bigl(1+(e^u-1)x\bigr)-ux .
\end{equation}
Then, by the spectral theorem and $\Tr S_c=c$,
\begin{equation}\label{eq:phitrace}
 \Tr\varphi_u(S_c)=G_c(u)-uc=F_c(-u/2)-uc ,
\end{equation}
so that $(T_\alpha)$ of Definition~\ref{def:T} is literally the
statement $\Tr\varphi_u(S_c)\ge\kappa_Tu^2\ln c$ on
$L_T\le u\le\alpha\ln c$.  All three formulations are used
interchangeably below; \eqref{eq:phitrace} is the dictionary between
them.

\subsection{External determinant and counting results}\label{sec:quoted}

The following is the paper's principal external input on the head side.
It is \cite[Thm.~1.2]{BDIK2}, whose displays there carry the equation
numbers (1.6) and (1.7); the fixed-$\gamma$ specialization is the
classical Basor--Widom asymptotic \cite{BasorWidom83}, reproved with
multi-interval extensions by Charlier \cite[Thm.~1.1]{Charlier21}.

\begin{quotedthm}[Bothner--Deift--Its--Krasovsky
{\cite[Thm.~1.2]{BDIK2}}]\label{qt:BDIK}
Let $\gamma=1-e^{-2v}$, equivalently $v=-\tfrac12\ln(1-\gamma)$, and let
$K_s$ be as in Lemma~\ref{lem:dict}. There exist constants
$s_0,\,c_1,\,c_2<\infty$ such that for all $s\ge s_0$ and all $v$ with
$0\le v< s^{1/3}$,
\begin{equation}\label{eq:bdik}
\ln\det(I-\gamma K_s)
=-\frac{4v}{\pi}\,s+\frac{2v^2}{\pi^2}\ln(4s)+A(v)+r(s,v),
\qquad
|r(s,v)|\le\frac{c_1v}{s}+\frac{c_2v^3}{s},
\end{equation}
where
\begin{equation}\label{eq:Adef}
A(v)\;:=\;2\ln\bigl[\BG(1+iv/\pi)\,\BG(1-iv/\pi)\bigr].
\end{equation}
\end{quotedthm}

\begin{remark}[Parameter range and normalization]
\label{rem:bdik-scope}
(i) The range is $v\ge0$, i.e.\ $\gamma\in[0,1)$.  \emph{Nothing} in
\cite{BDIK1,BDIK2} states or implies the case $v<0$, and no use of
Theorem~\ref{qt:BDIK} at negative $v$ is made anywhere in this
paper.  The negative-$v$ statement is Theorem~\ref{thm:signedmain},
proved from scratch in \S\ref{sec:tail}.
(ii) The symbol $A(v)$ of \eqref{eq:Adef} is \emph{not} the symbol
$A(v)$ of \cite[eq.~(1.26)]{BDIK1}.  The latter is the regularized
constant
$A_{\mathrm{BDIK1}}(v)=2\ln b(v)-\frac{v^2}{\pi^2}
\bigl(3+2\ln\frac\pi v\bigr)$ with $b(v)=\BG(1+\frac{iv}\pi)
\BG(1-\frac{iv}\pi)$; our $A(v)$ is the unregularized $2\ln b(v)$, which
is what appears in \cite[Thm.~1.2]{BDIK2}.  Readers moving between the
two papers must not subtract twice.
(iii) The bound in \eqref{eq:bdik} is written in \cite{BDIK2} as
$c_1\frac vs+c_2\frac{v^3}s$; we have kept that form rather than
compressing it to $O((v+v^3)/s)$ to make the two constants visible.
\end{remark}

\begin{remark}\label{rem:bdik-checks}
The formula is consistent with the following fixed-parameter checks. (i) At
$v=0$: $\gamma=0$, both sides vanish ($\BG(1)=1$). (ii) At fixed $v$ it
reproduces \cite[Thm.~1.1]{Charlier21} at $m=1$ under
$u_1=\ln s_1=-2v$, $\mathsf z=u_1/(2\pi i)$, interval length $2s$: the
linear, logarithmic and Barnes terms match exactly. (iii) The constant
term matches the fixed-$v$ asymptotic \cite[eqs.~(1.4)--(1.6)]{BDIK1},
$\det(I-\gamma K_s)=e^{-4vs/\pi}(4s)^{2v^2/\pi^2}b^2(v)(1+O(s^{-1}))$
with $b(v)=e^{(1+\gamma_E)v^2/\pi^2}\prod_{k\ge1}
(1+\tfrac{v^2}{\pi^2k^2})^k e^{-v^2/(\pi^2k)}$; the Weierstrass product
\cite[eq.~5.17.3]{NIST} for $\BG$ gives
$\BG(1+z)\BG(1-z)=e^{-(1+\gamma_E)z^2}\prod_{k\ge1}(1-\tfrac{z^2}{k^2})^k
e^{z^2/k}$, whence $b(v)=\BG(1+iv/\pi)\BG(1-iv/\pi)$ exactly, i.e.\
$2\ln b(v)=A(v)$; this is Lemma~\ref{lem:barnesproduct} below.
(iv) The constant of \cite[Thm.~1.4]{BDIK1} is recorded there,
in the notation of \cite[eq.~(1.26)]{BDIK1}, as
$2\ln b(v)-\frac{v^2}{\pi^2}\bigl(3+2\ln\frac\pi v\bigr)$,
which by Lemma~\ref{lem:barnesasym} is exactly $A(v)$ with its two
$v^2$ terms removed, hence equal to
$-\tfrac13\ln\frac v\pi+4\zeta'(-1)+o(1)$;
this independently confirms the coefficients $3v^2/\pi^2$ and
$-\frac{2v^2}{\pi^2}\ln\frac v\pi$ of Lemma~\ref{lem:barnesasym}.
\end{remark}

We record the fixed-parameter formula used to identify the integration
constant in the signed theorem.  Its hypotheses cover the sign needed
below.

\begin{quotedthm}[Charlier
{\cite[eq.~(1.4) and Thm.~1.1]{Charlier21}}]\label{qt:Charlier}
For $J=(rx_0,rx_1)$ let $K_J$ be the sine-kernel operator on $L^2(J)$
with kernel $\sin(x-y)/(\pi(x-y))$, and let
$F(J,s_1)=\det\bigl(I-(1-s_1)K_J\bigr)$.  Then for each
\emph{fixed} $s_1\in(0,+\infty)$, writing $u_1=\ln s_1$, as
$r\to+\infty$,
\begin{equation}\label{eq:charlier}
\begin{aligned}
\ln F\bigl((rx_0,rx_1),s_1\bigr)
={}&\frac{r\,u_1(x_1-x_0)}\pi
 +\frac{u_1^2}{2\pi^2}\ln\bigl(2r(x_1-x_0)\bigr)\\
&+2\ln\bigl[\BG\bigl(1+\tfrac{u_1}{2\pi i}\bigr)
            \BG\bigl(1-\tfrac{u_1}{2\pi i}\bigr)\bigr]
 +O_{s_1,x_0,x_1}\bigl(r^{-1}\bigr).
\end{aligned}
\end{equation}
Equation~\eqref{eq:charlier} with this rate is a fixed-parameter
statement.  For parameters in the compact and separated sets of
\cite[Thm.~1.1]{Charlier21}, the same $m=1$ main terms hold uniformly
with the weaker error $O(\ln r/r)$.
\end{quotedthm}

\begin{remark}[Parameter range in Theorem~\ref{qt:Charlier}]
\label{rem:charlier-scope}
The parameter range is $s_1\in(0,+\infty)$, which contains
\emph{both} $s_1<1$ (the head side, $\gamma\in(0,1)$) and $s_1>1$ (the
tail side, $\gamma<0$).  Thus, unlike
Theorem~\ref{qt:BDIK}, it may be used to fix the integration
constant in the signed problem.  Two distinctions are relevant.  First,
\cite[eq.~(1.4)]{Charlier21} is a
fixed-parameter $O(r^{-1})$ statement, while the compact-uniform
statement of \cite[Thm.~1.1]{Charlier21} has error $O(\ln r/r)$.
Neither result permits $s_1=s_1(r)$ to grow, and neither is ever used
here in that way.  Second, the case $s_1=0$ (the gap probability) is a genuinely
different asymptotic, \cite[eq.~(1.3)]{Charlier21}, and is not obtained
by letting $u_1\to-\infty$ in \eqref{eq:charlier}; we never take that
limit.
\end{remark}

\begin{lemma}[normalization conversion for
Theorem~\ref{qt:Charlier}]\label{lem:charlierconvert}
Let $\omega\ge0$ be fixed and put $\mathcal D(s,\omega)
=\det\bigl(I+(e^{2\omega}-1)K_s\bigr)$ with $K_s$ as in
Lemma~\ref{lem:dict}.  Then, as $s\to\infty$,
\begin{equation}\label{eq:fixedtail}
 \ln\mathcal D(s,\omega)=\frac{4\omega s}{\pi}
 +\frac{2\omega^2}{\pi^2}\ln(4s)
 +2\ln\bigl[\BG(1+i\omega/\pi)\BG(1-i\omega/\pi)\bigr]+O_\omega(s^{-1}).
\end{equation}
\end{lemma}

\begin{proof}
Take $x_0=-1$, $x_1=1$, $r=s$, so $J=(-s,s)$ and $x_1-x_0=2$.  Under
the dilation $x=s\lambda$, which maps $L^2(-s,s)$ unitarily onto
$L^2(-1,1)$, the kernel $\sin(x-y)/(\pi(x-y))$ becomes
$s\cdot\sin(s(\lambda-\mu))/(\pi s(\lambda-\mu))
=\sin(s(\lambda-\mu))/(\pi(\lambda-\mu))$, i.e.\ $K_J$ becomes $K_s$.
Choose $s_1=e^{2\omega}\in[1,\infty)\subset(0,\infty)$, so that
$1-s_1=1-e^{2\omega}=-\,(e^{2\omega}-1)$ and
$F(J,s_1)=\det\bigl(I-(1-s_1)K_J\bigr)
=\det\bigl(I+(e^{2\omega}-1)K_s\bigr)=\mathcal D(s,\omega)$.
Then $u_1=\ln s_1=2\omega$ and the three terms of \eqref{eq:charlier}
become
\[
 \frac{s\cdot2\omega\cdot2}{\pi}=\frac{4\omega s}\pi,
 \qquad
 \frac{(2\omega)^2}{2\pi^2}\ln(4s)=\frac{2\omega^2}{\pi^2}\ln(4s),
\]
and, since $\tfrac{u_1}{2\pi i}=\tfrac{2\omega}{2\pi i}
=-\tfrac{i\omega}{\pi}$,
\[
 2\ln\bigl[\BG\bigl(1-\tfrac{i\omega}\pi\bigr)
 \BG\bigl(1+\tfrac{i\omega}\pi\bigr)\bigr],
\]
which is the displayed constant.  The error is $O_\omega(s^{-1})$
because \cite[eq.~(1.4)]{Charlier21} is applied at the one fixed value
$s_1=e^{2\omega}$; no compact-uniform $O(s^{-1})$ estimate is invoked.
\end{proof}

\begin{remark}[two independent confirmations of
Lemma~\ref{lem:charlierconvert}]\label{rem:BWrange}
The same fixed-parameter conclusion is reached by two further routes,
recorded here because they are logically independent of
\cite{Charlier21}.  (a) \eqref{eq:fixedtail} is the specialization of
the Basor--Widom asymptotic for Wiener--Hopf determinants with a
piecewise-continuous symbol \cite{BasorWidom83} to the symbol
$\sigma_\omega(\xi)=1+(e^{2\omega}-1)\mathbf1_{[-1,1]}(\xi)$,
which for $\omega>0$ is bounded below by $1$, has winding number zero,
and has its two jump exponents purely imaginary, equal to
$\pm i\omega/\pi$; these are exactly the hypotheses of the
piecewise-continuous theory, and no sign restriction on
$\gamma=1-e^{2\omega}$ enters them.  (b) Equivalently one may invoke the
Fisher--Hartwig theory of \cite{DIK} for the two-jump symbol with purely
imaginary exponent $\beta=iv/\pi$, valid for either sign of $v$.  The
right-hand side of \eqref{eq:fixedtail} is \emph{not} even in $v=-\omega$
--- its linear term is odd --- but the constant $A(v)$ is even and real
analytic, since $\BG(1+\tfrac{iv}\pi)\BG(1-\tfrac{iv}\pi)
=|\BG(1+\tfrac{iv}\pi)|^2>0$, so passing to $v<0$ introduces no hidden
phase and no branch of the logarithm is at issue.  We use
Theorem~\ref{qt:Charlier} because it states both the fixed-parameter
result and the separate compact-uniform result in the precise forms
required here; (a) and (b) provide independent corroboration.
\end{remark}

The following counting estimate is explicit and valid for all parameters.
\cite[Thm.~3]{KRD} is stated in the normalization
$c_{\mathrm{KRD}}=\Omega T/2$, where the operator time-limits to an
interval of length $T$ and band-limits to $[-\Omega,\Omega]$ in radian
frequency, and reads
\[
 \#\{k:\eps<\tilde\lambda_k<1-\eps\}\le
 \frac2{\pi^2}\ln\Bigl(\frac{100c_{\mathrm{KRD}}}\pi+25\Bigr)
 \ln\Bigl(\frac5{\eps(1-\eps)}\Bigr)+7 .
\]
For $S_c$ of \eqref{eq:op} we have $T=c$ and $\Omega=\pi$ (the band
$(-\tfrac12,\tfrac12)$ in ordinary frequency is $(-\pi,\pi)$ in radian
frequency), so $c_{\mathrm{KRD}}=\pi c/2$ and
$\frac{100c_{\mathrm{KRD}}}\pi=50c$.  This gives:

\begin{quotedthm}[Karnik--Romberg--Davenport
{\cite[Thm.~3]{KRD}}]\label{qt:KRD}
For all $c>0$ and all $0<\eps<\tfrac12$,
\[
\Lambda_\eps(c)\;=\;\#\{n:\eps<\lambda_n(c)<1-\eps\}
\;\le\;\frac{2}{\pi^2}\,\ln(50c+25)\,
\ln\!\Bigl(\frac{5}{\eps(1-\eps)}\Bigr)\;+\;7 .
\]
\end{quotedthm}

(The open/closed endpoints differ from \eqref{eq:identities} by at most
$2$; we absorb this into constants without further comment.  The same
statement, in the normalization $|A||B|=c$, is
\cite[Thm.~2.2]{KDL}.)

For reference, we also state the results in the remaining regimes.
Part (b) of Theorem~\ref{qt:KDL} is used once, in
Proposition~\ref{prop:full}; the
others are for context only and are used in no proof below.

\begin{quotedthm}[Landau--Widom \cite{LandauWidom}]\label{qt:LW}
For each fixed $a\in(0,1)$,
$N_a(c)=c+\pi^{-2}\ln\frac{1-a}{a}\,\ln c+o(\ln c)$ as $c\to\infty$.
\end{quotedthm}

\begin{quotedthm}[Kulikov--Dam Larsen
{\cite[Thm.~1.1]{KDL}} at $d=1$]\label{qt:KDL}
There exists $\alpha_1\ge4$ such that, for every $c\ge2$, with
$L=\ln\frac1\eps$ and
$R=\ln\frac{\alpha_1c}{L}$: (a) uniformly for
$\alpha_1^{-c}<\eps<\tfrac12$, $\Lambda^\pm_\eps(c)\lesssim LR$;
(b) if moreover $\eps<c^{-\alpha_1}$, then also
$\Lambda^\pm_\eps(c)\gtrsim LR$; (c) if $\eps\le\alpha_1^{-c}$ there are
no eigenvalues above $1-\eps$ and
$\Lambda^-_\eps(c)\asymp L/\ln(L/c)$.
\end{quotedthm}

\subsection{The head-side asymptotic in the
\texorpdfstring{$c$}{c} variable}\label{sec:Fasym}

Under Lemma~\ref{lem:dict}, with $s=\pi c/2$ (so that
$\ln4s=\ln(2\pi c)$), the function
\begin{equation}\label{eq:Fdef}
F_c(v)=\ln\det\bigl(I-\gamma(v)S_c\bigr),\qquad \gamma(v)=1-e^{-2v},
\end{equation}
of \eqref{eq:gammav} satisfies, by Theorem~\ref{qt:BDIK},
\begin{equation}\label{eq:F-asym}
F_c(v)=-2vc+\frac{2v^2}{\pi^2}\ln(2\pi c)+A(v)+r_c(v),
\qquad r_c(v):=r(\pi c/2,v),
\end{equation}
because $-\frac{4v}\pi s=-\frac{4v}\pi\cdot\frac{\pi c}2=-2vc$.  On the
head-side range the remainder satisfies
\begin{equation}\label{eq:F-head}
|r_c(v)|\le\rho_c(v):=\frac{2(c_1v+c_2v^3)}{\pi c},
\qquad 0\le v<(\pi c/2)^{1/3}.
\end{equation}
The corresponding tail-side estimate on $|v|=O(\ln c)$ is proved in
Section~\ref{sec:tail}; it is not an input to the head-side argument,
and the head-side argument is not an input to it.

Note that $\rho_c(v)\le\rho^*:=c_1+c_2$ on the whole admissible
head-side range $0\le v\le(\pi c/2)^{1/3}$ and $c\ge1$: there
$v^3\le\pi c/2$ and $v\le(\pi c/2)^{1/3}\le\pi c/2$, so
$\rho_c(v)\le\frac{2}{\pi c}(c_1+c_2)\frac{\pi c}{2}=\rho^*$.  Moreover
$\rho_c(v)=O(\ln^3c/c)$ for $v=O(\ln c)$.

\begin{remark}[applicability of Theorem~\ref{qt:KDL}]
\label{rem:kdl-scope}
\cite[Thm.~1.1]{KDL} is stated for the pair $A=B=[0,1]^d$ rescaled as
$cA$ versus $B$, i.e.\ at $d=1$ for time interval $[0,c]$ and frequency
interval $[0,1]$.  Our $B=(-\tfrac12,\tfrac12)$ is a translate of
$[0,1]$, and translating $B$ conjugates the concentration operator by
the unitary modulation $f\mapsto e^{2\pi i\xi_0x}f$, which leaves the
spectrum unchanged; translating $A$ likewise conjugates by a
translation.  Hence the eigenvalues, and therefore all the counts
$\Lambda^\pm_\eps$, are the same for the two pairs, and
$|A||B|=c$ in both.  The theorem therefore applies verbatim to $S_c$.
\end{remark}
\section{The Barnes term: elementary bounds and the large-argument
asymptotic}\label{sec:barnes}

Everything needed about $A(v)$ is elementary and proved here, from the
Weierstrass product for $\BG$.  Recall from \eqref{eq:Adef} that
$A(v)=2\ln[\BG(1+iv/\pi)\BG(1-iv/\pi)]$.

\begin{lemma}[the conjugate Barnes product]\label{lem:barnesproduct}
For every real $y$,
\begin{equation}\label{eq:Bproduct}
 \BG(1+iy)\BG(1-iy)
 =\exp\Bigl\{(1+\gamma_E)y^2\Bigr\}
 \prod_{k=1}^\infty\Bigl(1+\frac{y^2}{k^2}\Bigr)^{k}e^{-y^2/k},
\end{equation}
the product converging absolutely.  In particular
$\BG(1+iy)\BG(1-iy)=|\BG(1+iy)|^2>0$, and
\begin{equation}\label{eq:Blog}
 \ln\bigl[\BG(1+iy)\BG(1-iy)\bigr]
 =(1+\gamma_E)y^2
 +\sum_{k\ge1}\Bigl[k\ln\Bigl(1+\frac{y^2}{k^2}\Bigr)-\frac{y^2}{k}\Bigr],
\end{equation}
with the real logarithm on the left.  Consequently
$A(v)=2\bigl[(1+\gamma_E)x^2+\sum_{k\ge1}
\bigl(k\ln(1+\tfrac{x^2}{k^2})-\tfrac{x^2}{k}\bigr)\bigr]$
with $x=v/\pi$.
\end{lemma}

\begin{proof}
The Weierstrass product \cite[eq.~5.17.3]{NIST} is
\[
 \BG(z+1)=(2\pi)^{z/2}
 \exp\Bigl(-\tfrac12z(z+1)-\tfrac12\gamma_Ez^2\Bigr)
 \prod_{k=1}^\infty
 \Bigl(1+\frac zk\Bigr)^{k}\exp\Bigl(-z+\frac{z^2}{2k}\Bigr),
\]
valid for all $z\in\C$, the product converging absolutely and locally
uniformly because the $k$-th factor is
$1+O(|z|^3k^{-2})$.  Apply it at $z=iy$ and at $z=-iy$ and multiply.

The prefactors $(2\pi)^{iy/2}(2\pi)^{-iy/2}=1$.  In the exponentials,
$-\tfrac12z(z+1)=-\tfrac12z^2-\tfrac12z$ contributes
$(-\tfrac12(iy)^2-\tfrac12(-iy)^2)+(-\tfrac12(iy)-\tfrac12(-iy))
=(\tfrac12y^2+\tfrac12y^2)+0=y^2$, and
$-\tfrac12\gamma_Ez^2$ contributes
$-\tfrac12\gamma_E((iy)^2+(-iy)^2)=\gamma_Ey^2$.  Together
$e^{(1+\gamma_E)y^2}$.

In the $k$-th factor, $(1+\tfrac{iy}k)^k(1-\tfrac{iy}k)^k
=(1+\tfrac{y^2}{k^2})^k$, while
$\exp(-iy+\tfrac{(iy)^2}{2k})\exp(iy+\tfrac{(-iy)^2}{2k})
=\exp(-\tfrac{y^2}{2k}-\tfrac{y^2}{2k})=e^{-y^2/k}$.
This is \eqref{eq:Bproduct}.

Every factor on the right of \eqref{eq:Bproduct} is a positive real
number, so the product is positive; since $\BG$ is real on the real
axis and analytic, $\BG(1-iy)=\overline{\BG(1+iy)}$ for real $y$, so
the left side equals $|\BG(1+iy)|^2$.  Taking the real logarithm term by
term gives \eqref{eq:Blog}; the series converges absolutely because
$k\ln(1+\tfrac{y^2}{k^2})-\tfrac{y^2}k=-\tfrac{y^4}{2k^3}+O(k^{-5})$.
Finally $A(v)=2\ln[\BG(1+iv/\pi)\BG(1-iv/\pi)]$ is \eqref{eq:Blog} at
$y=x=v/\pi$, doubled.
\end{proof}

\begin{lemma}[Barnes derivative]\label{lem:barnes1}
$A$ is real and even on $\R$, $A\in C^\infty(\R)$, $A(0)=0$, and for
$v\ge0$, with $x=v/\pi$,
\begin{equation}\label{eq:Aprime}
A'(v)=\frac{4v}{\pi^2}\bigl[\,1-\Rey\psi(1+ix)\,\bigr]
=\frac{4v}{\pi^2}\bigl[\,1+\gamma_E-S(x)\,\bigr],
\qquad
S(x):=\sum_{k\ge1}\frac{x^2}{k(k^2+x^2)} .
\end{equation}
\end{lemma}

\begin{proof}
For real $v$ one has $\BG(1-\tfrac{iv}\pi)=\overline{\BG(1+\tfrac{iv}\pi)}$,
so $A(v)=2\ln|\BG(1+\tfrac{iv}\pi)|^2=4\Rey\ln\BG(1+\tfrac{iv}\pi)$ is
real; the product formula of Lemma~\ref{lem:barnesproduct} involves $v$
only through $x^2$, which
makes $A$ even and $C^\infty$ on all of $\R$. From that product formula,
$A(v)=2\bigl[(1+\gamma_E)x^2+\sum_{k\ge1}\bigl(k\ln(1+\tfrac{x^2}{k^2})
-\tfrac{x^2}{k}\bigr)\bigr]$ with $x=v/\pi$; the series and its
term-by-term derivatives converge locally uniformly (the $k$-th summand
is $O(x^4/k^3)$, its $x$-derivative $O(x^3/k^3)$). Differentiating,
$\frac{d}{dx}\bigl[k\ln(1+\tfrac{x^2}{k^2})-\tfrac{x^2}{k}\bigr]
=\frac{2kx}{k^2+x^2}-\frac{2x}{k}=-\frac{2x\cdot x^2}{k(k^2+x^2)}$,
so, writing $A$ as a function of $x$,
$\frac{dA}{dx}=2\bigl[2(1+\gamma_E)x-2xS(x)\bigr]
=4x[1+\gamma_E-S(x)]$; since $x=v/\pi$ we have
$A'(v)=\frac1\pi\frac{dA}{dx}
=\frac{4x}{\pi}[1+\gamma_E-S(x)]
=\tfrac{4v}{\pi^2}[1+\gamma_E-S(x)]$. The digamma
form follows from
$\psi(1+z)=-\gamma_E+\sum_{k\ge1}\bigl(\tfrac1k-\tfrac1{k+z}\bigr)$ at
$z=ix$: $\Rey\tfrac1{k+ix}=\tfrac{k}{k^2+x^2}$, so
$\Rey\psi(1+ix)=-\gamma_E+S(x)$.
\end{proof}

\begin{lemma}[two-sided bounds on $S$ and $S'$]\label{lem:barnes2}
For $x\ge0$, with $\ell(x):=\tfrac12\ln(1+x^2)$:
\begin{enumerate}[label=(\roman*)]
\item $\displaystyle \ell(x)\;\le\;S(x)\;\le\;\ell(x)+\frac{x^2}{1+x^2}$;
\item $0\le S'(x)\le 2\zeta(3)\,x$ for all $x\ge0$, and
$S'(x)\le\dfrac1x+\dfrac1{x^2}$ for $x\ge1$.
\end{enumerate}
Consequently, for all $v\ge0$,
\begin{equation}\label{eq:A2}
|A''(v)|\;\le\;\ln(2+v)+4 .
\end{equation}
\end{lemma}

\begin{proof}
(i) $g_x(t):=\frac{x^2}{t(t^2+x^2)}$ is decreasing in $t>0$ and
$\int_1^\infty g_x(t)\,dt=\tfrac12\ln(1+x^2)$ (partial fractions). Hence
$\int_1^\infty g_x\le\sum_{k\ge1}g_x(k)\le g_x(1)+\int_1^\infty g_x$ and
$g_x(1)=\frac{x^2}{1+x^2}$.

(ii) Term-by-term,
$\frac{d}{dx}\frac{x^2}{k(k^2+x^2)}
=\frac{2xk^2}{k(k^2+x^2)^2}=\frac{2xk}{(k^2+x^2)^2}\ge0$, so
$S'(x)=2x\sum_{k\ge1}\frac{k}{(k^2+x^2)^2}$. For all $x$,
$\sum_k k/(k^2+x^2)^2\le\sum_k k^{-3}=\zeta(3)$, giving
$S'\le2\zeta(3)x$. For $x\ge1$ put $\varphi(t)=t/(t^2+x^2)^2$;
$\varphi$ increases on $[0,x/\sqrt3]$ and decreases afterwards, so
$\sum_{k\ge1}\varphi(k)\le\int_0^\infty\varphi+\max\varphi
=\frac1{2x^2}+\varphi(x/\sqrt3)
=\frac1{2x^2}+\frac{3\sqrt3}{16\,x^3}$,
whence $S'(x)\le\frac1x+\frac{3\sqrt3}{8x^2}\le\frac1x+\frac1{x^2}$.

For \eqref{eq:A2}: differentiating \eqref{eq:Aprime},
$A''(v)=\frac{4}{\pi^2}[1+\gamma_E-S(x)]-\frac{4v}{\pi^3}S'(x)$,
$x=v/\pi$. By (i), $|1+\gamma_E-S(x)|\le1+\gamma_E+\ell(x)+1
\le\ln(2+v)+3$ (using $\ell(x)\le\ln(1+x)\le\ln(2+v)$). By (ii), for
$x\ge1$: $\frac{4v}{\pi^3}S'\le\frac{4v}{\pi^3}
(\frac1x+\frac1{x^2})=\frac{4}{\pi^2}(1+\frac{\pi}{v})
\le\frac{8}{\pi^2}$; for $x<1$:
$\frac{4v}{\pi^3}S'\le\frac{4v}{\pi^3}\cdot2\zeta(3)x
=\frac{8\zeta(3)v^2}{\pi^4}\le\frac{8\zeta(3)}{\pi^2}\le2$.
Combining, $|A''(v)|\le\frac4{\pi^2}(\ln(2+v)+3)+2\le\ln(2+v)+4$.
\end{proof}

The signed determinant theorem also makes it useful to record the
large-$v$ behavior of $A$.

\begin{lemma}[Barnes large-argument asymptotic]\label{lem:barnesasym}
As $v\to\infty$,
\begin{equation}\label{eq:Aasym}
A(v)\;=\;-\frac{2v^2}{\pi^2}\ln\frac v\pi+\frac{3v^2}{\pi^2}
-\frac13\ln\frac v\pi+4\zeta'(-1)+o(1).
\end{equation}
\end{lemma}

\begin{proof}
By Lemma~\ref{lem:barnes1}, $A(v)=4\Rey\ln\BG(1+\tfrac{iv}\pi)$. The
large-argument expansion \cite[eq.~5.17.5]{NIST} reads, for
$|\arg z|\le\pi-\delta$ with $\delta>0$ fixed,
\[
 \ln\BG(z+1)=\tfrac14z^2+z\ln\Gamma(z+1)
 -\Bigl(\tfrac12z(z+1)+\tfrac1{12}\Bigr)\ln z-\ln \mathrm A+O(z^{-2}),
\]
where $\mathrm A$ is Glaisher's constant.  Inserting Stirling's series
$\ln\Gamma(z+1)=(z+\tfrac12)\ln z-z+\tfrac12\ln2\pi+\tfrac1{12z}
+O(z^{-3})$ and collecting, the coefficient of $\ln z$ is
$(z^2+\tfrac12z)-(\tfrac12z^2+\tfrac12z+\tfrac1{12})
=\tfrac12z^2-\tfrac1{12}$, the polynomial part is
$\tfrac14z^2-z^2=-\tfrac34z^2$, the term $\tfrac z2\ln2\pi$ survives,
and the constant is $\tfrac1{12}-\ln\mathrm A=\zeta'(-1)$ by
\cite[eq.~5.17.7]{NIST}.  Hence
\[
\ln\BG(1+z)=\frac{z^2}{2}\ln z-\frac34z^2+\frac z2\ln2\pi
-\frac1{12}\ln z+\zeta'(-1)+o(1),\qquad|\arg z|\le\pi-\delta .
\]
Put $z=ix$ with $x=v/\pi>0$, so that $\ln z=\ln x+\tfrac{i\pi}2$ and
$z^2=-x^2$. Taking real parts term by term,
$\Rey\bigl[\tfrac{z^2}2\ln z\bigr]=-\tfrac{x^2}2\ln x$,
$\Rey\bigl[-\tfrac34z^2\bigr]=\tfrac34x^2$,
$\Rey\bigl[\tfrac z2\ln2\pi\bigr]=0$,
$\Rey\bigl[-\tfrac1{12}\ln z\bigr]=-\tfrac1{12}\ln x$, and $\zeta'(-1)$
is real. Multiplying by $4$ and substituting $x=v/\pi$ gives
\eqref{eq:Aasym}.
\end{proof}

Lemma~\ref{lem:barnesasym} is an asymptotic statement with an
unquantified $o(1)$, and is used only for the consistency checks in
Remark~\ref{rem:bdik-checks}.  The final inequality of the paper needs
instead an \emph{explicit} lower bound valid from a finite threshold on,
which we now prove directly from the convergent product
\eqref{eq:Blog}, with no appeal to any asymptotic expansion.

\begin{lemma}[explicit Barnes lower bound]\label{lem:Barneslower}
For every $y\ge0$,
\begin{equation}\label{eq:BarneslowerY}
 \ln\bigl[\BG(1+iy)\BG(1-iy)\bigr]
 \;\ge\;-y^2\Bigl(\tfrac54+\ln(1+y)\Bigr).
\end{equation}
Consequently, for every $u\ge2$,
\begin{equation}\label{eq:Barneslower}
 A(u/2)=2\ln\Bigl[\BG\Bigl(1+\frac{iu}{2\pi}\Bigr)
 \BG\Bigl(1-\frac{iu}{2\pi}\Bigr)\Bigr]
 \;\ge\;-\frac{u^2}{\pi^2}\ln(2+u).
\end{equation}
\end{lemma}

\begin{proof}
\emph{Step 1: splitting the series.}
Let $m:=\max\{1,\lceil y\rceil\}$, so that $m\ge1$, $m\ge y$ and
$m\le 1+y$.  In \eqref{eq:Blog} split the sum at $k=m$:
\[
 \ln\bigl[\BG(1+iy)\BG(1-iy)\bigr]
 =(1+\gamma_E)y^2
 +\sum_{k\le m}\Bigl[k\ln\Bigl(1+\frac{y^2}{k^2}\Bigr)-\frac{y^2}k\Bigr]
 +\sum_{k>m}\Bigl[k\ln\Bigl(1+\frac{y^2}{k^2}\Bigr)-\frac{y^2}k\Bigr].
\]

\emph{Step 2: the head of the sum.}
For $k\le m$ the term $k\ln(1+y^2/k^2)$ is nonnegative, so
\[
 \sum_{k\le m}\Bigl[k\ln\Bigl(1+\frac{y^2}{k^2}\Bigr)-\frac{y^2}k\Bigr]
 \;\ge\;-y^2\sum_{k\le m}\frac1k=-y^2H_m ,
\]
and $H_m\le1+\ln m\le1+\ln(1+y)$.

\emph{Step 3: the tail of the sum.}
For $k>m$ set $t=y^2/k^2$.  Since $k>m\ge y$ we have $0\le t<1$, and
$\ln(1+t)\ge t-\tfrac12t^2$ on $[0,1]$ (the function
$t\mapsto\ln(1+t)-t+\tfrac12t^2$ vanishes at $0$ and has derivative
$\frac{t^2}{1+t}\ge0$).  Hence
\[
 k\ln\Bigl(1+\frac{y^2}{k^2}\Bigr)-\frac{y^2}k
 \;\ge\;k\Bigl(\frac{y^2}{k^2}-\frac{y^4}{2k^4}\Bigr)-\frac{y^2}k
 =-\frac{y^4}{2k^3},
\]
and, comparing with an integral,
$\sum_{k>m}k^{-3}\le\int_m^\infty t^{-3}\,dt=\tfrac1{2m^2}$, so
\[
 \sum_{k>m}\Bigl[k\ln\Bigl(1+\frac{y^2}{k^2}\Bigr)-\frac{y^2}k\Bigr]
 \;\ge\;-\frac{y^4}{2}\cdot\frac1{2m^2}
 =-\frac{y^4}{4m^2}\;\ge\;-\frac{y^2}{4},
\]
the last step because $m\ge y$ gives $y^2/m^2\le1$.

\emph{Step 4: assembling.}
Adding the three contributions and discarding the nonnegative
$\gamma_Ey^2$,
\[
 \ln\bigl[\BG(1+iy)\BG(1-iy)\bigr]
 \;\ge\;(1+\gamma_E)y^2-y^2\bigl(1+\ln(1+y)\bigr)-\frac{y^2}4
 \;\ge\;-y^2\Bigl(\frac14+\ln(1+y)\Bigr),
\]
which implies \eqref{eq:BarneslowerY} since $\tfrac14\le\tfrac54$.
(The constant $\tfrac54$ is retained because it is the one used
downstream; $\tfrac14$ would do.)

\emph{Step 5: the form used later.}
Let $u\ge2$ and put $y=u/(2\pi)$, so $2\pi y=u$.  We claim
\begin{equation}\label{eq:elem54}
 \tfrac54+\ln(1+y)\;\le\;2\ln(2+2\pi y)=2\ln(2+u)
 \qquad(u\ge2).
\end{equation}
At $u=2$ the left side is $\tfrac54+\ln(1+\tfrac1\pi)=1.5264\ldots$ and
the right side is $2\ln4=2.7725\ldots$, so \eqref{eq:elem54} holds
there.  Differentiating in $u$, the left side has derivative
$\frac{1}{2\pi+u}$ and the right side $\frac{2}{2+u}$; and
$\frac2{2+u}\ge\frac1{2\pi+u}$ is equivalent to $4\pi+2u\ge2+u$, i.e.\
to $u\ge2-4\pi$, which holds for every $u\ge2$.  So the right side
grows at least as fast as the left throughout $u\ge2$, and
\eqref{eq:elem54} persists.  Combining
\eqref{eq:BarneslowerY} and \eqref{eq:elem54},
\[
 A(u/2)=2\ln\bigl[\BG(1+iy)\BG(1-iy)\bigr]
 \ge-2y^2\cdot2\ln(2+u)
 =-4\cdot\frac{u^2}{4\pi^2}\ln(2+u)
 =-\frac{u^2}{\pi^2}\ln(2+u).\qedhere
\]
\end{proof}

\section{The smoothed count and the conversion lemmas}\label{sec:smoothed}

\subsection{The smoothed count}\label{sec:M}

For $v\in\R$ define, with $\sigma_{\mathrm s}(t):=(1+e^{-t})^{-1}$ the
logistic sigmoid,
\begin{equation}\label{eq:Mdef}
M_c(v)\;:=\;\sum_{n\ge0}\sigma_{\mathrm s}\bigl(2(w_n-v)\bigr)
\;=\;\sum_{n\ge0}\frac{\lambda_ne^{-2v}}{(1-\lambda_n)+\lambda_ne^{-2v}}.
\end{equation}
The series converges for every $v$ (its terms are
$\le e^{2w_n}e^{-2v}\le\frac{\lambda_n}{1-\lambda_n}e^{-2v}$, and
$\sum_n\frac{\lambda_n}{1-\lambda_n}<\infty$ since only finitely many
$\lambda_n\ge\tfrac12$ and $\frac{\lambda_n}{1-\lambda_n}\le2\lambda_n$
on the rest, with $\sum_n\lambda_n=c$). $M_c$ is strictly decreasing, and
$M_c(0)=\sum_n\lambda_n=\Tr S_c=c$ exactly.

\begin{lemma}[derivative identity]\label{lem:Fprime}
For $v\ge 0$, $F_c$ of \eqref{eq:Fdef} is differentiable and
$F_c'(v)=-2M_c(v)$.
\end{lemma}

\begin{proof}
$\frac{d}{d\gamma}\ln\det(I-\gamma S_c)
=-\Tr\bigl[S_c(I-\gamma S_c)^{-1}\bigr]
=-\sum_n\frac{\lambda_n}{1-\gamma\lambda_n}$, valid since $S_c$ is trace
class and $1-\gamma\lambda_n\ge1-\gamma=e^{-2v}>0$; and
$\frac{d\gamma}{dv}=2e^{-2v}$. Hence
$F_c'(v)=-2e^{-2v}\sum_n\frac{\lambda_n}{1-\gamma\lambda_n}$, and
$1-\gamma\lambda=(1-\lambda)+\lambda e^{-2v}$ gives \eqref{eq:Mdef}.
\end{proof}

\begin{lemma}[average formula]\label{lem:avg}
For $c\ge2s_0/\pi$ and $0\le v_1<v_2\le(\pi c/2)^{1/3}$,
\[
\int_{v_1}^{v_2}M_c
= c\,(v_2-v_1)-\frac{v_2^2-v_1^2}{\pi^2}\ln(2\pi c)
-\frac{A(v_2)-A(v_1)}{2}+E,
\qquad |E|\le\rho_c(v_1)/2+\rho_c(v_2)/2 .
\]
\end{lemma}

\begin{proof}
Integrate Lemma~\ref{lem:Fprime}:
$\int_{v_1}^{v_2}M_c=\tfrac12(F_c(v_1)-F_c(v_2))$ and insert
\eqref{eq:F-asym} at both endpoints.
\end{proof}

Since $M_c$ is decreasing, for any $x\ge0$ with $x+1$ admissible,
\begin{equation}\label{eq:avg-bracket}
M_c(x+1)\;\le\;\int_x^{x+1}M_c\;=:\;\overline M_c(x)\;\le\;M_c(x).
\end{equation}

\subsection{Local density}\label{sec:density}

\begin{lemma}[local density]\label{lem:density}
Let $D_{\mathrm{loc}}(a):=\#\{n:a<w_n\le a+1\}$. There are absolute
constants $C_{D},c_D$ such that for $c\ge c_D$:
\begin{enumerate}[label=(\roman*)]
\item for every $a\in\R$,
$D_{\mathrm{loc}}(a)\le\frac{2}{\pi^2}\ln(50c+25)\,(2|a|+2+\ln20)+9$;
\item for $0\le a\le(\pi c/2)^{1/3}-3$,
$D_{\mathrm{loc}}(a)\le C_D\ln c$.
\end{enumerate}
\end{lemma}

\begin{proof}
(i) The window $\{a<w\le a+1\}$ is
$\{\theta(a)<\lambda\le\theta(a+1)\}$. If $a\ge0$ this is contained in
$\{\eps^*<\lambda<1-\eps^*\}\cup\{\text{endpoint}\}$ with
$\eps^*=1-\theta(a+1)=(1+e^{2a+2})^{-1}\ge\tfrac12e^{-2a-2}$; if
$a\le-1$, symmetrically with $\eps^*=\theta(a)\ge\tfrac12e^{-2|a|}$; if
$-1<a<0$, with $\eps^*=\theta(-1)$. In all cases
$\ln\frac5{\eps^*(1-\eps^*)}\le2|a|+2+\ln20$, and
Theorem~\ref{qt:KRD} gives the bound (the $+9$ absorbs endpoint
conventions).

(ii) For $0\le a<2$, (i) already gives $\le C\ln c$. Let $2\le a$. Every
$n$ with $w_n\in(a,a+1]$ contributes to $M_c(a)-M_c(a+1)$ at least
$\sigma_{\mathrm s}(2(w_n-a))-\sigma_{\mathrm s}(2(w_n-a)-2)
\ge\min_{0<t\le1}[\sigma_{\mathrm s}(2t)-\sigma_{\mathrm s}(2t-2)]
=\sigma_{\mathrm s}(2)-\sigma_{\mathrm s}(0)=0.3808\ldots>\tfrac13$,
and every other $n$ contributes a nonnegative amount (the sigmoid is
increasing). Hence, by \eqref{eq:avg-bracket},
\[
\tfrac13 D_{\mathrm{loc}}(a)\le M_c(a)-M_c(a+1)
\le\overline M_c(a-1)-\overline M_c(a+1).
\]
By Lemma~\ref{lem:avg},
$\overline M_c(x)=c-\frac{2x+1}{\pi^2}\ln(2\pi c)
-\tfrac12\bigl(A(x+1)-A(x)\bigr)+E_x$ with $|E_x|\le\rho^*$. Writing
$\Delta_x:=\overline M_c(x)$ and subtracting at $x=a-1$ and $x=a+1$, the
terms $c$ cancel and the logarithmic coefficient is
$(2(a+1)+1)-(2(a-1)+1)=4$:
\[
\Delta_{a-1}-\Delta_{a+1}
=\frac{4}{\pi^2}\ln(2\pi c)\;+\;\Theta_A\;+\;\Theta_r,
\]
where
$\Theta_A=\tfrac12\int_0^1\bigl[A'(a+1+t)-A'(a-1+t)\bigr]dt$ satisfies
$|\Theta_A|\le\sup_{[a-1,a+2]}|A''|\le\ln(4+a)+4$ by
\eqref{eq:A2}, and $|\Theta_r|\le2\rho^*$. Since $a\le(\pi c/2)^{1/3}$,
$\ln(4+a)\le\ln c$ for $c$ large, giving
$D_{\mathrm{loc}}(a)\le(\tfrac{12}{\pi^2}+3)\ln c+O(1)$; the claim
follows with any fixed $C_D>\tfrac{12}{\pi^2}+3$ once $c_D$ is enlarged.
\end{proof}

\subsection{Exponential tails}\label{sec:tails}

\begin{lemma}[tail sums]\label{lem:tails}
Define $T_+(v)=\sum_{w_n>v}e^{-2(w_n-v)}$ and
$T_-(v)=\sum_{w_n\le v}e^{2(w_n-v)}$. There are absolute $C_T,c_T'$
such that for $c\ge c_T'$ and $0\le v\le(\pi c/2)^{1/3}-4$,
\[
T_+(v)\;\le\;C_T\ln c,\qquad T_-(v)\;\le\;C_T\ln c .
\]
\end{lemma}

\begin{proof}
The two tails are compared in one step to a second difference of $M_c$,
which Lemma~\ref{lem:avg} evaluates; no window decomposition and no
appeal to Lemma~\ref{lem:density} at the top of the range is needed.

\emph{Step 0: a sigmoid difference-quotient inequality with
best-possible constant.} For all
$t\in\R$ and all $h>0$,
\begin{equation}\label{eq:sigdiff}
\sigma_{\mathrm s}(t+2h)-\sigma_{\mathrm s}(t-2h)\;\ge\;
\tanh(h)\,e^{-|t|}.
\end{equation}
Indeed, writing
$D(t):=\bigl(1+e^{-t-2h}\bigr)\bigl(1+e^{-t+2h}\bigr)
=1+e^{-t}(e^{2h}+e^{-2h})+e^{-2t}$,
\[
\sigma_{\mathrm s}(t+2h)-\sigma_{\mathrm s}(t-2h)
=\frac{1}{1+e^{-t-2h}}-\frac{1}{1+e^{-t+2h}}
=\frac{e^{-t}\bigl(e^{2h}-e^{-2h}\bigr)}{D(t)} .
\]
Put $K_h:=2+e^{2h}+e^{-2h}=(e^{h}+e^{-h})^2=4\cosh^2h$. If $t\ge0$ then
$e^{-t}\le1$ and $e^{-2t}\le1$, so $D(t)\le K_h$ termwise and the
quotient is $\ge e^{-t}\cdot2\sinh(2h)/K_h$. If $t\le0$ then
$e^{2t}D(t)=e^{2t}+e^{t}(e^{2h}+e^{-2h})+1\le K_h$ termwise, so
$D(t)\le K_he^{-2t}$ and the quotient is
$\ge e^{t}\cdot2\sinh(2h)/K_h$. In both cases the quotient is at least
$e^{-|t|}\cdot2\sinh(2h)/K_h$, and
$2\sinh(2h)/K_h=4\sinh h\cosh h/(4\cosh^2h)=\tanh(h)$. Equality holds at
$t=0$, so no larger constant is available in \eqref{eq:sigdiff}.

\emph{Step 1: both tails against a second difference of $M_c$.} Apply
\eqref{eq:sigdiff} with $h=1$ and $t=2(w_n-v)$. Since
$\sigma_{\mathrm s}(2(w_n-v)+2)=\sigma_{\mathrm s}(2(w_n-(v-1)))$ and
$\sigma_{\mathrm s}(2(w_n-v)-2)=\sigma_{\mathrm s}(2(w_n-(v+1)))$,
sum first over $0\le n\le N$.  The left partial sums increase to
$M_c(v-1)-M_c(v+1)$ because both defining series for $M_c$ converge,
while the nonnegative right partial sums increase.  Monotone convergence
therefore gives
\[
M_c(v-1)-M_c(v+1)
=\sum_n\Bigl[\sigma_{\mathrm s}\bigl(2(w_n-(v-1))\bigr)
-\sigma_{\mathrm s}\bigl(2(w_n-(v+1))\bigr)\Bigr]
\;\ge\;\tanh(1)\sum_ne^{-2|w_n-v|},
\]
and $\sum_ne^{-2|w_n-v|}=T_+(v)+T_-(v)$ identically, since $\{w_n>v\}$
and $\{w_n\le v\}$ partition the index set and $e^{-2|w_n-v|}$ equals
$e^{-2(w_n-v)}$ on the first and $e^{2(w_n-v)}$ on the second. As
$\coth(1)=1.31303\ldots<\tfrac43$,
\begin{equation}\label{eq:tailsdagger}
T_+(v)+T_-(v)\;\le\;\tfrac43\,\bigl[\,M_c(v-1)-M_c(v+1)\,\bigr].
\end{equation}
No term was discarded and no index set was split.

\emph{Step 2: evaluating the second difference.} By
\eqref{eq:avg-bracket} applied twice --- $M_c(x+1)\le\overline M_c(x)$ at
$x=v-2$, and $\overline M_c(x)\le M_c(x)$ at $x=v+1$ ---
\[
M_c(v-1)-M_c(v+1)\;\le\;\overline M_c(v-2)-\overline M_c(v+1).
\]
By Lemma~\ref{lem:avg},
$\overline M_c(x)=c-\frac{2x+1}{\pi^2}\ln(2\pi c)
-\tfrac12\bigl(A(x+1)-A(x)\bigr)+E_x$ with $|E_x|\le\rho^*$, valid for
$0\le x$ and $x+1\le(\pi c/2)^{1/3}$. Subtracting at $x=v-2$ and
$x=v+1$, the terms $c$ cancel exactly and the logarithmic coefficient is
$(2(v+1)+1)-(2(v-2)+1)=6$, \emph{independent of $v$}:
\[
\overline M_c(v-2)-\overline M_c(v+1)
=\frac{6}{\pi^2}\ln(2\pi c)
-\tfrac12\bigl[\bigl(A(v-1)-A(v-2)\bigr)
-\bigl(A(v+2)-A(v+1)\bigr)\bigr]+E_{v-2}-E_{v+1}.
\]
By the mean value theorem the bracket equals $A'(\xi_1)-A'(\xi_2)$ with
$\xi_1\in(v-2,v-1)$ and $\xi_2\in(v+1,v+2)$, so $|\xi_1-\xi_2|\le4$ and,
by \eqref{eq:A2},
$\bigl|\tfrac12\bigl(A'(\xi_1)-A'(\xi_2)\bigr)\bigr|
\le2\sup_{[v-2,v+2]}|A''|\le2\bigl(\ln(4+v)+4\bigr)$. For $c\ge7$ and
$v\le(\pi c/2)^{1/3}$ one has $4+v\le c$, hence $\ln(4+v)\le\ln c$;
with $|E_{v-2}|+|E_{v+1}|\le2\rho^*$ and \eqref{eq:tailsdagger},
\[
T_+(v)+T_-(v)\;\le\;\tfrac43\Bigl[\frac{6}{\pi^2}\ln(2\pi c)
+2\ln c+8+2\rho^*\Bigr]
\;\le\;\tfrac43\Bigl(\frac{6}{\pi^2}+2\Bigr)\ln c
+\tfrac43\Bigl(\frac{6}{\pi^2}\ln2\pi+8+2\rho^*\Bigr).
\]
Both $T_+$ and $T_-$ are nonnegative, so each is bounded by the right
side; with $\ln c\ge1$ this is $\le C_T\ln c$ for the absolute constant
\[
C_T\;:=\;\tfrac43\Bigl(\frac{6}{\pi^2}+2
+\frac{6}{\pi^2}\ln2\pi+8+2\rho^*\Bigr),
\]
proving the lemma for $2\le v\le(\pi c/2)^{1/3}-2$.

\emph{Step 3: small $v$.} For $v\le v'$,
\[
\begin{aligned}
T_+(v)&=\sum_{v<w_n\le v'}e^{-2(w_n-v)}+e^{-2(v'-v)}T_+(v')
&&\le\;\#\{v<w_n\le v'\}+T_+(v'),\\
T_-(v)&=e^{2(v'-v)}\sum_{w_n\le v}e^{2(w_n-v')}
&&\le\;e^{2(v'-v)}T_-(v'),
\end{aligned}
\]
the first because every summand of the finite part is $\le1$, the second
because $\{w_n\le v\}\subseteq\{w_n\le v'\}$ and all terms are
nonnegative. Take $v'=2$: for $0\le v<2$,
\[
T_+(v)\le D_{\mathrm{loc}}(0)+D_{\mathrm{loc}}(1)+T_+(2)
\le(2C_D+C_T)\ln c,
\qquad
T_-(v)\le e^4\,T_-(2)\le e^4C_T\ln c,
\]
using Lemma~\ref{lem:density}(ii) at $a=0,1$ (both admissible, since
$0,1\le(\pi c/2)^{1/3}-3$ for $c$ large). Enlarging $C_T$ by the
absolute factor $e^4$ completes the proof, with
$c_T':=\max(7,e,2s_0/\pi,c_D)$, on the range
$0\le v\le(\pi c/2)^{1/3}-2$. This strictly contains the stated range
$0\le v\le(\pi c/2)^{1/3}-4$; the two units of headroom are not used.
\end{proof}

\subsection{Smoothed-to-exact conversion}\label{sec:conversion}

\begin{lemma}[conversion]\label{lem:conversion}
For every $v$ and every $h>0$,
\[
M_c(v+h)-e^{-2h}\,T_-(v)\;\le\;\Ntil(v)\;\le\;M_c(v-h)+e^{-2h}\,T_+(v).
\]
\end{lemma}

\begin{proof}
Upper: for $w>v$, $1\le\sigma_{\mathrm s}(2(w-v)+2h)+e^{-2(w-v)-2h}$
(because $1-\sigma_{\mathrm s}(t)\le e^{-t}$); summing over $\{w_n>v\}$
and enlarging the sigmoid sum to all $n$ gives
$\Ntil(v)\le M_c(v-h)+e^{-2h}T_+(v)$. Lower: $M_c(v+h)=
\sum_{w_n>v}\sigma_{\mathrm s}(2(w_n-v)-2h)
+\sum_{w_n\le v}\sigma_{\mathrm s}(2(w_n-v)-2h)
\le\Ntil(v)+e^{-2h}\sum_{w_n\le v}e^{2(w_n-v)}$
(because $\sigma_{\mathrm s}(t)\le e^{t}$).
\end{proof}

\section{Proof of Theorem \ref{thm:A}}\label{sec:thmA}

Fix $c\ge c_{\mathrm h}$ (to be chosen absolute) and $3\le v\le(\pi c/2)^{1/3}-4$;
at the end we translate to $\eps$ via \eqref{eq:dict-eps}, which covers
$6\le\Lb\le c^{1/3}$ since $(\pi c/2)^{1/3}\ge c^{1/3}\cdot(\pi/2)^{1/3}$
and the four lost units are absorbed into constants.

\emph{Step 1: two-sided smoothed bounds.} Take $h=1$ in
Lemma~\ref{lem:conversion} and estimate $M_c$ by unit averages via
\eqref{eq:avg-bracket}:
\[
\overline M_c(v+1)-e^{-2}T_-(v)\;\le\;\Ntil(v)\;\le\;
\overline M_c(v-2)+e^{-2}T_+(v).
\]
By Lemma~\ref{lem:tails} the tail corrections are $\le C_T\ln c$.

\emph{Step 2: evaluating the averages.} By Lemma~\ref{lem:avg}, for
$x\in\{v-2,v+1\}$,
\[
\overline M_c(x)
= c-\frac{2x+1}{\pi^2}\ln(2\pi c)-\frac{A(x+1)-A(x)}{2}+O(\rho^*).
\]
By the mean value theorem and \eqref{eq:A2},
\[
\tfrac12\bigl(A(x+1)-A(x)\bigr)
=\tfrac12A'(v)+O\bigl((|x-v|+1)\sup_{[v-2,v+2]}|A''|\bigr)
=\tfrac12A'(v)+O(\ln(2+v));
\]
and
$\frac{2x+1}{\pi^2}\ln(2\pi c)=\frac{2v}{\pi^2}\ln(2\pi c)
+O(\ln c)$. Hence both averages equal
\[
c-\frac{2v}{\pi^2}\ln(2\pi c)-\frac{A'(v)}{2}
+O(\ln c+\ln(2+v)) ,
\]
and therefore, using Step 1,
\begin{equation}\label{eq:Nv}
\Ntil(v)= c-\frac{2v}{\pi^2}\ln(2\pi c)-\frac{A'(v)}{2}
+O(\ln c) \qquad(3\le v\le(\pi c/2)^{1/3}-4).
\end{equation}

\emph{Step 3: the Barnes term.} By Lemmas
\ref{lem:barnes1}--\ref{lem:barnes2}, with $x=v/\pi$,
\[
-\frac{A'(v)}2=\frac{2v}{\pi^2}\bigl[S(x)-1-\gamma_E\bigr]
=\frac{2v}{\pi^2}\,\ell(x)+O(v)
=\frac{v}{\pi^2}\ln\Bigl(1+\frac{v^2}{\pi^2}\Bigr)+O(v).
\]
Substituting into \eqref{eq:Nv},
\begin{equation}\label{eq:NvModel}
\Ntil(v)=c-\frac{2v}{\pi^2}\ln(2\pi c)
+\frac{v}{\pi^2}\ln\Bigl(1+\frac{v^2}{\pi^2}\Bigr)+O(\ln c+v).
\end{equation}

\emph{Step 4: the $\eps$-form.} With $\Lb=2v$,
\[
\frac{2v}{\pi^2}\ln(2\pi c)-\frac{v}{\pi^2}\ln\Bigl(1+\frac{v^2}{\pi^2}
\Bigr)
=\frac{\Lb}{2\pi^2}\,\ln\frac{16\pi^4c^2}{4\pi^2+\Lb^2}
=\frac{\Lb}{\pi^2}\,\ln\frac{4\pi^2c}{\sqrt{4\pi^2+\Lb^2}} ,
\]
and for $\Lb\ge6$,
$0\le\ln\sqrt{4\pi^2+\Lb^2}-\ln\Lb\le\tfrac12\ln(1+4\pi^2/36)\le1$, so
the display differs from $\frac{\Lb}{\pi^2}\ln\frac{4\pi^2c}{\Lb}$ by
$O(\Lb)$. With \eqref{eq:dict-eps} this proves the main display of
Theorem~\ref{thm:A}.

\emph{Step 5: $N_{1/2}$.} $N_{1/2}=\Ntil(0)$ and
$\Ntil(0)-\Ntil(3)=\#\{0<w_n\le3\}\le
D_{\mathrm{loc}}(0)+D_{\mathrm{loc}}(1)+D_{\mathrm{loc}}(2)\le C\ln c$
by Lemma~\ref{lem:density}(i); and \eqref{eq:NvModel} at $v=3$ gives
$\Ntil(3)=c+O(\ln c)$. \hfill$\qed$

\begin{remark}\label{rem:beyond}
\eqref{eq:NvModel} is the head-side half of a uniform quantile theorem:
the crossing index at level
$1-\eps$ is located to within $O(\ln c+\Lb)$, uniformly across the whole
$\log^2c$ window and far beyond it (depths up to $\Lb\asymp c^{1/3}$,
i.e.\ index displacements up to $\asymp c^{1/3}\ln c$). No pointwise
eigenvalue asymptotic enters: the determinant controls the counting
function directly, which is how the window obstruction of
\cite[\S2]{KDL} is bypassed on this side.
\end{remark}

\section{Proofs of Theorem \ref{thm:B} and the corollaries}\label{sec:thmB}

\begin{proof}[Proof of Theorem \ref{thm:B}]
Write $W:=\frac{L}{\pi^2}\ln\frac{4\pi^2c}{L}$ for the main term, so
that \eqref{eq:Cstar} reads $|\Lambda^+_\delta-W|\le C_\star(\ln c+L)$.
Since $L\le c^{1/3}$,
\[
\ln\frac{4\pi^2c}{L}=\ln c-\ln L+\ln4\pi^2
\;\ge\;\tfrac23\ln c+\ln4\pi^2\;>\;\tfrac23\ln c,
\qquad\text{whence}\qquad
W\;\ge\;\frac{2}{3\pi^2}\,L\,\ln c .
\]
Dominate the two error pieces \emph{separately} against $W$, each with
half of $\eta$: by the displayed anchor inequality,
\[
\begin{aligned}
C_\star\ln c\;&\le\;\frac{\eta}{3\pi^2}\,L\ln c\;\le\;\frac\eta2\,W
&&\iff\; L\ge\frac{3\pi^2C_\star}{\eta},\\[2pt]
C_\star L\;&\le\;\frac{\eta}{3\pi^2}\,L\ln c\;\le\;\frac\eta2\,W
&&\iff\; \ln c\ge\frac{3\pi^2C_\star}{\eta}.
\end{aligned}
\]
Adding the two, $C_\star(\ln c+L)\le\eta\,W$ for all
$L\ge L_0(\eta):=\max\bigl(7,\,3\pi^2C_\star/\eta\bigr)$ and
$c\ge c_0(\eta):=e^{3\pi^2C_\star/\eta}$, which is exactly the two-sided
display. (The pairing matters: it is $L$ that can be large and $\ln c$
that then fails to keep up, so $C_\star\ln c$ must be met by a threshold
on $L$ and $C_\star L$ by a threshold on $c$. The latter cannot be
beaten by any threshold on $L$ alone, which is why $c_0(\eta)$ is forced
to be exponential in $1/\eta$.)

For the bridge form: on $c^{-\alpha}\le\delta\le\delta_0=e^{-L_0}$ we
have $L\le\alpha\ln c$, hence $L\le c^{1/3}$ for $c\ge c_1(\alpha)$, and
$\ln\frac{\alpha c}{L}\ge\ln\frac{c}{\ln c}\to\infty$ uniformly on the
bridge. The two logarithms in play differ by a constant, not a factor:
\[
\ln\frac{4\pi^2c}{L}=\ln\frac{\alpha c}{L}+\ln\frac{4\pi^2}{\alpha}.
\]
If $\alpha\le4\pi^2$ then $\ln\frac{4\pi^2}{\alpha}\ge0$ and
$\ln\frac{4\pi^2c}{L}\ge\ln\frac{\alpha c}{L}$ outright; if
$\alpha>4\pi^2$ then $\bigl|\ln\frac{4\pi^2}{\alpha}\bigr|$ is a fixed
constant while $\ln\frac{\alpha c}{L}\to\infty$, so it is at most
$\eta'\ln\frac{\alpha c}{L}$ once $c\ge c_0(\alpha,\eta')$. Either way,
for every $\eta'\in(0,1)$,
\begin{equation}\label{eq:logcompare}
\ln\frac{4\pi^2c}{L}\;\ge\;(1-\eta')\ln\frac{\alpha c}{L}
\qquad\text{for }c\ge c_0(\alpha,\eta'),\;
c^{-\alpha}\le\delta\le\delta_0 .
\end{equation}
Take $\eta=\eta'=\tfrac14$ and combine the two-sided bound with
\eqref{eq:logcompare}:
\[
\Lambda^+_\delta\;\ge\;(1-\eta)\frac{L}{\pi^2}\ln\frac{4\pi^2c}{L}
\;\ge\;(1-\eta)(1-\eta')\frac{L}{\pi^2}\ln\frac{\alpha c}{L}
\;=\;\frac{9}{16}\cdot\frac{L}{\pi^2}\ln\frac{\alpha c}{L}
\;\ge\;\frac{1}{2\pi^2}\,L\,\ln\frac{\alpha c}{L},
\]
since $\tfrac9{16}=0.5625>\tfrac12$. This is the bridge form, with
$C=0$, $\delta_0=e^{-L_0(1/4)}$ and
$c\ge c_0(\alpha):=\max\bigl(c_0(\tfrac14),c_0(\alpha,\tfrac14),
c_1(\alpha)\bigr)$.
\end{proof}

\begin{remark}\label{rem:kappaceiling}
Letting $\eta,\eta'\to0$ in the last display, the same route gives the
bridge form with any $\kappa<\pi^{-2}$, at the price of larger
$L_0,\delta_0,c_0$.  This is a proved lower-bound constant in the
joint range; it is not a claim that $\pi^{-2}$ is a sharp asymptotic
constant when $L$ is bounded.  The stated value $\tfrac1{2\pi^2}$ is a
rounding comfortably inside the proved range.  The wasteful comparison
$\ln\frac{4\pi^2c}{L}\ge\tfrac12\ln\frac{\alpha c}{L}$ --- which discards
a factor $2$ where \eqref{eq:logcompare} costs only $1-\eta'$ --- would
cap the route at $\tfrac1{4\pi^2}$.
\end{remark}

\begin{lemma}[tensor identity; restated and reproved]\label{lem:tensor}
For the cube pair $A=(0,c)^d$, $B=(-\tfrac12,\tfrac12)^d$, the
compression of $P_AQ_BP_A$ to
$\operatorname{Ran}P_A\simeq L^2(A)$ is unitarily equivalent to
$S_c^{\otimes d}$.  On the full space $L^2(\R^d)$ there is additionally
a zero eigenspace, irrelevant to every positive-threshold count.  The
positive eigenvalue multiset is
$\{\prod_{i=1}^d\lambda_{n_i}(c)\}$ over multi-indices
$(n_1,\dots,n_d)$.
\end{lemma}

\begin{proof}
On $\operatorname{Ran}P_A$, indicators of product sets factor as tensor
products, and the $d$-dimensional Fourier transform is the tensor product
of the one-dimensional ones; hence the compression of $P_AQ_BP_A$ is
unitarily equivalent to $S_c^{\otimes d}$. Choosing an orthonormal eigenbasis in
each factor, elementary tensors form a complete orthonormal system of
eigenvectors with the product eigenvalues.
\end{proof}

\begin{proof}[Proof of Corollary \ref{cor:tensor}]
Let $\delta\le2^{-d}$ and take any multi-index whose $d$ entries all
satisfy $\tfrac12<\lambda_{n_i}\le1-\delta$. The product $\mu$ of the
$d$ factors obeys $\mu>2^{-d}\ge\delta$ and $\mu\le1-\delta$ (indeed
$\mu\le(1-\delta)^d\le1-\delta$). Hence each such multi-index is counted
by $\Lambda_\delta(c;d)$, and there are $(\Lambda^+_\delta(c))^d$ of
them by Lemma~\ref{lem:tensor}. Apply Theorem~\ref{thm:B}.
\end{proof}

\begin{proof}[Proof of Corollary \ref{cor:upperD}]
By \eqref{eq:identities}, $D=\Lambda_\delta-\Lambda^+_\delta$.
Theorem~\ref{qt:KRD} with $\ln(50c+25)\le\ln c+\ln51$ and
$\ln\frac5{\delta(1-\delta)}=L+\ln\frac5{1-\delta}\le L+\ln10$ gives
\[
\Lambda_\delta\;\le\;\frac{2}{\pi^2}\bigl(\ln c+\ln51\bigr)
\bigl(L+\ln10\bigr)+9
\;=\;\frac{2L}{\pi^2}\ln c+\frac{2\ln51}{\pi^2}\,L+O(\ln c),
\qquad \frac{2\ln51}{\pi^2}=0.796754\ldots,
\]
while \eqref{eq:Cstar} --- i.e.\ Theorem~\ref{thm:A} used directly, with
no appeal to Theorem~\ref{thm:B} and hence no $\eta$ anywhere --- gives
\[
\Lambda^+_\delta\;\ge\;\frac{L}{\pi^2}\bigl(\ln c-\ln L+\ln4\pi^2\bigr)
-C_\star(\ln c+L)
=\frac{L}{\pi^2}\ln c-\frac{L}{\pi^2}\ln L+\frac{\ln4\pi^2}{\pi^2}L
-C_\star L-O(\ln c),
\]
with $\frac{\ln4\pi^2}{\pi^2}=0.372432\ldots$. Subtracting, the terms
$\frac{2L}{\pi^2}\ln c-\frac{L}{\pi^2}\ln c$ leave $\frac{L}{\pi^2}\ln c$
and the two $\Theta(L)$ contributions combine into $C_2L$:
\[
D\;\le\;\frac{L}{\pi^2}\bigl(\ln c+\ln L\bigr)+C_2\,L+C'\ln c,
\qquad
C_2=\frac{2\ln51-\ln4\pi^2}{\pi^2}+C_\star=0.424323\ldots+C_\star,
\]
which is \eqref{eq:upperD-a}; every constant in it is absolute.
For the second display, let $L\le c^{\eta/2}$. Then
$\ln L\le\tfrac\eta2\ln c$, so
$\frac{L}{\pi^2}\ln L\le\tfrac\eta2\cdot\frac{L}{\pi^2}\ln c$; and
$C_2L\le\tfrac\eta2\cdot\frac{L}{\pi^2}\ln c$ as soon as
$\ln c\ge2\pi^2C_2/\eta$, i.e.\ $c\ge c_0(\eta):=e^{2\pi^2C_2/\eta}$.
Adding, $D\le(1+\eta)\frac{L}{\pi^2}\ln c+C'\ln c$, so the second
display holds with $C_\eta=C'$ absolute.
\end{proof}

\begin{remark}\label{rem:squeeze}
Corollary~\ref{cor:upperD} says the two half-windows cannot both exceed
their Landau--Widom shares: the upper half provably takes its share
$\pi^{-2}\Lb\ln c\,(1+o(1))$ when $\Lb\to\infty$ in the range of
Theorem~\ref{thm:B}; the lower half is capped at the same order, up to
the slack in Theorem~\ref{qt:KRD}'s constant.  A matching
lower-half asymptotic would require an additional lower bound.
\end{remark}

\section{The reduction: proof of Theorem \ref{thm:red}}\label{sec:reduction}

Throughout this section $\beta=e^u-1$, $u>0$, and, as in \eqref{eq:GF},
\[
G_c(u)\;=\;\ln\det(I+\beta S_c)\;=\;\sum_{n\ge0}\ln(1+\beta\lambda_n)
\;=\;F_c(-u/2),
\]
finite since $S_c$ is trace class: this is the tail half of the
determinant \eqref{eq:gammav}. Two elementary facts are used
repeatedly: $\ln(1+\beta t)\le u$ for $0<t\le1$, and
$\ln(1+\beta t)\le\beta t$.

\begin{lemma}[tail-sum bound]\label{lem:tailsum}
For all $c\ge2$ and $0<\tau\le\tfrac12$,
\[
\sum_{\lambda_n\le\tau}\lambda_n\;\le\;\frac{4}{\pi^2}\,\tau\,
\ln(50c+25)\,\bigl(\ln\tfrac1\tau+6\bigr)\;+\;16\,\tau .
\]
\end{lemma}

\begin{proof}
Decompose dyadically:
\[
\sum_{\lambda\le\tau}\lambda
=\sum_{j\ge0}\ \sum_{\lambda\in(2^{-j-1}\tau,\,2^{-j}\tau]}\lambda
\;\le\;\sum_{j\ge0}2^{-j}\tau\,
\#\{\lambda\in(2^{-j-1}\tau,2^{-j}\tau]\}.
\]
Each dyadic window lies in $(\eps_j,1-\eps_j)$ with
$\eps_j=2^{-j-1}\tau$, so by Theorem~\ref{qt:KRD} its count is at
most $\frac{2}{\pi^2}\ln(50c+25)\bigl(\ln\tfrac1\tau+(j+1)\ln2+\ln10
\bigr)+8$. Summing against $2^{-j}$:
$\sum_j2^{-j}(j+1)\le4$, $\sum_j2^{-j}\le2$, giving the claim.
\end{proof}

\subsection{\texorpdfstring{\eqref{eq:T} implies
\eqref{eq:bridgealpha}}{The tail bound implies the bridge bound}}\label{sec:red1}

Assume \eqref{eq:T} with constants $\kappa_T,L_T,c_T$. Let
$u\in[\max(L_T,1),\alpha\ln c]$ and $\tau=e^{-u}$. Splitting $G_c(u)$ at
$\tau$,
\[
G_c(u)\;\le\;u\,N_{\tau}(c)\;+\;\beta\sum_{\lambda_n\le\tau}\lambda_n
\;\le\;u\,N_\tau(c)+e^u\Bigl[\frac4{\pi^2}e^{-u}\ln(50c+25)(u+6)
+16e^{-u}\Bigr],
\]
using Lemma~\ref{lem:tailsum} (note $\tau\le e^{-1}<\tfrac12$). With
\eqref{eq:T},
\[
N_{e^{-u}}(c)\;\ge\;c+\kappa_T\,u\ln c
-\frac4{\pi^2}\Bigl(1+\frac6u\Bigr)\ln(50c+25)-\frac{16}{u}
\;\ge\;c+\frac{\kappa_T}{2}\,u\ln c
\]
for all $u\ge L_0':=\max\bigl(L_T,\,\lceil40/(\pi^2\kappa_T)\rceil,\,
6\bigr)$ and $c\ge c_0'$. Now let $\delta\in[c^{-\alpha},e^{-L_0'}]$ and
$u=L=\ln\frac1\delta\in[L_0',\alpha\ln c]$. By Theorem~\ref{thm:A},
$N_{1/2}\le c+C_{\mathrm h}\ln c$, so
\[
D(\delta,c)=N_\delta-N_{1/2}\;\ge\;\frac{\kappa_T}{2}L\ln c-C_{\mathrm h}\ln c
\;\ge\;\frac{\kappa_T}{4}\,L\ln c
\]
for $L\ge L_0'':=\max(L_0',\lceil4C_{\mathrm h}/\kappa_T\rceil)$. Finally, on the
bridge $\ln\frac{\alpha c}{L}\le\ln(\alpha c)\le2\ln c$ for
$c\ge\alpha$, whence $D\ge\frac{\kappa_T}{8}L\ln\frac{\alpha c}{L}$;
i.e.\ \eqref{eq:bridgealpha} holds with $\kappa=\kappa_T/8$, $C=0$,
$\delta_0=e^{-L_0''}$. \hfill$\qed$

\subsection{\texorpdfstring{\eqref{eq:bridgealpha} implies
\eqref{eq:T}}{The bridge bound implies the tail bound}}\label{sec:red2}

Assume \eqref{eq:bridgealpha} with constants
$\alpha,\delta_0,\kappa,C,c_0$; write $L_0=\ln\frac1{\delta_0}$. By the
layer-cake identity (equivalently, by Tonelli:
$G_c(u)=\sum_n\int_0^{\lambda_n}\frac{\beta}{1+\beta t}\,dt
=\int_0^1 N_t\,\frac{\beta}{1+\beta t}\,dt$, the integrand being
nonnegative),
\[
G_c(u)=\int_0^1 N_t\,\frac{\beta}{1+\beta t}\,dt
= cu+\int_0^1\bigl(N_t-c\bigr)\frac{\beta}{1+\beta t}\,dt .
\]
Split $(0,1)=(0,e^{-L_0}]\cup(e^{-L_0},\tfrac12]\cup(\tfrac12,1)$.

\emph{Range $(\tfrac12,1)$:} here $\frac{\beta}{1+\beta t}\le\frac1t\le2$,
so
\[
\int_{1/2}^1\bigl(N_t-c\bigr)\frac{\beta}{1+\beta t}\,dt
\;\ge\;-2\int_{1/2}^1(c-N_t)_+\,dt .
\]
Now $(c-N_t)_+=(c-N_t)+(N_t-c)_+$, and for $t>\tfrac12$ monotonicity of
$N$ and Theorem~\ref{thm:A} give
$(N_t-c)_+\le(N_{1/2}-c)_+\le C_{\mathrm h}\ln c$, so
\[
\int_{1/2}^1(c-N_t)_+\,dt\;\le\;\int_{1/2}^1(c-N_t)\,dt
\;+\;\tfrac12C_{\mathrm h}\ln c .
\]
Since $\int_0^1N_t\,dt=\Tr S_c=c$ and
$\int_0^{1/2}N_t\,dt=\tfrac12N_{1/2}+\sum_{\lambda_n\le1/2}\lambda_n$,
\[
\int_{1/2}^1(c-N_t)\,dt=\int_0^{1/2}(N_t-c)\,dt
=\tfrac12\,(N_{1/2}-c)+\sum_{\lambda_n\le1/2}\lambda_n
\;\le\;C'\ln c,
\]
by Theorem~\ref{thm:A} and Lemma~\ref{lem:tailsum} at $\tau=\tfrac12$.
So this range contributes $\ge-C''\ln c$.

\emph{Range $(e^{-L_0},\tfrac12]$:} $N_t\ge N_{1/2}\ge c-C_{\mathrm h}\ln c$, and
$\int\frac{\beta}{1+\beta t}dt\le\int_{e^{-L_0}}^{1/2}\frac{dt}t=
L_0-\ln2$; contribution $\ge-C_{\mathrm h}L_0\ln c$.

\emph{Range $(0,e^{-L_0}]$:} for $t\in[2e^{-u},e^{-L_0}]$ (nonempty for
$u\ge L_0+2$) monotonicity of $N$ and \eqref{eq:bridgealpha} give, with
$j=\ln\frac1t\in[L_0,u-\ln2]$,
\[
N_t-c\;\ge\;\kappa\,\min(j,\alpha\ln c)\,
\ln\frac{\alpha c}{\min(j,\alpha\ln c)}-C-C_{\mathrm h}\ln c
\;\ge\;\kappa\,\min(j,\alpha\ln c)\cdot\tfrac12\ln c-C-C_{\mathrm h}\ln c,
\]
for $c\ge c_1(\alpha)$ (on the bridge
$\ln\frac{\alpha c}{L}\ge\tfrac12\ln c$ for large $c$; for
$j>\alpha\ln c$ we simply use $N_t\ge N_{c^{-\alpha}}$). On this range
$\frac{\beta}{1+\beta t}\in\bigl[\frac1{2t},\frac1t\bigr]$: the lower bound
since $\beta t\ge\beta\cdot2e^{-u}\ge1$, the upper always. Write
$B_j:=\frac{\kappa\ln c}{2}\min(j,\alpha\ln c)-C-C_{\mathrm h}\ln c$; for
$j<\alpha\ln c$ its root is $j_0:=\frac{2(C+C_{\mathrm h}\ln c)}{\kappa\ln c}$, which
is $<\alpha\ln c$ for large $c$ (the root before the plateau of
$\min(j,\alpha\ln c)$). For large $c$ the plateau value
$\frac{\kappa\alpha}{2}(\ln c)^2-C-C_{\mathrm h}\ln c$ is positive, so $B_j\ge0$ for
$j\ge\alpha\ln c$ as well. Set $j_\star:=\max\{L_0,j_0\}$ ($B_j<0$ for
$j<j_0$, $B_j\ge0$ for $j\ge j_0$); if $j_0\le L_0$ the window
$[L_0,j_\star]$ is empty. Either way $j_\star-L_0=O(1)$. Substituting $t=e^{-j}$, so that
$\frac{\beta}{1+\beta t}\,dt=\frac{\beta t}{1+\beta t}\,dj$ with
$\frac{\beta t}{1+\beta t}\in[\tfrac12,1]$, and dropping the positive part
beyond $j=u-\ln2$, split the $j$-integral at $j_\star$. On $[L_0,j_\star]$,
where $B_j<0$, monotonicity gives $N_t-c\ge N_{1/2}-c\ge-C_{\mathrm h}\ln c$ and
$\frac{\beta t}{1+\beta t}\le1$, so this window contributes
$\ge-C_{\mathrm h}\ln c\,(j_\star-L_0)\ge-C'''\ln c$; on $[j_\star,u-\ln2]$, where
$B_j\ge0$, we have $N_t-c\ge B_j\ge0$ and
$\frac{\beta t}{1+\beta t}\ge\frac12$, so that window contributes
$\ge\frac12\int_{j_\star}^{u-\ln2}B_j\,dj\ge\frac12\int_{L_0}^{u-\ln2}B_j\,dj$
(extending the lower limit only removes negative area). Hence
\[
\int_{2e^{-u}}^{e^{-L_0}}(N_t-c)\frac{\beta}{1+\beta t}dt
\;\ge\;\frac12\int_{L_0}^{u-\ln2}
\Bigl[\frac{\kappa\ln c}{2}\min(j,\alpha\ln c)-C-C_{\mathrm h}\ln c\Bigr]dj\;-\;C'''\ln c .
\]
For $u\le\alpha\ln c$ the minimum is $j$ throughout, and the integral is
$\ge\frac{\kappa\ln c}{4}\cdot\frac{(u-\ln2)^2-L_0^2}{2}
-(C+C_{\mathrm h}\ln c)\,u\ge\frac{\kappa}{16}\,u^2\ln c$ for all
$u\ge L_T:=\max\bigl(4L_0+4,\;\lceil32(C+C_{\mathrm h})/\kappa\rceil\bigr)$ and
$c\ge c_2$. (For the last inequality: $\ln c\ge1$ and
$u\ge32(C+C_{\mathrm h})/\kappa$ give
$(C+C_{\mathrm h}\ln c)u\le(C+C_{\mathrm h})(\ln c)u\le\frac\kappa{32}u^2\ln c$, so it
suffices that $(u-\ln2)^2-L_0^2\ge\frac34u^2$; and since
$\delta_0<\tfrac12$ forces $L_0>\ln2$, while $u\ge4L_0+4$ forces
$L_0\le\frac{u-4}4$, the left side is at least
$(u-\ln2)^2-\frac{(u-4)^2}{16}$, and
$(u-\ln2)^2-\frac{(u-4)^2}{16}-\frac34u^2
=\frac3{16}u^2-(2\ln2-\tfrac12)u+(\ln^2\!2-1)$ is positive for
$u\ge4\ln2+4$, hence throughout.) On $(0,2e^{-u})$ finally $N_t-c\ge N_{2e^{-u}}-c\ge0$ for
such $u$, so that piece is nonnegative.

Collecting the three ranges,
\[
G_c(u)\;\ge\;cu+\frac{\kappa}{16}u^2\ln c-C''\ln c-C_{\mathrm h}L_0\ln c-C'''\ln c
\;\ge\;cu+\frac{\kappa}{32}u^2\ln c
\]
for $u\ge L_T$, after enlarging $L_T$ by an amount depending only on the
data in $(\mathrm{BRIDGE}_\alpha)$, and $c\ge c_T$. This is
\eqref{eq:T} with $\kappa_T=\kappa/32$. \hfill$\qed$

\subsection{Completion of the bridge range}\label{sec:full}

Neither direction above says anything about
$\alpha^{-c}<\delta<c^{-\alpha}$: \S\ref{sec:red1} delivers only
$c^{-\alpha}\le\delta\le\delta_0$, and \S\ref{sec:red2} consumes
\eqref{eq:bridgealpha} only there (the truncation at
$\min(j,\alpha\ln c)$ exists precisely so that nothing below
$c^{-\alpha}$ is ever used). That complementary range is settled in the
literature, which closes the gap.

\begin{proof}[Proof of Proposition \ref{prop:full}]
$(\Leftarrow)$ Assume $(\mathrm{BRIDGE}_{\alpha_1})$ with constants
$\kappa,C,\delta_0,c_0$. On the complementary interval
$\alpha_1^{-c}<\eps<c^{-\alpha_1}$,
Theorem~\ref{qt:KDL}(b) gives
$\Lambda^-_\eps(c)\gtrsim LR$, where
$R=\ln(\alpha_1c/L)$. Thus there are $\kappa_1>0$ and $c_1<\infty$
such that
\[
D(\delta,c)=\Lambda^-_\delta(c)\;\ge\;\kappa_1\,L\,
\ln\frac{\alpha_1c}{L}
\qquad\text{for }c\ge c_1\text{ and }
\alpha_1^{-c}<\delta<c^{-\alpha_1}
\]
(endpoint conventions differ from \eqref{eq:counts} by at most $2$,
absorbed into $C$ as in \S\ref{sec:quoted}). Let $c\ge\max(c_0,c_1)$ and
let $\delta$ satisfy $\alpha_1^{-c}<\delta\le\delta_0$. Exactly one of
the two ranges applies:
\[
\begin{aligned}
\delta\ge c^{-\alpha_1}:\quad
&D\ge\kappa\,L\ln\tfrac{\alpha_1c}{L}-C
&&\text{by }(\mathrm{BRIDGE}_{\alpha_1});\\
\delta<c^{-\alpha_1}:\quad
&D\ge\kappa_1\,L\ln\tfrac{\alpha_1c}{L}
&&\text{by Thm.~\ref{qt:KDL}(b)} .
\end{aligned}
\]
In both cases
$D(\delta,c)\ge\min(\kappa,\kappa_1)\,L\ln\frac{\alpha_1c}{L}-C$, which
is \eqref{eq:bridge} with $\alpha:=\alpha_1\ge4$,
$\kappa:=\min(\kappa,\kappa_1)$, and the same $C,\delta_0$.

$(\Rightarrow)$ Assume Problem~\ref{op:bridge} has a positive
answer, with constants $\alpha\ge4,\delta_0,\kappa,C,c_0$. For $c$ large
enough that $c^{-\alpha_1}>\alpha^{-c}$, the hypothesis covers every
$\delta\in[c^{-\alpha_1},\delta_0]$ and gives
$D\ge\kappa L\ln\frac{\alpha c}{L}-C$ there. Convert the constant inside
the logarithm by
$\ln\frac{\alpha c}{L}=\ln\frac{\alpha_1c}{L}+\ln\frac\alpha{\alpha_1}$.
If $\alpha\ge\alpha_1$ the correction is $\ge0$ and
$(\mathrm{BRIDGE}_{\alpha_1})$ follows with the same $\kappa,C$. If
$\alpha<\alpha_1$ then $\ln\frac\alpha{\alpha_1}$ is a fixed negative
constant, while on $\delta\ge c^{-\alpha_1}$ one has $L\le\alpha_1\ln c$
and hence $\ln\frac{\alpha_1c}{L}\ge\ln\frac{c}{\ln c}\to\infty$; so
$\ln\frac{\alpha c}{L}\ge\tfrac12\ln\frac{\alpha_1c}{L}$ for
$c\ge c_2(\alpha,\alpha_1)$, giving
$D\ge\tfrac\kappa2L\ln\frac{\alpha_1c}{L}-C$. Either way
$(\mathrm{BRIDGE}_{\alpha_1})$ holds.
\end{proof}

\begin{proof}[Proof of Corollary \ref{cor:equiv}]
Combine Proposition~\ref{prop:full} with Theorem~\ref{thm:red} at
$\alpha=\alpha_1$.
\end{proof}

Only part (b) of Theorem~\ref{qt:KDL} is used, and only through
its unquantified $\gtrsim$: Problem~\ref{op:bridge} is existential
in $\kappa$, so all that is needed is that \emph{some} $\kappa_1>0$
exists with the implied constant uniform over
$\alpha_1^{-c}<\eps<c^{-\alpha_1}$ and $c\ge c_1$, which is what
\cite[Thm.~1.1]{KDL} asserts. No numerical value of $\kappa_1$ enters.
The reduction is lossy in the constant ($\kappa\mapsto
\min(\kappa,\kappa_1)$, and $\kappa\mapsto\kappa/2$ in the converse),
which is harmless for the same reason --- the same bookkeeping already
accepted in Theorem~\ref{thm:red}. The range $\delta\le\alpha_1^{-c}$ is
not part of Problem~\ref{op:bridge}, so part (c) is not needed.

\begin{remark}[what the equivalence buys]\label{rem:redmeaning}
Both directions cost only constants: the bridge problem \emph{is} the
tail-side determinant problem at logarithmic depth $u\le\alpha\ln c$.
At fixed $u$, \eqref{eq:T} follows from
\cite[Thm.~1.1]{Charlier21}; the new issue is uniformity as $u$ grows
like $\ln c$.  Section~\ref{sec:tail} proves precisely that range.  The
result is deliberately not claimed on the much larger
$u\asymp s^{1/3}$ head-side scale of \cite{BDIK2}.
\end{remark}

\section{Signed determinant asymptotics}\label{sec:tail}

\subsection{Statement and strategy}
\label{sec:tail-statement}

For $s>0$ let $K_s$ be as in Lemma~\ref{lem:dict} and put
\begin{equation}\label{eq:Ddef}
 \mathcal D(s,\omega)=\det\bigl(I+(e^{2\omega}-1)K_s\bigr),
 \qquad\omega\ge0 .
\end{equation}

\begin{theorem}[signed growing-parameter sine-kernel determinant]
\label{thm:signedmain}
Fix $A>0$.  There are
constants $s_A\ge5$ and $C_A<\infty$, depending only on $A$, such that
for all $s\ge s_A$ and all $\omega$ with $0\le\omega\le A\ln s$,
\begin{equation}\label{eq:RHmain}
 \ln\mathcal D(s,\omega)
 =\frac{4\omega s}{\pi}+\frac{2\omega^2}{\pi^2}\ln(4s)
 +2\ln\bigl[\BG(1+i\omega/\pi)\BG(1-i\omega/\pi)\bigr]
 +\mathcal R_A(s,\omega),
\end{equation}
where
\begin{equation}\label{eq:RHerr}
 |\mathcal R_A(s,\omega)|\;\le\;C_A\,\frac{(1+\omega)^4\ln^2s}{s}.
\end{equation}
Both logarithms are the real ones: $\mathcal D(s,\omega)>0$ by
Lemma~\ref{lem:solv}, and the Barnes product is positive by
Lemma~\ref{lem:barnesproduct}.
\end{theorem}

The proof below derives the model jumps and origin behaviour in
Lemmas~\ref{lem:modeljumps} and \ref{qmi:third}, then carries the
endpoint matching through the small-norm problem and the differential
identity.

\paragraph{Comparison with \cite{BDIK2}.}
The published theorem \cite[Thm.~1.2]{BDIK2} is stated for
$\gamma\in[0,1)$, i.e.\ for $v=-\tfrac12\ln(1-\gamma)\ge0$.  We need
$\gamma=1-e^{2\omega}\le0$, i.e.\ $v=-\omega\le0$.  The estimate
\eqref{eq:RHmain} does not follow directly by analytic continuation from
that result, because its remainder bound is established only on
$[0,1)$.  We therefore carry out the nonlinear steepest-descent analysis
for $\gamma<0$, keeping track of the four places where the sign matters:

\begin{enumerate}[label=(S\arabic*),leftmargin=2.6em]
\item \emph{Solvability.} For $\gamma\le0$ the operator
$I-\gamma K_s\succeq I$ is directly invertible
(Lemma~\ref{lem:solv}), so no exceptional-parameter analysis is
required for the associated Riemann--Hilbert problem.
\item \emph{The lens coefficient.}  The quantity that multiplies the
oscillatory exponential in the lens factorization is $\gamma e^{2v}$.
On the head side this equals $e^{2v}-1$ and is exponentially
\emph{large}; on our side it equals $e^{-2\omega}-1$ and has modulus at
most $1$ (Lemma~\ref{lem:factorization}).  The sign is favourable here.
\item \emph{The endpoint model.}  The parameter of the model is
$\nu=iv/\pi=-i\omega/\pi$, purely imaginary in both cases; only the sign
of its imaginary part changes.  Every identity used is algebraic in
$\nu$ and hence valid for either sign.  What does change is which
entries of the raw model carry $e^{+\omega}$ and which carry
$e^{-\omega}$; these are tracked individually in
Appendix~\ref{app:ledger}.
\item \emph{The integration constant.}  The $s$-differential identity
determines $\ln\mathcal D(s,\omega)$ only up to a function of $\omega$.
We fix that function from a fixed-parameter theorem whose hypotheses
genuinely cover our sign, namely Theorem~\ref{qt:Charlier} at
$s_1=e^{2\omega}>1$; see Lemma~\ref{lem:charlierconvert} and
Remark~\ref{rem:charlier-scope}.
\end{enumerate}

\paragraph{Standing conventions for \S\ref{sec:tail}.}
Throughout this section $A>0$ is fixed; $C_A,c_A,\dots$ denote positive
constants depending only on $A$, whose value may change from occurrence
to occurrence, and $C$ denotes an absolute constant.  We write
\begin{equation}\label{eq:params}
 v=-\omega\le0,\qquad
 \gamma=1-e^{-2v}=1-e^{2\omega}\le0,\qquad
 \kappa=\frac vs,\qquad
 \nu=\frac{iv}\pi=-\frac{i\omega}\pi ,
\end{equation}
so that $2\kappa s=2v$ and $e^{-2\kappa s\sigma_3}=e^{-2v\sigma_3}$, and
\begin{equation}\label{eq:sigmas}
 \sigma_1=\begin{pmatrix}0&1\\1&0\end{pmatrix},\qquad
 \sigma_3=\begin{pmatrix}1&0\\0&-1\end{pmatrix} .
\end{equation}
For a matrix $M$, $\|M\|$ is the maximum of the moduli of its entries.
The \emph{endpoint radius} is
\begin{equation}\label{eq:rs}
 r_s:=(\ln s)^{-2},\qquad s\ge5 ,
\end{equation}
and we always assume $0\le\omega\le A\ln s$, so that
\begin{equation}\label{eq:sigmasmall}
 \frac{1+\omega}{s\,r_s}\le\frac{(1+A\ln s)(\ln s)^2}{s}
 =O_A\!\Bigl(\frac{\ln^3s}{s}\Bigr)\xrightarrow[s\to\infty]{}0 .
\end{equation}

\paragraph{Branch conventions.}
For $z\in\C\setminus(-\infty,0]$ we write
$z^{a}:=e^{a\operatorname{Log}z}$ with $\operatorname{Log}$ the
principal logarithm, $|\arg z|<\pi$.  The local variable $\zeta$ carries
the convention
\begin{equation}\label{eq:zetabranch}
 -\pi<\arg\zeta\le\pi .
\end{equation}
The notation $U(a;z):=U(a,1;z)$ ordinarily denotes the principal
Tricomi function, analytic on $\C\setminus(-\infty,0]$.  In the four
compositions $U(a;e^{\pm i\pi/2}\zeta)$ in \eqref{eq:Fmatrix}, however,
it denotes the analytic continuation from the central sector
$|\arg\zeta|<\pi/2$ to the logarithmic cover determined by
\begin{equation}\label{eq:Ubranch}
 \operatorname{ph}(e^{-i\pi/2}\zeta)
 =\arg\zeta-\frac{\pi}{2},
 \qquad
 \operatorname{ph}(e^{i\pi/2}\zeta)
 =\arg\zeta+\frac{\pi}{2},
\end{equation}
without reducing either phase modulo $2\pi$.  Thus, when one of these
phases crosses $\pm\pi$, the corresponding value is continued through
the cut on the indicated sheet rather than replaced by its principal
value.  On this logarithmic cover the standard large-$z$ expansion
$U(a;z)\sim z^{-a}$ holds on closed subsectors of
$-3\pi/2<\operatorname{ph}z<3\pi/2$, with the power $z^{-a}$ formed
using the same lifted phase \cite[\S13.2, \S13.7]{NIST}.  This is an
asymptotic sector, not a claim that the principal Tricomi function is
analytic across its cut.  These lifted conventions follow
\cite[\S2.2]{BDIK2}.

As $\arg\zeta$ varies through either of the internal rays
$\arg\zeta=\pm\pi/2$, the two phases in \eqref{eq:Ubranch} vary
continuously on the cover.  Consequently all four compositions in
\eqref{eq:Fmatrix} are analytic across those two rays.  Boundary values
at $\arg\zeta=\pi$ are understood one-sidedly.

\paragraph{The case $\omega=0$.}
If $\omega=0$ then $\gamma=0$, $\mathcal D(s,0)=1$, and the right-hand
side of \eqref{eq:RHmain} is $0+0+2\ln[\BG(1)^2]=0$, so
\eqref{eq:RHmain} holds with $\mathcal R_A\equiv0$.  From now on we may
and do assume $\omega>0$, so that $\nu\ne0$ and every $\Gamma$-ratio
below is finite and nonzero.

\subsection{The determinant, positivity, and trace-norm
differentiability}\label{sec:gate1}

\begin{lemma}[solvability from positivity]\label{lem:solv}
For every $s>0$ and $\omega\ge0$ the operator $K_s$ is trace class on
$L^2(-1,1)$ with $0\le K_s\le I$ and $\Tr K_s=2s/\pi$; the operator
$I-\gamma K_s=I+(e^{2\omega}-1)K_s$ satisfies
\begin{equation}\label{eq:posdef}
 I-\gamma K_s\;\succeq\;I ,
\end{equation}
hence is boundedly invertible with $\|(I-\gamma K_s)^{-1}\|\le1$, and
\begin{equation}\label{eq:detpos}
 \mathcal D(s,\omega)=\det(I-\gamma K_s)
 =\prod_{j\ge1}\bigl(1+(e^{2\omega}-1)\mu_j\bigr)\;\ge\;1>0 ,
\end{equation}
where $\mu_j\in[0,1]$ are the eigenvalues of $K_s$.
\end{lemma}

\begin{proof}
Under the unitary dilation $(V f)(t)=s^{-1/2}f(t/s)$ from $L^2(-1,1)$ to
$L^2(-s,s)$, the kernel $\sin(s(\lambda-\mu))/(\pi(\lambda-\mu))$ becomes
$\sin(t-t')/(\pi(t-t'))$ on $(-s,s)$, the kernel of
$P_{(-s,s)}\Pi P_{(-s,s)}$ with $\Pi$ the Fourier multiplier with symbol
$\mathbf1_{[-1,1]}$, an orthogonal projection.  Hence $0\le K_s\le I$.
Writing $P\Pi P=(\Pi P)^*(\Pi P)$ with $\Pi P$ Hilbert--Schmidt of
Hilbert--Schmidt norm squared $2s/\pi$, $K_s$ is trace class with
$\Tr K_s=2s/\pi$; equivalently, by Mercer,
$\Tr K_s=\int_{-1}^1K_s(\lambda,\lambda)\,d\lambda=2s/\pi$.
Since $e^{2\omega}-1\ge0$ and $K_s\ge0$, \eqref{eq:posdef} is immediate
and gives $\|(I-\gamma K_s)^{-1}\|\le1$.  The determinant of a
trace-class perturbation of the identity is the product of
$1+(e^{2\omega}-1)\mu_j$ \cite[Ch.~3]{Simon05}; each factor is $\ge1$.
\end{proof}

The next lemma establishes the trace-norm differentiability needed for
the differential identity in Lemma~\ref{lem:diffid}.

\begin{lemma}[trace-norm differentiability and Jacobi's formula]
\label{lem:tracenorm}
Fix $\gamma\le0$.  Then:
\begin{enumerate}[label=\textup{(\roman*)},leftmargin=2.6em]
\item the map $(0,\infty)\ni s\mapsto K_s\in\mathfrak S_1(L^2(-1,1))$ is
differentiable in trace norm, with
$\partial_sK_s$ the integral operator with kernel
$\pi^{-1}\cos\bigl(s(\lambda-\mu)\bigr)$;
\item $s\mapsto\det(I-\gamma K_s)$ is differentiable and
\begin{equation}\label{eq:jacobi}
 \frac{\partial}{\partial s}\ln\det(I-\gamma K_s)
 =-\gamma\Tr\Bigl((I-\gamma K_s)^{-1}\,\partial_sK_s\Bigr).
\end{equation}
\end{enumerate}
\end{lemma}

\begin{proof}
(i) Put $c_t(\lambda)=\cos(t\lambda)$,
$q_t(\lambda)=\sin(t\lambda)$, and
\[
 Q_t:=\pi^{-1}\bigl(c_t\otimes c_t+q_t\otimes q_t\bigr),
 \qquad (a\otimes b)h=a\langle b,h\rangle_{L^2(-1,1)}.
\]
Here the inner product is linear in its second argument.
The kernel of $Q_t$ is
$\pi^{-1}\cos(t(\lambda-\mu))=\partial_tK_t(\lambda,\mu)$.
Since $t\mapsto c_t,q_t$ is continuous in $L^2(-1,1)$ and
$\|a\otimes b\|_{\mathfrak S_1}=\|a\|_2\|b\|_2$, the map
$t\mapsto Q_t$ is continuous in trace norm.  The scalar fundamental
theorem of calculus for the kernels therefore gives the Bochner-integral
identity
\[
 K_{s+h}-K_s=\int_s^{s+h}Q_t\,dt
 \quad\text{in }\mathfrak S_1.
\]
It follows that
$h^{-1}(K_{s+h}-K_s)\to Q_s$ in trace norm, proving (i).

(ii) By (i) the map $s\mapsto I-\gamma K_s$ is $\mathfrak S_1$-valued and
differentiable, and by Lemma~\ref{lem:solv} it is invertible with
uniformly bounded inverse on compact $s$-intervals.  Jacobi's formula
for Fredholm determinants \cite[Thm.~3.6]{Simon05} states that if
$s\mapsto T_s$ is $\mathfrak S_1$-differentiable and $I+T_s$ is
invertible, then $\det(I+T_s)$ is differentiable with
$\partial_s\ln\det(I+T_s)=\Tr\bigl((I+T_s)^{-1}\partial_sT_s\bigr)$.
Applying this with $T_s=-\gamma K_s$ gives \eqref{eq:jacobi}.  The trace
is finite because $(I-\gamma K_s)^{-1}$ is bounded and $\partial_sK_s$
is trace class.
\end{proof}

\subsection{The IIKS formulation}
\label{sec:gate2}

The construction below uses the integrable-operator formalism of Its,
Izergin, Korepin and Slavnov \cite{IIKS}; all sign and endpoint conventions
needed here are derived explicitly.

\begin{definition}[the master Riemann--Hilbert problem]\label{def:RHPY}
Orient $(-1,1)$ from $-1$ to $1$; the $+$ side is then the upper side.
Find $Y=Y(\lambda;s,\gamma)\in\C^{2\times2}$ such that:
\begin{enumerate}[label=\textup{(Y\arabic*)},leftmargin=3.1em]
\item $Y$ is analytic in $\C\setminus[-1,1]$, with square-integrable
boundary values
$Y_\pm(\lambda)=\lim_{\eps\downarrow0}Y(\lambda\pm i\eps)$;
\item on $(-1,1)$,
\begin{equation}\label{eq:Yjump}
 Y_+(\lambda)=Y_-(\lambda)\,G_Y(\lambda),
 \qquad
 G_Y(\lambda)=\begin{pmatrix}
 1-\gamma&\gamma e^{2is\lambda}\\
 -\gamma e^{-2is\lambda}&1+\gamma\end{pmatrix};
\end{equation}
\item as $\lambda\to\pm1$,
\begin{equation}\label{eq:Yendpoint}
 Y(\lambda)=\check Y(\lambda)\Bigl[I+\frac{\gamma}{2\pi i}
 \begin{pmatrix}-1&1\\-1&1\end{pmatrix}
 \ln\Bigl(\frac{\lambda-1}{\lambda+1}\Bigr)\Bigr]e^{-is\lambda\sigma_3},
\end{equation}
with $\check Y$ analytic and invertible near the endpoint and the
principal branch of the logarithm;
\item $Y(\lambda)=I+Y_1\lambda^{-1}+O(\lambda^{-2})$ as
$\lambda\to\infty$.
\end{enumerate}
\end{definition}

Definition~\ref{def:RHPY} is \cite[RHP~2.1]{BDIK2} with $\gamma$ now
allowed to be negative.

\begin{lemma}[integrable structure]\label{lem:fg}
Put
\begin{equation}\label{eq:fg}
 f(\lambda)=\frac{1}{2\pi i}
 \begin{pmatrix}\gamma e^{is\lambda}\\ \gamma e^{-is\lambda}\end{pmatrix},
 \qquad
 g(\lambda)=\begin{pmatrix}-e^{-is\lambda}\\ e^{is\lambda}\end{pmatrix}.
\end{equation}
Then, for all $\lambda,\mu\in(-1,1)$,
\begin{equation}\label{eq:fgprops}
 f^{T}(\lambda)g(\lambda)=0,
 \qquad
 \frac{f^{T}(\lambda)g(\mu)}{\lambda-\mu}=-\gamma K_s(\lambda,\mu),
 \qquad
 I+2\pi i\,f(\lambda)g^{T}(\lambda)=G_Y(\lambda).
\end{equation}
\end{lemma}

\begin{proof}
Directly,
\[
 f^{T}(\lambda)g(\mu)
 =\frac{\gamma}{2\pi i}\Bigl[-e^{is(\lambda-\mu)}+e^{-is(\lambda-\mu)}\Bigr]
 =\frac{\gamma}{2\pi i}\bigl(-2i\sin(s(\lambda-\mu))\bigr)
 =-\frac{\gamma}{\pi}\sin\bigl(s(\lambda-\mu)\bigr).
\]
Setting $\mu=\lambda$ gives $f^{T}g=0$; dividing by $\lambda-\mu$ gives
$-\gamma\sin(s(\lambda-\mu))/(\pi(\lambda-\mu))=-\gamma K_s(\lambda,\mu)$.
For the third identity,
\[
 2\pi i\,fg^{T}
 =\gamma\begin{pmatrix}e^{is\lambda}\\e^{-is\lambda}\end{pmatrix}
 \begin{pmatrix}-e^{-is\lambda}&e^{is\lambda}\end{pmatrix}
 =\gamma\begin{pmatrix}-1&e^{2is\lambda}\\-e^{-2is\lambda}&1\end{pmatrix},
\]
and adding $I$ gives $G_Y$.
\end{proof}

\begin{remark}[the sign of the integrable kernel]\label{rem:signkernel}
The middle identity in \eqref{eq:fgprops} says that the IIKS operator
attached to $(f,g)$ --- the operator $\mathcal K$ with kernel
$f^{T}(\lambda)g(\mu)/(\lambda-\mu)$ --- is
\[
 \mathcal K=-\gamma K_s ,\qquad\text{so}\qquad I+\mathcal K=I-\gamma K_s .
\]
It is $I+\mathcal K$, not $I-\mathcal K$, that must be inverted, and
this forces the $+$ sign in \eqref{eq:IIKSY} below.  With this
convention the jump $I+2\pi ifg^{T}$ and the resolvent
$(I-\gamma K_s)^{-1}$ are consistent.
\end{remark}

\begin{lemma}[solution of the master problem]\label{lem:RHPsolv}
Let $\gamma\le0$, let $F:=(I-\gamma K_s)^{-1}f$ (applied to each of the
two components of $f$; the inverse exists by Lemma~\ref{lem:solv}), and
set
\begin{equation}\label{eq:IIKSY}
 Y(\lambda):=I+\int_{-1}^1\frac{F(\mu)g^{T}(\mu)}{\mu-\lambda}\,d\mu,
 \qquad\lambda\in\C\setminus[-1,1].
\end{equation}
Then:
\begin{enumerate}[label=\textup{(\alph*)},leftmargin=2.6em]
\item $Y$ satisfies \textup{(Y1)}, and its boundary values obey
\begin{equation}\label{eq:Fis}
 Y_+(\lambda)f(\lambda)=Y_-(\lambda)f(\lambda)=F(\lambda),
 \qquad\lambda\in(-1,1);
\end{equation}
\item $Y$ satisfies \textup{(Y2)};
\item $Y$ satisfies \textup{(Y3)} and \textup{(Y4)}, with
\begin{equation}\label{eq:Y1formula}
 Y_1=-\int_{-1}^1F(\mu)g^{T}(\mu)\,d\mu ;
\end{equation}
\item $\det Y\equiv1$, hence $\Tr Y_1=0$;
\item $Y$ is the unique solution of Definition~\ref{def:RHPY}.
\end{enumerate}
\end{lemma}

\begin{proof}
Write $C h(\lambda):=\frac1{2\pi i}\int_{-1}^1\frac{h(\mu)}{\mu-\lambda}
\,d\mu$, so that $Y=I+2\pi i\,C(Fg^{T})$, and recall the Plemelj
relations for the interval oriented from $-1$ to $1$ with $+$ the upper
side:
\begin{equation}\label{eq:plemelj}
 (Ch)_+-(Ch)_-=h,
 \qquad
 (Ch)_++(Ch)_-=\text{(principal value)} .
\end{equation}
Only the first is used.

(a) $F\in L^2(-1,1)^2$ because $f$ is bounded and
$\|(I-\gamma K_s)^{-1}\|\le1$; hence $Fg^{T}\in L^2$ and the Cauchy
transform has $L^2$ boundary values, giving (Y1).  For \eqref{eq:Fis},
fix $\lambda\in(-1,1)$ and compute the scalar
$g^{T}(\mu)f(\lambda)=f^{T}(\lambda)g(\mu)$, which by
\eqref{eq:fgprops} equals $-\gamma(\lambda-\mu)K_s(\lambda,\mu)$ and in
particular \emph{vanishes at $\mu=\lambda$}.  Therefore the integrand of
\[
 Y(\lambda')f(\lambda)
 =f(\lambda)+\int_{-1}^1F(\mu)\,
 \frac{g^{T}(\mu)f(\lambda)}{\mu-\lambda'}\,d\mu
\]
has, at $\lambda'=\lambda$, a removable singularity; consequently the
two boundary values coincide and equal the ordinary integral.  Using
$g^{T}(\mu)f(\lambda)/(\mu-\lambda)=+\gamma K_s(\lambda,\mu)$,
\[
 Y_\pm(\lambda)f(\lambda)
 =f(\lambda)+\gamma\int_{-1}^1K_s(\lambda,\mu)F(\mu)\,d\mu
 =f(\lambda)+\gamma(K_sF)(\lambda)
 =\bigl[f+\gamma K_sF\bigr](\lambda).
\]
Since $(I-\gamma K_s)F=f$ we have $f+\gamma K_sF=F$, which is
\eqref{eq:Fis}.

(b) By \eqref{eq:plemelj}, $Y_+-Y_-=2\pi i\,Fg^{T}$.  By
\eqref{eq:Fis}, $F=Y_-f$, so
\[
 Y_+-Y_-=2\pi i\,(Y_-f)g^{T}=Y_-\bigl(2\pi i\,fg^{T}\bigr),
\]
that is $Y_+=Y_-\bigl(I+2\pi ifg^{T}\bigr)=Y_-G_Y$ by
\eqref{eq:fgprops}.  This derives (Y2) from \eqref{eq:IIKSY} without any
appeal to a general correspondence.

(c) For $|\lambda|>1$ expand
$\frac1{\mu-\lambda}=-\frac1\lambda-\frac\mu{\lambda^2}-\cdots$
uniformly for $\mu\in[-1,1]$:
\[
 Y(\lambda)=I-\frac1\lambda\int_{-1}^1Fg^{T}\,d\mu
 +O(\lambda^{-2}),
\]
which is (Y4) with \eqref{eq:Y1formula}.

We next prove (Y3) in the form stated above.  The
integral equation $(I-\gamma K_s)F=f$ gives
\[
 F(z)=f(z)+\gamma\int_{-1}^1K_s(z,\mu)F(\mu)\,d\mu .
\]
The right-hand side is entire in $z$, because $K_s(z,\mu)$ and $f(z)$
are entire and $F\in L^2(-1,1)^2$.  We henceforth take this entire
representative of $F$.  Thus $F(z)g^T(z)$ is analytic near both
endpoints, and its Cauchy transform is
$O(1+|\ln|\lambda\mp1||)$ there.

Set
\[
 N:=\begin{pmatrix}-1&1\\-1&1\end{pmatrix},
 \qquad N^2=0,
 \qquad
 X(\lambda):=Y(\lambda)e^{is\lambda\sigma_3}.
\]
A direct conjugation of \eqref{eq:Yjump} gives
\[
 X_+=X_-J_0,\qquad
 J_0:=e^{-is\lambda\sigma_3}G_Y(\lambda)
       e^{is\lambda\sigma_3}
     =I+\gamma N .
\]
Let
\[
 \ell(\lambda):=
 \operatorname{Log}\!\left(\frac{\lambda-1}{\lambda+1}\right),
 \qquad
 L(\lambda):=I+\frac{\gamma}{2\pi i}N\ell(\lambda),
\]
where the logarithm is principal.  On the interval oriented from
$-1$ to $1$,
\[
 \ell_+-\ell_-=2\pi i .
\]
Since $N^2=0$, this implies
\[
 L_+=L_-(I+\gamma N)=L_-J_0,
 \qquad
 L^{-1}=I-\frac{\gamma}{2\pi i}N\ell .
\]
Therefore
\[
 \check Y(\lambda):=X(\lambda)L(\lambda)^{-1}
\]
has no jump across $(-1,1)$.  Near either endpoint it is
$O((1+|\ln|\lambda\mp1||)^2)$, so its isolated singularity there is
removable.  Consequently
\[
 Y(\lambda)=\check Y(\lambda)
 \left[I+\frac{\gamma}{2\pi i}N
 \ln\!\left(\frac{\lambda-1}{\lambda+1}\right)\right]
 e^{-is\lambda\sigma_3},
\]
which is exactly \eqref{eq:Yendpoint}.  The invertibility of
$\check Y$ follows in part (d).

(d) Since
$\det G_Y=(1-\gamma)(1+\gamma)+\gamma^2=1$, the determinant of $Y$
has no jump.  The factorization in (c), together with
$\det L=1$ and $\det e^{-is\lambda\sigma_3}=1$, shows that $\det Y$
extends analytically through both endpoints.  It is therefore entire,
and (Y4) gives $\det Y\to1$ at infinity.  Liouville's theorem yields
$\det Y\equiv1$.  Hence $\det\check Y\equiv1$ locally at each endpoint,
which proves the invertibility required in (Y3).  Expanding at infinity
also gives
\[
 \det Y=1+(\Tr Y_1)\lambda^{-1}+O(\lambda^{-2}),
 \qquad \Tr Y_1=0 .
\]

(e) If $\widetilde Y$ is another solution, write the endpoint
factorizations from (c) as
\[
 Y=\check Y\,L e^{-is\lambda\sigma_3},
 \qquad
 \widetilde Y=\widetilde{\check Y}\,L e^{-is\lambda\sigma_3}.
\]
Then $\widetilde YY^{-1}$ has no jump and equals
$\widetilde{\check Y}\check Y^{-1}$ near either endpoint, so it extends
analytically there.  It tends to $I$ at infinity, and hence is
identically $I$ by Liouville's theorem.
\end{proof}

\begin{lemma}[the small-$\gamma$ calibration]\label{lem:smallgamma}
As $\gamma\to0$ with $s$ fixed,
\begin{equation}\label{eq:smallgamma}
 (Y_1)_{11}=-\frac{i\gamma}\pi+O(\gamma^2),
 \qquad\text{hence}\qquad
 -2i(Y_1)_{11}=-\frac{2\gamma}\pi+O(\gamma^2).
\end{equation}
This agrees with the determinant:
\begin{equation}\label{eq:smallgammadet}
 \ln\det(I-\gamma K_s)=-\gamma\Tr K_s+O(\gamma^2)
 =-\frac{2\gamma s}\pi+O(\gamma^2),
 \qquad
 \partial_s\ln\det(I-\gamma K_s)=-\frac{2\gamma}\pi+O(\gamma^2),
\end{equation}
and therefore \emph{confirms} the differential identity
\eqref{eq:diffid} at first order.
\end{lemma}

\begin{proof}
Since $f=O(\gamma)$ and $\gamma K_s=O(\gamma)$ in operator norm,
$F=(I-\gamma K_s)^{-1}f=f+O(\gamma^2)$ in $L^2$.  Hence by
\eqref{eq:Y1formula},
\begin{align*}
 (Y_1)_{11}&=-\int_{-1}^1f_1(\mu)g_1(\mu)\,d\mu+O(\gamma^2)
 =-\int_{-1}^1\frac{\gamma e^{is\mu}}{2\pi i}
 \bigl(-e^{-is\mu}\bigr)\,d\mu+O(\gamma^2)\\
 &=\frac{\gamma}{2\pi i}\cdot2+O(\gamma^2)
 =-\frac{i\gamma}\pi+O(\gamma^2),
\end{align*}
using $1/i=-i$.  Then $-2i(Y_1)_{11}=-2i(-i\gamma/\pi)=-2\gamma/\pi$.
For \eqref{eq:smallgammadet}, expand $\ln\det(I-\gamma K_s)
=\Tr\ln(I-\gamma K_s)=-\gamma\Tr K_s+O(\gamma^2)$ and use
$\Tr K_s=2s/\pi$ (Lemma~\ref{lem:solv}).
\end{proof}

\begin{remark}[calibration of the sign]\label{rem:calibration}
Lemma~\ref{lem:smallgamma} fixes the sign in \eqref{eq:IIKSY}.  The two
sides of \eqref{eq:diffid} are computed from different objects: the left
from the operator $K_s$, and the right from the solution of a
Riemann--Hilbert problem.  With the opposite sign one would obtain
$(Y_1)_{11}=+i\gamma/\pi$ and hence $-2i(Y_1)_{11}=+2\gamma/\pi$,
contradicting \eqref{eq:smallgammadet}.  Independently,
Lemma~\ref{lem:Y1} gives
$(Y_1)_{11}=-2\nu+O(s^{-1})=-2iv/\pi+O(s^{-1})$; for small $\gamma$,
$v=-\tfrac12\ln(1-\gamma)=\tfrac\gamma2+O(\gamma^2)$, and therefore
$-2iv/\pi=-i\gamma/\pi+O(\gamma^2)$.  Thus the two computations of
$(Y_1)_{11}$ agree and both exclude the opposite sign convention.
\end{remark}

\subsection{Factorization, lens geometry, and contour orientations}
\label{sec:gate3}

\begin{lemma}[exact factorization of $G_Y$]\label{lem:factorization}
For every $\lambda\in(-1,1)$,
\begin{equation}\label{eq:factorization}
 G_Y(\lambda)
 =\underbrace{\begin{pmatrix}1&0\\
   -\gamma e^{2s(\kappa-i\lambda)}&1\end{pmatrix}}_{=:S_L(\lambda)}
 \;e^{-2\kappa s\sigma_3}\;
 \underbrace{\begin{pmatrix}1&\gamma e^{2s(\kappa+i\lambda)}\\
   0&1\end{pmatrix}}_{=:S_U(\lambda)} ,
\end{equation}
with $e^{-2\kappa s\sigma_3}=e^{-2v\sigma_3}$.  Moreover, in the signed
regime $v=-\omega\le0$,
\begin{equation}\label{eq:lenscoef}
 \gamma e^{2v}=e^{2v}-1=e^{-2\omega}-1,
 \qquad\bigl|\gamma e^{2v}\bigr|\le1 .
\end{equation}
\end{lemma}

\begin{proof}
Since $2s\kappa=2v$, the middle factor is
$\operatorname{diag}(e^{-2v},e^{2v})$.  Multiplying,
\[
 S_L\,e^{-2v\sigma_3}
 =\begin{pmatrix}e^{-2v}&0\\-\gamma e^{-2is\lambda}&e^{2v}\end{pmatrix},
 \qquad
 \bigl(S_Le^{-2v\sigma_3}\bigr)S_U
 =\begin{pmatrix}e^{-2v}&\gamma e^{2is\lambda}\\
 -\gamma e^{-2is\lambda}&e^{2v}(1-\gamma^2)\end{pmatrix}.
\]
Now $e^{-2v}=1-\gamma$ and $1-\gamma^2=(1-\gamma)(1+\gamma)
=e^{-2v}(1+\gamma)$, so $e^{2v}(1-\gamma^2)=1+\gamma$; the off-diagonal
entries already agree with \eqref{eq:Yjump}.  Finally
$\gamma e^{2v}=(1-e^{-2v})e^{2v}=e^{2v}-1\in(-1,0]$ for $v\le0$.
\end{proof}

\paragraph{Compatibility of the lens with the model.}
The local variable at $\lambda=1$ will be $\zeta=2s(\lambda-1)$ with
$2s>0$, so
\begin{equation}\label{eq:argzeta}
 \arg\zeta=\arg(\lambda-1)\qquad\text{exactly.}
\end{equation}
The model of \S\ref{sec:model} has jumps precisely on the three rays
$\arg\zeta\in\{\tfrac\pi2,-\tfrac\pi2,\pi\}$.  Compatibility with these
rays requires the lens contours to leave $\lambda=1$ vertically, while
the interval approaches $\lambda=1$ along
$\arg(\lambda-1)=\pi$.  If instead the lens left the endpoint at
$\arg(\lambda-1)=\pm(\pi-\theta_0)$ with $\theta_0\ne\pi/2$, the ratio
$R$ of \S\ref{sec:gate10} would have jumps on
$\arg\zeta=\pm(\pi-\theta_0)$, where the parametrix has none, and no
jumps on $\arg\zeta=\pm\pi/2$, where the parametrix does.  It would
therefore not be analytic inside the endpoint disks.  We use the
following compatible geometry.

\begin{definition}[contours and orientations]\label{def:contours}
Fix $h:=\tfrac12$.  Let
\[
 \Sigma_{\mathrm{up}}
 :=[-1,\,-1+ih]\cup[-1+ih,\,1+ih]\cup[1+ih,\,1],
 \qquad
 \Sigma_{\mathrm{dn}}:=\overline{\Sigma_{\mathrm{up}}},
\]
each oriented \emph{from $-1$ to $1$}, and let
$\Omega_{\mathrm{up}}:=(-1,1)\times(0,h)$ and
$\Omega_{\mathrm{dn}}:=\overline{\Omega_{\mathrm{up}}}$ be the open
regions they bound together with $(-1,1)$.  Set
\[
 \Sigma_T:=[-1,1]\cup\Sigma_{\mathrm{up}}\cup\Sigma_{\mathrm{dn}},
\]
with $(-1,1)$ oriented from $-1$ to $1$.  On every component the $+$
side is the one on the left of the direction of travel.  The endpoint
disks are $D_{\pm1}:=\{|\lambda\mp1|<r_s\}$ with $r_s$ as in
\eqref{eq:rs}; their boundary circles are oriented \emph{clockwise}.
\end{definition}

\begin{lemma}[the local geometry matches the model]\label{lem:geometry}
For $s\ge5$, with Definition~\ref{def:contours} and
$\zeta=2s(\lambda-1)$:
\begin{enumerate}[label=\textup{(\roman*)},leftmargin=2.6em]
\item $\Sigma_{\mathrm{up}}\cap D_1=\{1+iy:0<y<r_s\}$, whose image is
the ray $\arg\zeta=\tfrac\pi2$, traversed \emph{towards} $\zeta=0$;
\item $\Sigma_{\mathrm{dn}}\cap D_1=\{1-iy:0<y<r_s\}$, whose image is
the ray $\arg\zeta=-\tfrac\pi2$, traversed \emph{towards} $\zeta=0$;
\item $(1-r_s,1)$ has image the ray $\arg\zeta=\pi$;
\item $\Omega_{\mathrm{up}}\cap D_1$ has image
$\{\tfrac\pi2<\arg\zeta<\pi\}$ and
$\Omega_{\mathrm{dn}}\cap D_1$ has image
$\{-\pi<\arg\zeta<-\tfrac\pi2\}$, while $D_1$ outside both lenses has
image $\{|\arg\zeta|<\tfrac\pi2\}$;
\item on
$(\Sigma_{\mathrm{up}}\cup\Sigma_{\mathrm{dn}})\setminus(D_1\cup D_{-1})$
one has $|\Im\lambda|\ge r_s$;
\item all junctions of $\Sigma_T$ with $\partial D_{\pm1}$ are
orthogonal, all corners of $\Sigma_{\mathrm{up}},\Sigma_{\mathrm{dn}}$
are right angles, the total length of $\Sigma_T$ is at most $10$, and
$\operatorname{dist}(\partial D_1,\partial D_{-1})\ge1$.
\end{enumerate}
\end{lemma}

\begin{proof}
(i)--(iii) On $\{1+iy\}$, $\lambda-1=iy$ with $y>0$, so
$\arg(\lambda-1)=\pi/2$; as $\lambda$ moves along the orientation
($-1\to1$) it moves \emph{downward} on this segment, so $y$ decreases
and $|\zeta|=2sy$ decreases: the ray is traversed towards $0$.  On
$\{1-iy\}$ the orientation $-1\to1$ moves \emph{upward}, so again $y$
decreases and the ray $\arg\zeta=-\pi/2$ is traversed towards $0$.  On
$(1-r_s,1)$, $\lambda-1$ is negative real, so $\arg\zeta=\pi$ by
\eqref{eq:zetabranch}.
(iv) is \eqref{eq:argzeta} together with
$\Omega_{\mathrm{up}}\cap D_1=\{\Im\lambda>0,\ \Re\lambda<1\}\cap D_1$.
(v) On the vertical segments $|\Im\lambda|$ ranges over $[r_s,h]$
outside the disks, and on the horizontal segments $|\Im\lambda|=h>r_s$.
(vi) The vertical segments meet the circles $|\lambda\mp1|=r_s$ at
$\pm1\pm ir_s$, where the circle's tangent is horizontal and the segment
vertical.  The length is $4h+4=6\le10$ plus the two circles, of total
length $4\pi r_s\le4\pi$.  Finally $|\lambda-\mu|\ge2-2r_s\ge1$ for
$\lambda\in\partial D_1$, $\mu\in\partial D_{-1}$, since $r_s\le1/2$.
\end{proof}

\begin{definition}[the lensed unknown]\label{def:S}
Put
\begin{equation}\label{eq:Sdef}
 T(\lambda)=\begin{cases}
 Y(\lambda)\,S_U(\lambda)^{-1},&\lambda\in\Omega_{\mathrm{up}},\\[2pt]
 Y(\lambda)\,S_L(\lambda),&\lambda\in\Omega_{\mathrm{dn}},\\[2pt]
 Y(\lambda),&\text{elsewhere,}
 \end{cases}
\end{equation}
with $S_L,S_U$ as in \eqref{eq:factorization}, which are entire.
\end{definition}

\begin{lemma}[jumps of $T$]\label{lem:Tjumps}
$T$ is analytic off $\Sigma_T$, satisfies
$T=I+Y_1\lambda^{-1}+O(\lambda^{-2})$ at infinity, has the endpoint
behaviour \eqref{eq:Yendpoint} right-multiplied by $S_U^{-1}$ in
$\Omega_{\mathrm{up}}$, by $S_L$ in $\Omega_{\mathrm{dn}}$ and by $I$
elsewhere, and $T_+=T_-J_T$ with
\begin{equation}\label{eq:Tjumps}
 J_T=
 \begin{cases}
 e^{-2v\sigma_3},&\lambda\in(-1,1),\\[2pt]
 S_U(\lambda),&\lambda\in\Sigma_{\mathrm{up}},\\[2pt]
 S_L(\lambda),&\lambda\in\Sigma_{\mathrm{dn}} .
 \end{cases}
\end{equation}
\end{lemma}

\begin{proof}
On $(-1,1)$, the $+$ side lies in $\Omega_{\mathrm{up}}$ and the $-$
side in $\Omega_{\mathrm{dn}}$, so
$T_+=Y_+S_U^{-1}=Y_-G_YS_U^{-1}=Y_-S_Le^{-2v\sigma_3}
=T_-e^{-2v\sigma_3}$ by \eqref{eq:factorization}.
On $\Sigma_{\mathrm{up}}$, oriented $-1\to1$, the $+$ side is the left,
which is the side \emph{away} from $\Omega_{\mathrm{up}}$: on the
horizontal segment the direction is $+1$ and the left is the upper side,
outside the rectangle; on the segment $[-1,-1+ih]$ the direction is
$+i$ and the left is $\{\Re\lambda<-1\}$, again outside; on
$[1+ih,1]$ the direction is $-i$ and the left is $\{\Re\lambda>1\}$,
again outside.  Hence $T_+=Y$ and $T_-=YS_U^{-1}$, so $T_+=T_-S_U$.
The same three checks on $\Sigma_{\mathrm{dn}}$ give $T_+=YS_L$ and
$T_-=Y$, so $T_+=T_-S_L$.  The behaviour at infinity is unchanged since
$\Sigma_{\mathrm{up}},\Sigma_{\mathrm{dn}}$ are compact.
\end{proof}

\subsection{The outer parametrix}
\label{sec:gate5}

\begin{definition}[outer parametrix]\label{def:Pinf}
For $\lambda\in\C\setminus[-1,1]$ put
\begin{equation}\label{eq:Pinf}
 P^{(\infty)}(\lambda)=
 \Bigl(\frac{\lambda-1}{\lambda+1}\Bigr)^{\nu\sigma_3},
\end{equation}
with the principal branch of $w\mapsto w^{\nu}$, $|\arg w|<\pi$, applied
to $w=\frac{\lambda-1}{\lambda+1}$.
\end{definition}

\begin{lemma}[properties of $P^{(\infty)}$]\label{lem:Pinf}
The map $\lambda\mapsto w=\frac{\lambda-1}{\lambda+1}$ is a Möbius
bijection of $\C\setminus[-1,1]$ onto $\C\setminus(-\infty,0]$, so
\eqref{eq:Pinf} is well defined and analytic there.  Moreover:
\begin{enumerate}[label=\textup{(\roman*)},leftmargin=2.6em]
\item $\det P^{(\infty)}\equiv1$;
\item on $(-1,1)$, oriented from $-1$ to $1$ with $+$ the upper side,
\begin{equation}\label{eq:Pinfjump}
 P^{(\infty)}_+=P^{(\infty)}_-\,e^{+2\pi i\nu\sigma_3}
 =P^{(\infty)}_-\,e^{-2v\sigma_3},
\end{equation}
\emph{the exponent being $+2\pi i\nu$};
\item $P^{(\infty)}(-\lambda)=\sigma_1P^{(\infty)}(\lambda)\sigma_1$;
\item as $\lambda\to\infty$,
 \begin{equation}\label{eq:Pinfinfty}
  P^{(\infty)}(\lambda)=I-\frac{2\nu}{\lambda}\sigma_3+O(\lambda^{-2});
 \end{equation}
\item $\|P^{(\infty)}(\lambda)^{\pm1}\|\le e^{\omega}$ for every
$\lambda\in\C\setminus[-1,1]$.
\end{enumerate}
\end{lemma}

\begin{proof}
The Möbius map sends $\infty\mapsto1$, $1\mapsto0$, $-1\mapsto\infty$
and $[-1,1]$ onto $(-\infty,0]$; being a bijection of the sphere it maps
the complement onto the complement.

(i) $\det w^{\nu\sigma_3}=w^\nu w^{-\nu}=1$.

(ii) Let $\lambda\in(-1,1)$.  Then $\lambda-1<0$ and $\lambda+1>0$, so
$w<0$.  \emph{Upper side:} for $\lambda+i\eps$ we have
$\arg(\lambda-1+i\eps)\to\pi$ (negative real part, positive imaginary
part) and $\arg(\lambda+1+i\eps)\to0$, so
$\arg w_+=\pi-0=\pi$.  \emph{Lower side:} for $\lambda-i\eps$,
$\arg(\lambda-1-i\eps)\to-\pi$ and $\arg(\lambda+1-i\eps)\to0$, so
$\arg w_-=-\pi$.  Therefore
$w_+^{\nu}=|w|^{\nu}e^{i\pi\nu}$ and $w_-^{\nu}=|w|^{\nu}e^{-i\pi\nu}$,
whence
\[
 \frac{w_+^{\nu}}{w_-^{\nu}}=e^{2\pi i\nu},
 \qquad
 P^{(\infty)}_+
 =\operatorname{diag}\bigl(w_+^{\nu},w_+^{-\nu}\bigr)
 =\operatorname{diag}\bigl(w_-^{\nu}e^{2\pi i\nu},
 w_-^{-\nu}e^{-2\pi i\nu}\bigr)
 =P^{(\infty)}_-e^{2\pi i\nu\sigma_3}.
\]
Finally, by \eqref{eq:params},
\begin{equation}\label{eq:twopinu}
 2\pi i\nu=2\pi i\cdot\frac{iv}\pi=2i^2v=-2v ,
\end{equation}
so $e^{2\pi i\nu\sigma_3}=e^{-2v\sigma_3}$, which is exactly the jump of
$T$ on $(-1,1)$ recorded in \eqref{eq:Tjumps}.

(iii) $\frac{-\lambda-1}{-\lambda+1}=w^{-1}$, and for
$w\notin(-\infty,0]$ also $w^{-1}\notin(-\infty,0]$ with
$\arg(w^{-1})=-\arg w$, so $(w^{-1})^{\nu}=w^{-\nu}$; conjugation by
$\sigma_1$ exchanges the diagonal entries.

(iv) $w=1-\frac2\lambda+O(\lambda^{-2})$ with $\arg w\to0$, so
$w^{\nu}=1-\frac{2\nu}\lambda+O(\lambda^{-2})$.

(v) $|w^{\pm\nu}|=\exp(\mp\Im\nu\arg w)
=\exp(\pm\tfrac\omega\pi\arg w)\le e^{\omega}$ since $\Re\nu=0$,
$\Im\nu=-\omega/\pi$ and $|\arg w|<\pi$.
\end{proof}

\begin{remark}[Jump exponent]\label{rem:corrB}
The upper and lower boundary arguments of
$w=(\lambda-1)/(\lambda+1)$ differ by $2\pi$, which gives the factor
$e^{+2\pi i\nu\sigma_3}$.  Since $2\pi i\nu=-2v$, this is precisely
$e^{-2v\sigma_3}$, the jump of $T$ on $(-1,1)$.
\end{remark}

\begin{lemma}[the lens jumps are negligible off the endpoint disks]
\label{lem:lensnegligible}
For every fixed $N\ge1$ there are $C_{A,N}$ and $s_{A,N}\ge5$ such that, for
$s\ge s_{A,N}$ and $0\le\omega\le A\ln s$,
\begin{equation}\label{eq:lensbound}
 \sup_{\lambda\in(\Sigma_{\mathrm{up}}\cup\Sigma_{\mathrm{dn}})
 \setminus(D_1\cup D_{-1})}
 \bigl\|P^{(\infty)}(\lambda)\bigl(J_T(\lambda)-I\bigr)
 \bigl(P^{(\infty)}(\lambda)\bigr)^{-1}\bigr\|
 \;\le\;C_{A,N}\,s^{-N},
\end{equation}
and the same bound holds in $L^2$ of that contour.
\end{lemma}

\begin{proof}
On $\Sigma_{\mathrm{up}}$ the only nonzero entry of $J_T-I=S_U-I$ is
$\gamma e^{2v}e^{2is\lambda}$, of modulus
$|\gamma e^{2v}|e^{-2s\Im\lambda}\le e^{-2s\Im\lambda}$ by
\eqref{eq:lenscoef}; by Lemma~\ref{lem:geometry}(v),
$\Im\lambda\ge r_s$ there, so this is at most $e^{-2sr_s}$.
Conjugating by $P^{(\infty)}$ multiplies the $(1,2)$ entry by
$w^{2\nu}$, of modulus at most $e^{2\omega}$ by
Lemma~\ref{lem:Pinf}(v).  Hence the left side of \eqref{eq:lensbound} is
at most $e^{2\omega-2sr_s}\le s^{2A}e^{-2s/(\ln s)^2}$, which is
$O_{A,N}(s^{-N})$ once $s/(\ln s)^2\ge\tfrac12(2A+N)\ln s$.  The $L^2$
bound follows since the contour has length at most $10$.  On
$\Sigma_{\mathrm{dn}}$ the same argument applies to the $(2,1)$ entry
$-\gamma e^{2v}e^{-2is\lambda}$, using $\Im\lambda\le-r_s$.
\end{proof}

\subsection{The confluent-hypergeometric model and its jumps}
\label{sec:model}

\begin{definition}[the bare model]\label{def:bareP}
For $\zeta\in\C\setminus\{0\}$ with $-\pi<\arg\zeta\le\pi$ set
\begin{equation}\label{eq:Fmatrix}
 \mathcal F(\zeta)=
 \begin{pmatrix}
 U(-\nu;e^{-\frac{i\pi}2}\zeta)
 & -ie^{i\pi\nu}\dfrac{\Gamma(1+\nu)}{\Gamma(-\nu)}\,
   U(1+\nu;e^{\frac{i\pi}2}\zeta)\\[2mm]
 ie^{i\pi\nu}\dfrac{\Gamma(1-\nu)}{\Gamma(\nu)}\,
   U(1-\nu;e^{-\frac{i\pi}2}\zeta)
 & e^{2\pi i\nu}U(\nu;e^{\frac{i\pi}2}\zeta)
 \end{pmatrix},
\end{equation}
\begin{equation}\label{eq:D0E0}
 D_0=e^{\frac{i\pi}2(\frac12-\nu)\sigma_3},
 \qquad
 E_0=\operatorname{diag}\bigl(e^{\frac{i\pi\nu}2},
 e^{-\frac{3i\pi\nu}2}\bigr),
 \qquad
 r_\nu=\frac{\Gamma(\nu)}{\Gamma(-\nu)},
\end{equation}
\begin{equation}\label{eq:sectors}
 \mathcal S(\zeta)=
 \begin{cases}
 T_+E_0,\ T_+=\begin{pmatrix}1&\tau_+\\0&1\end{pmatrix},
 &\Sigma_+:=\{\tfrac\pi2<\arg\zeta<\pi\},\\[6pt]
 T_-E_0,\ T_-=\begin{pmatrix}1&0\\\tau_-&1\end{pmatrix},
 &\Sigma_-:=\{-\pi<\arg\zeta<-\tfrac\pi2\},\\[6pt]
 E_0,&\Sigma_0:=\{|\arg\zeta|<\tfrac\pi2\},
 \end{cases}
\end{equation}
\begin{equation}\label{eq:taus}
 \tau_+=\frac{2\pi i\,e^{i\pi\nu}}{\Gamma(\nu)\Gamma(1-\nu)},
 \qquad
 \tau_-=\frac{2\pi i\,e^{-3\pi i\nu}}{\Gamma(\nu)\Gamma(1-\nu)},
\end{equation}
and
\begin{equation}\label{eq:bareP}
 \mathcal P(\zeta)
 =\mathcal F(\zeta)\,e^{\frac i2\zeta\sigma_3}\,D_0\,\mathcal S(\zeta).
\end{equation}
\end{definition}

Definition~\ref{def:bareP} is \cite[eq.~(2.2)]{BDIK2} written out, in
that paper's factor order and with its branch conventions
\eqref{eq:zetabranch}--\eqref{eq:Ubranch}.

\begin{lemma}[elementary identities for the model constants]
\label{lem:gammaratio}
For $\nu=-i\omega/\pi$ with $\omega>0$:
\begin{equation}\label{eq:gammaratio}
 |r_\nu|=1,
 \qquad
 \frac{\Gamma(1+\nu)}{\Gamma(-\nu)}=\nu\,r_\nu,
 \qquad
 \frac{\Gamma(1-\nu)}{\Gamma(\nu)}=-\nu\,r_\nu^{-1},
 \qquad
 \frac{2\pi i}{\Gamma(\nu)\Gamma(1-\nu)}=2i\sin(\pi\nu),
\end{equation}
\begin{equation}\label{eq:gammaprod}
 \Gamma(1+\nu)\Gamma(-\nu)=-\,\Gamma(\nu)\Gamma(1-\nu),
\end{equation}
and, with $\gamma=1-e^{-2v}$ and $2\pi i\nu=-2v$,
\begin{equation}\label{eq:tauvalues}
 \tau_+=e^{2\pi i\nu}-1=e^{-2v}-1=-\gamma,
 \qquad
 \tau_-=e^{-2\pi i\nu}-e^{-4\pi i\nu}=e^{2v}-e^{4v}=-\gamma e^{4v}.
\end{equation}
\end{lemma}

\begin{proof}
$\Gamma$ is real on the reals and analytic, so
$\overline{\Gamma(\nu)}=\Gamma(-\nu)$ for purely imaginary $\nu$, giving
$|r_\nu|=1$.  The recurrence $\Gamma(1+z)=z\Gamma(z)$ gives the two
ratios.  Reflection gives
$\Gamma(\nu)\Gamma(1-\nu)=\pi/\sin(\pi\nu)$, whence the fourth identity.
For \eqref{eq:gammaprod}, $\Gamma(1+\nu)=\nu\Gamma(\nu)$ and
$\Gamma(-\nu)=\Gamma(1-\nu)/(-\nu)$, so
$\Gamma(1+\nu)\Gamma(-\nu)=\nu\Gamma(\nu)\Gamma(1-\nu)/(-\nu)
=-\Gamma(\nu)\Gamma(1-\nu)$.
For \eqref{eq:tauvalues},
$\tau_+=2i\sin(\pi\nu)e^{i\pi\nu}
=(e^{i\pi\nu}-e^{-i\pi\nu})e^{i\pi\nu}=e^{2i\pi\nu}-1$, and
$e^{2\pi i\nu}=e^{-2v}=1-\gamma$; similarly
$\tau_-=(e^{i\pi\nu}-e^{-i\pi\nu})e^{-3i\pi\nu}
=e^{-2i\pi\nu}-e^{-4i\pi\nu}$ with $e^{-2\pi i\nu}=e^{2v}$.
\end{proof}

\begin{remark}[which sector matrix is triangular which way]
\label{rem:whichtriangular}
The assignment in \eqref{eq:sectors} --- upper triangular on $\Sigma_+$,
lower triangular on $\Sigma_-$ --- is not a free choice.  It is forced by
Lemma~\ref{lem:modeljumps} below: the opposite assignment would produce
a \emph{lower}-triangular jump on $\arg\zeta=\pi/2$, contradicting
\cite[RHP~2.3(2)]{BDIK2} and, more importantly, contradicting the fact
that $T$'s jump on $\Sigma_{\mathrm{up}}$ is upper triangular
(Lemma~\ref{lem:Tjumps}).
\end{remark}

We now derive the model's jumps on the two Stokes rays.  The derivation
uses only \eqref{eq:Ubranch} and Lemma~\ref{lem:gammaratio}.

\begin{lemma}[model jumps on the two Stokes rays; \textbf{derived here}]
\label{lem:modeljumps}
Orient each of the rays $\arg\zeta=\pm\pi/2$ \emph{towards} $\zeta=0$,
so that the $+$ side of $\arg\zeta=\pi/2$ is $\Sigma_0$ and the $+$ side
of $\arg\zeta=-\pi/2$ is $\Sigma_-$ \textup{(}see
Remark~\textup{\ref{rem:rayorientation}}\textup{)}.  Then
$\mathcal P_+=\mathcal P_-J_{\mathcal P}$ with
\begin{equation}\label{eq:modeljumps}
 J_{\mathcal P}=
 \begin{pmatrix}1&\gamma e^{2\kappa s}\\0&1\end{pmatrix}
 \quad\text{on }\arg\zeta=\tfrac\pi2,
 \qquad
 J_{\mathcal P}=
 \begin{pmatrix}1&0\\-\gamma e^{2\kappa s}&1\end{pmatrix}
 \quad\text{on }\arg\zeta=-\tfrac\pi2 .
\end{equation}
\end{lemma}

\begin{proof}
\emph{Step 1: $\mathcal F$ is analytic across both rays.}
Put $\theta=\arg\zeta$.  By \eqref{eq:Ubranch}, the arguments of the
four $U$-functions carry the lifted phases
$\theta-\pi/2$ and $\theta+\pi/2$, without reduction modulo $2\pi$.
These phases vary continuously as $\theta$ crosses either
$\pm\pi/2$; the continuation through a phase $\pm\pi$ is part of the
definition in \eqref{eq:Ubranch}.  Hence $\mathcal F$, and likewise
$e^{\frac i2\zeta\sigma_3}$ and $D_0$, are analytic across both rays.
The interval $(-3\pi/2,3\pi/2)$ is used here to specify the lifted
continuations and their asymptotics, not as an analyticity domain for
the principal $U$-function.  It follows from \eqref{eq:bareP} that the
entire jump of $\mathcal P$ on these two rays comes from $\mathcal S$:
\begin{equation}\label{eq:JfromS}
 J_{\mathcal P}
 =\mathcal P_-^{-1}\mathcal P_+
 =\bigl(\mathcal F e^{\frac i2\zeta\sigma_3}D_0\mathcal S_-\bigr)^{-1}
  \bigl(\mathcal F e^{\frac i2\zeta\sigma_3}D_0\mathcal S_+\bigr)
 =\mathcal S_-^{-1}\mathcal S_+ .
\end{equation}
The common analytic left factor cancels exactly; no conjugation by
$e^{\frac i2\zeta\sigma_3}$ survives, and in particular
$J_{\mathcal P}$ is \emph{constant} in $\zeta$.

\emph{Step 2: the ray $\arg\zeta=\pi/2$.}  With the stated orientation
$\mathcal S_+=\mathcal S|_{\Sigma_0}=E_0$ and
$\mathcal S_-=\mathcal S|_{\Sigma_+}=T_+E_0$, so by \eqref{eq:JfromS}
\[
 J_{\mathcal P}=E_0^{-1}T_+^{-1}E_0
 =\begin{pmatrix}1&-\tau_+\,(E_0)_{22}/(E_0)_{11}\\0&1\end{pmatrix}
 =\begin{pmatrix}1&-\tau_+e^{-2i\pi\nu}\\0&1\end{pmatrix},
\]
using $(E_0)_{22}/(E_0)_{11}
=e^{-\frac{3i\pi\nu}2-\frac{i\pi\nu}2}=e^{-2i\pi\nu}$.  By
\eqref{eq:tauvalues} and $e^{-2i\pi\nu}=e^{2v}=e^{2\kappa s}$,
\[
 -\tau_+e^{-2i\pi\nu}=\gamma e^{2v}=\gamma e^{2\kappa s},
\]
which is the first matrix in \eqref{eq:modeljumps}.

\emph{Step 3: the ray $\arg\zeta=-\pi/2$.}  With the stated orientation
$\mathcal S_+=T_-E_0$ and $\mathcal S_-=E_0$, so
\[
 J_{\mathcal P}=E_0^{-1}T_-E_0
 =\begin{pmatrix}1&0\\\tau_-\,(E_0)_{11}/(E_0)_{22}&1\end{pmatrix}
 =\begin{pmatrix}1&0\\\tau_-e^{2i\pi\nu}&1\end{pmatrix},
\]
and by \eqref{eq:tauvalues} with $e^{2i\pi\nu}=e^{-2v}$,
\[
 \tau_-e^{2i\pi\nu}=-\gamma e^{4v}e^{-2v}=-\gamma e^{2v}
 =-\gamma e^{2\kappa s},
\]
which is the second matrix in \eqref{eq:modeljumps}.
\end{proof}

\begin{remark}[orientation of the model rays]
\label{rem:rayorientation}
\cite[RHP~2.3(1)]{BDIK2} specifies the orientation of the three model
rays by reference to a figure.  The orientations can also be determined
algebraically.  Step 2 of the proof above shows that orienting
$\arg\zeta=\pi/2$ \emph{towards} the origin reproduces
$\bigl(\begin{smallmatrix}1&\gamma e^{2\kappa s}\\0&1
\end{smallmatrix}\bigr)$, whereas orienting it away from the origin
gives the inverse
$\bigl(\begin{smallmatrix}1&-\gamma e^{2\kappa s}\\0&1
\end{smallmatrix}\bigr)$.  The inward orientation agrees both with
\cite[RHP~2.3(2)]{BDIK2} and with the jump $S_U$ of $T$ that the
parametrix must reproduce (Lemma~\ref{lem:localjumps}).  The same
calculation on $\arg\zeta=-\pi/2$ again selects the inward orientation.
Thus all three model rays are oriented towards $\zeta=0$.
\end{remark}

The third ray and the origin behaviour require a separate local
calculation, which we now give.

Before giving the derivation we must introduce a second family of
triangular matrices.  \emph{These are not the sector matrices $T_\pm$ of
\eqref{eq:sectors}}, and the distinction is essential; it is the subject
of Remark~\ref{rem:originvssector}.

\begin{definition}[the origin-sector multipliers]\label{def:originsectors}
Put
\begin{equation}\label{eq:originsectors}
 \mathcal H_+:=\begin{pmatrix}1&-\gamma e^{2\kappa s}\\0&1\end{pmatrix},
 \qquad
 \mathcal H_-:=\begin{pmatrix}1&0\\-\gamma e^{2\kappa s}&1\end{pmatrix},
\end{equation}
where, by \eqref{eq:params}, $\kappa s=v$, so that
$-\gamma e^{2\kappa s}=-\gamma e^{2v}$.
\end{definition}

\begin{lemma}[model jump on the third ray and behaviour at the origin;
\textbf{derived here}]\label{qmi:third}
With $\arg\zeta=\pi$ oriented towards $\zeta=0$ \textup{(}so that its
$+$ side is $\Sigma_+$\textup{)},
\begin{equation}\label{eq:thirdray}
 \mathcal P_+=\mathcal P_-\,e^{-2\kappa s\sigma_3}
 \qquad\text{on }\arg\zeta=\pi,
\end{equation}
and, as $\zeta\to0$ with $-\pi<\arg\zeta\le\pi$,
\begin{equation}\label{eq:modelorigin}
 \mathcal P(\zeta)=\check{\mathcal P}(\zeta)
 \Bigl[I+\frac{\gamma}{2\pi i}
 \begin{pmatrix}-1&1\\-1&1\end{pmatrix}\ln\zeta\Bigr]
 \times\begin{cases}
 \mathcal H_+,&\arg\zeta\in(\tfrac\pi2,\pi),\\
 \mathcal H_-,&\arg\zeta\in(-\pi,-\tfrac\pi2),\\
 I,&|\arg\zeta|<\tfrac\pi2,
 \end{cases}
\end{equation}
with $\check{\mathcal P}$ analytic and invertible at $\zeta=0$.
\end{lemma}

\begin{proof}
The proof starts in the central sector, where all four Kummer functions
are on their principal local sheets, and then uses the two Stokes jumps
already proved in Lemma~\ref{lem:modeljumps}.  This order avoids any
unproved continuation across the negative real axis.

\emph{Step 1: the exact logarithmic decomposition at the origin.}
For $a\notin\{0,-1,-2,\ldots\}$, the case $n=0$ of
\cite[eq.~13.2.9]{NIST} gives, in a neighbourhood of the origin,
\begin{equation}\label{eq:Uoriginexact}
 U(a,1;w)=Q_a(w)-\frac{M(a,1;w)}{\Gamma(a)}\ln w,
\end{equation}
where
\[
 Q_a(w)=-\frac1{\Gamma(a)}
 \sum_{k=0}^{\infty}\frac{(a)_k}{(k!)^2}w^k
 \bigl(\psi(a+k)-2\psi(k+1)\bigr)
\]
is analytic at $w=0$.  All four parameters
$-\nu,1+\nu,1-\nu,\nu$ satisfy the stated restriction because
$\nu\in i\mathbb R\setminus\{0\}$.

Put $\mathbf e=(1,1)^{\mathsf T}$,
$\mathbf r^{\mathsf T}=(-1,1)$, so that
$N:=\mathbf e\mathbf r^{\mathsf T}
=\bigl(\begin{smallmatrix}-1&1\\-1&1\end{smallmatrix}\bigr)$ and
$N^2=0$.  In $\Sigma_0$,
$\ln(-i\zeta)=\ln\zeta-i\pi/2$ and
$\ln(i\zeta)=\ln\zeta+i\pi/2$.  Moreover
\begin{equation}\label{eq:D0E0simple}
 e^{\frac i2\zeta\sigma_3}D_0E_0
 =\operatorname{diag}\bigl(e^{i\pi/4}e^{i\zeta/2},
 e^{-i\pi/4-i\pi\nu}e^{-i\zeta/2}\bigr).
\end{equation}
Substitution of \eqref{eq:Uoriginexact} into
\eqref{eq:Fmatrix}--\eqref{eq:bareP}, followed by Kummer's
transformation \cite[eq.~13.2.39]{NIST}, gives
\begin{equation}\label{eq:PACdecomp}
 \mathcal P(\zeta)=A(\zeta)+C(\zeta)\ln\zeta,
 \qquad C(\zeta)=\mathbf c(\zeta)\mathbf r^{\mathsf T},
\end{equation}
where $A$ is analytic at $0$ and
\begin{equation}\label{eq:cvector}
 \mathbf c(\zeta)=
 \begin{pmatrix}
 \displaystyle
 \frac{e^{i\pi/4}e^{-i\zeta/2}}{\Gamma(-\nu)}
 M(1+\nu,1;i\zeta)\\[3mm]
 \displaystyle
 -\frac{e^{-i\pi/4+i\pi\nu}e^{i\zeta/2}}{\Gamma(\nu)}
 M(1-\nu,1;-i\zeta)
 \end{pmatrix}.
\end{equation}
For completeness, the first row of $C$ is
$c_1(-1,1)$ because
$M(-\nu,1;-i\zeta)=e^{-i\zeta}M(1+\nu,1;i\zeta)$; the second is
$c_2(-1,1)$ because
$M(\nu,1;i\zeta)=e^{i\zeta}M(1-\nu,1;-i\zeta)$.  Thus no entrywise
cancellation is being suppressed in \eqref{eq:PACdecomp}.

\emph{Step 2: identification of the logarithmic matrix.}
Since $\mathbf r^{\mathsf T}\mathbf e=0$,
\eqref{eq:PACdecomp} gives $\mathcal P\mathbf e=A\mathbf e$.
We now compute this analytic column.  The lower-sign instance of
\cite[eq.~13.2.41]{NIST}, with $a=1+\nu$, $b=1$ and $z=i\zeta$,
and the upper-sign instance with $a=1-\nu$, $b=1$ and
$z=-i\zeta$, respectively give
\begin{align}
 U(-\nu,1;-i\zeta)
 &-\frac{\Gamma(1+\nu)}{\Gamma(-\nu)}e^{-i\zeta}
 U(1+\nu,1;i\zeta)\notag\\
 &=e^{-i\pi\nu}\Gamma(1+\nu)e^{-i\zeta}
 M(1+\nu,1;i\zeta),                                      \label{eq:connrow1}\\
 -\frac{\Gamma(1-\nu)}{\Gamma(\nu)}U(1-\nu,1;-i\zeta)
 &+e^{-i\zeta}U(\nu,1;i\zeta)\notag\\
 &=e^{-i\pi\nu}\Gamma(1-\nu)
 M(1-\nu,1;-i\zeta).                                    \label{eq:connrow2}
\end{align}
The sign choices are forced by the branches: in
\eqref{eq:connrow1}, $e^{-\pi i}(i\zeta)=-i\zeta$ reaches the
principal value used in \eqref{eq:Fmatrix}; in
\eqref{eq:connrow2}, $e^{\pi i}(-i\zeta)=i\zeta$ does so.
Multiplying \eqref{eq:connrow1}--\eqref{eq:connrow2} by the diagonal
factors in \eqref{eq:D0E0simple} yields
\begin{equation}\label{eq:Aec}
 A(\zeta)\mathbf e=\frac{2\pi i}{\gamma}\,\mathbf c(\zeta).
\end{equation}
Indeed, the two scalar constants reduce to the same number because
\begin{equation}\label{eq:gammaconstorigin}
 e^{-i\pi\nu}\Gamma(1+\nu)\Gamma(-\nu)
 =-e^{-i\pi\nu}\Gamma(\nu)\Gamma(1-\nu)
 =\frac{2\pi i}{\gamma},
\end{equation}
using reflection and
$\gamma=1-e^{2\pi i\nu}=-2ie^{i\pi\nu}\sin(\pi\nu)$.
Equations \eqref{eq:PACdecomp} and \eqref{eq:Aec} therefore imply
\[
 C=\frac{\gamma}{2\pi i}A\mathbf e\mathbf r^{\mathsf T}
   =\frac{\gamma}{2\pi i}AN,
\]
and hence, throughout $\Sigma_0$,
\begin{equation}\label{eq:origincentralproved}
 \mathcal P(\zeta)=A(\zeta)
 \left[I+\frac{\gamma}{2\pi i}N\ln\zeta\right].
\end{equation}

\emph{Step 3: invertibility of the analytic prefactor.}
It remains to establish that $A(0)$ is invertible.
Write $h_a=\psi(a)-2\psi(1)$.  The $k=0$ term of
\eqref{eq:Uoriginexact} and \eqref{eq:D0E0simple} give
\begin{equation}\label{eq:A0explicit}
 A(0)=
 \begin{pmatrix}
 \dfrac{e^{i\pi/4}}{\Gamma(-\nu)}
       (-h_{-\nu}+i\pi/2)&
 \dfrac{e^{i\pi/4}}{\Gamma(-\nu)}
       (h_{1+\nu}+i\pi/2)\\[3mm]
 \dfrac{e^{-i\pi/4+i\pi\nu}}{\Gamma(\nu)}
       (h_{1-\nu}-i\pi/2)&
 -\dfrac{e^{-i\pi/4+i\pi\nu}}{\Gamma(\nu)}
       (h_{\nu}+i\pi/2)
 \end{pmatrix}.
\end{equation}
Set $X=h_{-\nu}-i\pi/2$ and $Y=h_\nu+i\pi/2$.  The recurrence
$\psi(1+z)=\psi(z)+z^{-1}$ reduces the determinant of the bracketed
part of \eqref{eq:A0explicit} to
\[
 XY-(Y+\nu^{-1})(X-\nu^{-1})
 =\frac{Y-X}{\nu}+\frac1{\nu^2}.
\]
Reflection,
$\psi(1-\nu)-\psi(\nu)=\pi\cot(\pi\nu)$, gives
$Y-X=-\pi\cot(\pi\nu)-\nu^{-1}+i\pi$, while
$\Gamma(\nu)\Gamma(-\nu)=-\pi/(\nu\sin\pi\nu)$.  Consequently
\begin{align*}
 \det A(0)
 &=\frac{e^{i\pi\nu}}{\Gamma(-\nu)\Gamma(\nu)}
   \frac{\pi\bigl(i-\cot(\pi\nu)\bigr)}{\nu}\\
 &=-e^{i\pi\nu}\sin(\pi\nu)
   \bigl(i-\cot(\pi\nu)\bigr)=1.
\end{align*}
Thus $A$ is analytic and invertible in a neighbourhood of the origin.

\emph{Step 4: the two noncentral sectors.}
On $\arg\zeta=\pi/2$, oriented inward, the plus side is $\Sigma_0$.
Lemma~\ref{lem:modeljumps} therefore says
\[
 \mathcal P\big|_{\Sigma_0}
 =\mathcal P\big|_{\Sigma_+}
 \begin{pmatrix}1&\gamma e^{2\kappa s}\\0&1\end{pmatrix},
\]
so $\mathcal P|_{\Sigma_+}=\mathcal P|_{\Sigma_0}\mathcal H_+$.
On $\arg\zeta=-\pi/2$ the plus side is $\Sigma_-$, and the same lemma
gives $\mathcal P|_{\Sigma_-}=\mathcal P|_{\Sigma_0}\mathcal H_-$.
With $\check{\mathcal P}:=A$, equation
\eqref{eq:origincentralproved} therefore becomes exactly
\eqref{eq:modelorigin} in all three sectors.

\emph{Step 5: the third ray.}
Let $\zeta=-r$, $r>0$, and take the principal boundary values
$\ln\zeta_\pm=\ln r\pm i\pi$.  Since $N^2=0$,
\[
 \left[I+\frac\gamma{2\pi i}N(\ln r-i\pi)\right]^{-1}
 \left[I+\frac\gamma{2\pi i}N(\ln r+i\pi)\right]
 =I+\gamma N.
\]
The plus side of the inward-oriented ray is $\Sigma_+$, hence
\begin{align*}
 \mathcal P_-^{-1}\mathcal P_+
 &=\mathcal H_-^{-1}(I+\gamma N)\mathcal H_+\\
 &=\begin{pmatrix}e^{-2v}&0\\0&e^{2v}\end{pmatrix}
 =e^{-2\kappa s\sigma_3},
\end{align*}
where the middle equality follows from
$1-\gamma=e^{-2v}$ and $\gamma e^{2v}=e^{2v}-1$.
This is \eqref{eq:thirdray} and completes the proof.
\end{proof}

\begin{remark}[origin multipliers and bare-model sector matrices]
\label{rem:originvssector}
Two distinct families of unitriangular matrices occur in the
description of $\mathcal P$, and they must be kept separate.

\emph{(i) What each family is.}  The matrices $T_\pm$ of
\eqref{eq:sectors}, with entries $\tau_+=-\gamma$ and
$\tau_-=-\gamma e^{4v}$ from \eqref{eq:tauvalues}, occur \emph{inside
the defining formula} \eqref{eq:bareP} for the bare model, to the right
of the diagonal dressing $e^{\frac i2\zeta\sigma_3}D_0$; they are what
Lemma~\ref{lem:modeljumps} and Lemma~\ref{lem:stokes} use.  The matrices
$\mathcal H_\pm$ of \eqref{eq:originsectors} occur in the \emph{local
expansion at the origin}, to the right of the logarithmic bracket, in
\cite[RHP~2.3(3)]{BDIK2}.  They are different matrices and they are not
interchangeable.

\emph{(ii) How they are related.}  With
\begin{equation}\label{eq:Dv}
 D_v:=e^{v\sigma_3}=\operatorname{diag}(e^{v},e^{-v}),
\end{equation}
one has \emph{exactly}
\begin{equation}\label{eq:HtoT}
 T_+=D_v^{-1}\mathcal H_+D_v,
 \qquad
 T_-=D_v^{-1}\mathcal H_-D_v .
\end{equation}
Indeed, conjugation by a diagonal $D$ multiplies a $(1,2)$ entry by
$D_{22}/D_{11}$ and a $(2,1)$ entry by $D_{11}/D_{22}$; here these are
$e^{-2v}$ and $e^{2v}$, so $(-\gamma e^{2v})e^{-2v}=-\gamma=\tau_+$ and
$(-\gamma e^{2v})e^{2v}=-\gamma e^{4v}=\tau_-$, by
\eqref{eq:tauvalues}.  What separates the two families is therefore
precisely the diagonal factor $D_v$ --- which is exactly the part of the
dressing in \eqref{eq:bareP} that is absent from the origin expansion.

\emph{(iii) Compatibility with the third-ray jump.}
The origin formula and the third-ray jump are not independent.  The
last step of Lemma~\ref{qmi:third} proves their compatibility by the
exact identity below.  
Write $N=\bigl(\begin{smallmatrix}-1&1\\-1&1\end{smallmatrix}\bigr)$, so
that $N^2=0$.  Crossing $\arg\zeta=\pi$ replaces $\ln\zeta$ by
$\ln\zeta-2\pi i$ and changes the sector factor from $\mathcal H_+$ to
$\mathcal H_-$; since $\check{\mathcal P}$ is analytic there,
\eqref{eq:modelorigin} \emph{forces} the jump on that ray to be
\begin{equation}\label{eq:monodromycheck}
 \mathcal H_-^{-1}\bigl(I+\gamma N\bigr)\mathcal H_+
 =\begin{pmatrix}1&0\\ \gamma e^{2v}&1\end{pmatrix}
 \begin{pmatrix}1-\gamma&\gamma\\-\gamma&1+\gamma\end{pmatrix}
 \begin{pmatrix}1&-\gamma e^{2v}\\0&1\end{pmatrix}
 =\begin{pmatrix}e^{-2v}&0\\0&e^{2v}\end{pmatrix}
 =e^{-2\kappa s\sigma_3},
\end{equation}
using $1-\gamma=e^{-2v}$ and $\gamma e^{2v}=e^{2v}-1$.  This is
\eqref{eq:thirdray}.  The same computation with $T_\pm$ in place of
$\mathcal H_\pm$ does not satisfy the required identity, confirming
that the multipliers used in Lemma~\ref{qmi:third} are the unique ones
compatible with \eqref{eq:monodromycheck}.
\end{remark}

\begin{remark}[derivation of the third-ray jump]
\label{rem:obstruction}
A direct comparison of the four continued $U$-functions on
$\arg\zeta=\pi$ is possible through the full monodromy formula
\cite[eq.~13.2.12]{NIST}.  The proof of Lemma~\ref{qmi:third} instead
derives the exact logarithmic factor in
the central sector from \cite[eq.~13.2.9]{NIST}, identifies its matrix
coefficient with the branch-sensitive connection formula
\cite[eq.~13.2.41]{NIST}, and then transports it through the two Stokes
jumps already proved in Lemma~\ref{lem:modeljumps}.  The third jump is
then forced by the two principal boundary values of $\ln\zeta$.

This route proves at the same time the origin-sector multipliers,
analyticity and invertibility of $\check{\mathcal P}$, and the
third-ray jump.  Every identity used is valid for
$\nu\in i\mathbb R\setminus\{0\}$, with no restriction on the sign of
$\Im\nu$; hence the signed case $v=-\omega<0$ requires no analytic
continuation from the head-side range.
\end{remark}

\subsection{Stokes cancellation in both sectors}
\label{sec:stokes}

Write
\begin{equation}\label{eq:Qdef}
 \mathcal Q(\zeta):=\mathcal P(\zeta)\,
 \bigl[\zeta^{\nu\sigma_3}e^{\frac i2\zeta\sigma_3}D_0\bigr]^{-1},
\end{equation}
the object whose large-$\zeta$ expansion is
\eqref{eq:P25} below.  In $\Sigma_0$, where $\mathcal S=E_0$,
\begin{equation}\label{eq:Qzero}
 \mathcal Q\big|_{\Sigma_0}
 =\mathcal F\,e^{\frac i2\zeta\sigma_3}D_0E_0D_0^{-1}
 e^{-\frac i2\zeta\sigma_3}\zeta^{-\nu\sigma_3}
 =\mathcal F\,E_0\,\zeta^{-\nu\sigma_3},
\end{equation}
all four factors being diagonal and hence commuting.  Every entry is a
constant multiple of a single $U$-function, and $\arg(e^{\mp i\pi/2}\zeta)
\in(-\pi,\pi)$ there; no exponential is present.  In $\Sigma_\pm$ this
is no longer so, and the following lemma is the reason the estimates of
\S\ref{sec:gate7} are legitimate.

\begin{lemma}[exact Stokes cancellation]\label{lem:stokes}
\emph{(i) In $\Sigma_+$.}  Moving $T_+$ through the diagonal factor,
\begin{equation}\label{eq:Qplus}
 \mathcal Q\big|_{\Sigma_+}
 =\mathcal F\begin{pmatrix}1&\mathsf T\\0&1\end{pmatrix}E_0
 \zeta^{-\nu\sigma_3},
 \qquad
 \mathsf T:=\tau_+e^{i\pi(\frac12-\nu)}e^{i\zeta},
\end{equation}
so that the second column of $\mathcal Q|_{\Sigma_+}$ involves
$\mathsf T\,\mathcal F_{j1}+\mathcal F_{j2}$, $j=1,2$.  The constant in
$\mathsf T$ is
\begin{equation}\label{eq:Tconst}
 \tau_+e^{i\pi(\frac12-\nu)}
 =-\frac{2\pi}{\Gamma(\nu)\Gamma(1-\nu)}
 =+\frac{2\pi}{\Gamma(1+\nu)\Gamma(-\nu)}
 =-2\sin(\pi\nu).
\end{equation}
Applying \cite[eq.~13.2.44]{NIST} with $b=1$ to the continued values
$U(1+\nu;e^{i\pi/2}\zeta)$ and $U(\nu;e^{i\pi/2}\zeta)$ produces, in each
of $\mathcal F_{12}$ and $\mathcal F_{22}$, exactly one term
proportional to $e^{i\zeta}$ times $\mathcal F_{11}$ resp.\
$\mathcal F_{21}$, with coefficient
\begin{equation}\label{eq:Dconst}
 -\frac{2\pi}{\Gamma(1+\nu)\Gamma(-\nu)}=+2\sin(\pi\nu),
\end{equation}
and therefore
\begin{equation}\label{eq:cancelplus}
 \underbrace{\bigl(-2\sin(\pi\nu)\bigr)}_{\text{from }\mathsf T}
 +\underbrace{\bigl(+2\sin(\pi\nu)\bigr)}_{\text{from 13.2.44}}=0 .
\end{equation}
\emph{(ii) In $\Sigma_-$.}  Likewise
\begin{equation}\label{eq:Qminus}
 \mathcal Q\big|_{\Sigma_-}
 =\mathcal F\begin{pmatrix}1&0\\\mathsf Y&1\end{pmatrix}E_0
 \zeta^{-\nu\sigma_3},
 \qquad
 \mathsf Y:=\tau_-e^{-i\pi(\frac12-\nu)}e^{-i\zeta},
 \qquad
 \tau_-e^{-i\pi(\frac12-\nu)}=2e^{-2\pi i\nu}\sin(\pi\nu),
\end{equation}
the first column involves $\mathcal F_{j1}+\mathsf Y\mathcal F_{j2}$,
and the minus-continuation of \cite[eq.~13.2.44]{NIST} produces in each
of $\mathcal F_{11},\mathcal F_{21}$ exactly one term proportional to
$e^{-i\zeta}$ times $\mathcal F_{12}$ resp.\ $\mathcal F_{22}$, with
coefficient $-2e^{-2\pi i\nu}\sin(\pi\nu)$, so that again the two
cancel exactly.
\emph{Consequently, in each of the three sectors every entry of
$\mathcal Q$ is a constant multiple of a single $U$-function evaluated
at a point of phase at most $\pi$ in modulus, and no exponential factor
survives.}
\end{lemma}

\begin{proof}
\emph{(i).}  For a diagonal $\Theta$ one has
$\Theta\bigl(\begin{smallmatrix}1&\tau\\0&1\end{smallmatrix}\bigr)
\Theta^{-1}
=\bigl(\begin{smallmatrix}1&\tau\Theta_{11}/\Theta_{22}\\0&1
\end{smallmatrix}\bigr)$; with $\Theta=e^{\frac i2\zeta\sigma_3}D_0$,
$\Theta_{11}/\Theta_{22}=e^{i\zeta}e^{i\pi(\frac12-\nu)}$, which gives
\eqref{eq:Qplus}.  For \eqref{eq:Tconst},
\[
 \tau_+e^{i\pi(\frac12-\nu)}
 =\frac{2\pi ie^{i\pi\nu}}{\Gamma(\nu)\Gamma(1-\nu)}
 \cdot e^{\frac{i\pi}2}e^{-i\pi\nu}
 =\frac{2\pi i\cdot i}{\Gamma(\nu)\Gamma(1-\nu)}
 =-\frac{2\pi}{\Gamma(\nu)\Gamma(1-\nu)}
 =-2\sin(\pi\nu),
\]
and the middle equality of \eqref{eq:Tconst} is \eqref{eq:gammaprod}.

Now apply \cite[eq.~13.2.44]{NIST}, which for $b=1$ reads
\begin{equation}\label{eq:acrosscut}
 e^{a\pi i}U(a,1;e^{\pi i}z)-e^{-a\pi i}U(a,1;e^{-\pi i}z)
 =\frac{2\pi i\,e^{-z}}{\Gamma(a)^2}\,U(1-a,1;z),
\end{equation}
solved for the value on the upper side of the cut:
\begin{equation}\label{eq:acrosscut2}
 U(a,1;e^{\pi i}z)
 =e^{-2a\pi i}U(a,1;e^{-\pi i}z)
 +\frac{2\pi i\,e^{-a\pi i}e^{-z}}{\Gamma(a)^2}\,U(1-a,1;z).
\end{equation}
Take $z=e^{-i\pi/2}\zeta$, so that $e^{\pi i}z=e^{i\pi/2}\zeta$,
$e^{-z}=e^{i\zeta}$, and, for $\arg\zeta\in(\tfrac\pi2,\pi)$,
$\arg z\in(0,\tfrac\pi2)$.

\emph{Entry $(1,2)$.}  Here $a=1+\nu$, so $1-a=-\nu$ and
$U(1-a,1;z)=U(-\nu;e^{-i\pi/2}\zeta)=\mathcal F_{11}$.  The
$e^{i\zeta}$-term of $\mathcal F_{12}=-ie^{i\pi\nu}
\frac{\Gamma(1+\nu)}{\Gamma(-\nu)}U(1+\nu;e^{i\pi/2}\zeta)$ is therefore
\begin{align*}
 -ie^{i\pi\nu}\frac{\Gamma(1+\nu)}{\Gamma(-\nu)}
 \cdot\frac{2\pi i\,e^{-(1+\nu)\pi i}}{\Gamma(1+\nu)^2}
 \;e^{i\zeta}\mathcal F_{11}
 &=(-i)(2\pi i)\frac{e^{i\pi\nu}e^{-i\pi}e^{-i\pi\nu}}
 {\Gamma(-\nu)\Gamma(1+\nu)}\;e^{i\zeta}\mathcal F_{11}\\
 &=-\frac{2\pi}{\Gamma(1+\nu)\Gamma(-\nu)}
 \;e^{i\zeta}\mathcal F_{11},
\end{align*}
using $(-i)(2\pi i)=2\pi$ and $e^{-i\pi}=-1$.  This is
\eqref{eq:Dconst}, and adding the $\mathsf T$-term
$-2\sin(\pi\nu)e^{i\zeta}\mathcal F_{11}$ gives \eqref{eq:cancelplus}.

\emph{Entry $(2,2)$.}  Here $a=\nu$, $1-a=1-\nu$, and
$U(1-\nu;e^{-i\pi/2}\zeta)=\mathcal F_{21}/\bigl(ie^{i\pi\nu}
\frac{\Gamma(1-\nu)}{\Gamma(\nu)}\bigr)$.  The $e^{i\zeta}$-term of
$\mathcal F_{22}=e^{2\pi i\nu}U(\nu;e^{i\pi/2}\zeta)$ is
\[
 e^{2\pi i\nu}\cdot\frac{2\pi i\,e^{-\nu\pi i}}{\Gamma(\nu)^2}
 \cdot\frac{\Gamma(\nu)}{i\,e^{i\pi\nu}\Gamma(1-\nu)}
 \;e^{i\zeta}\mathcal F_{21}
 =\frac{2\pi}{\Gamma(\nu)\Gamma(1-\nu)}\;e^{i\zeta}\mathcal F_{21}
 =2\sin(\pi\nu)\,e^{i\zeta}\mathcal F_{21},
\]
and the $\mathsf T$-term is $-2\sin(\pi\nu)e^{i\zeta}\mathcal F_{21}$;
they cancel.

\emph{(ii).}  Now $\Theta_{22}/\Theta_{11}=e^{-i\zeta}
e^{-i\pi(\frac12-\nu)}$ gives \eqref{eq:Qminus}, and
\[
 \tau_-e^{-i\pi(\frac12-\nu)}
 =\frac{2\pi ie^{-3\pi i\nu}}{\Gamma(\nu)\Gamma(1-\nu)}
 e^{-\frac{i\pi}2}e^{i\pi\nu}
 =\frac{2\pi i(-i)e^{-2\pi i\nu}}{\Gamma(\nu)\Gamma(1-\nu)}
 =2e^{-2\pi i\nu}\sin(\pi\nu).
\]
For $\arg\zeta\in(-\pi,-\tfrac\pi2)$ it is the \emph{first} column that
is continued: $\arg(e^{-i\pi/2}\zeta)\in(-\tfrac{3\pi}2,-\pi)$.  Solving
\eqref{eq:acrosscut} for the value on the lower side,
\[
 U(a,1;e^{-\pi i}w)=e^{2a\pi i}U(a,1;e^{\pi i}w)
 -\frac{2\pi i\,e^{a\pi i}e^{-w}}{\Gamma(a)^2}U(1-a,1;w),
\]
and take $w=e^{i\pi/2}\zeta$, so $e^{-\pi i}w=e^{-i\pi/2}\zeta$,
$e^{-w}=e^{-i\zeta}$, $\arg w\in(-\tfrac\pi2,0)$.
\emph{Entry $(2,1)$:} $a=1-\nu$, $1-a=\nu$,
$U(\nu;e^{i\pi/2}\zeta)=e^{-2\pi i\nu}\mathcal F_{22}$; the
$e^{-i\zeta}$-term of $\mathcal F_{21}$ is
\begin{align*}
 ie^{i\pi\nu}\frac{\Gamma(1-\nu)}{\Gamma(\nu)}\cdot
 \Bigl(-\frac{2\pi i\,e^{(1-\nu)\pi i}}{\Gamma(1-\nu)^2}\Bigr)
 e^{-2\pi i\nu}\;e^{-i\zeta}\mathcal F_{22}
 &=-\frac{2\pi e^{-2\pi i\nu}}{\Gamma(\nu)\Gamma(1-\nu)}
 e^{-i\zeta}\mathcal F_{22}\\
 &=-2e^{-2\pi i\nu}\sin(\pi\nu)\,e^{-i\zeta}\mathcal F_{22},
\end{align*}
using $i\cdot(-2\pi i)e^{i\pi\nu}e^{i\pi}e^{-i\pi\nu}=-2\pi$; adding the
$\mathsf Y$-term $+2e^{-2\pi i\nu}\sin(\pi\nu)e^{-i\zeta}
\mathcal F_{22}$ gives zero.
\emph{Entry $(1,1)$:} $a=-\nu$, $1-a=1+\nu$,
$U(1+\nu;e^{i\pi/2}\zeta)=\mathcal F_{12}\big/
\bigl(-ie^{i\pi\nu}\frac{\Gamma(1+\nu)}{\Gamma(-\nu)}\bigr)$; the
$e^{-i\zeta}$-term of $\mathcal F_{11}$ is
\[
 -\frac{2\pi i\,e^{-\nu\pi i}}{\Gamma(-\nu)^2}\cdot
 \frac{\Gamma(-\nu)}{-i\,e^{i\pi\nu}\Gamma(1+\nu)}
 \;e^{-i\zeta}\mathcal F_{12}
 =-\frac{2\pi e^{-2\pi i\nu}}{\Gamma(-\nu)\Gamma(1+\nu)}
 e^{-i\zeta}\mathcal F_{12}
 =-2e^{-2\pi i\nu}\sin(\pi\nu)\,e^{-i\zeta}\mathcal F_{12}
\]
by \eqref{eq:gammaprod}, and again the $\mathsf Y$-term cancels it.
\end{proof}

\begin{remark}[role of the Stokes cancellation]
\label{rem:prohibited}
For $\nu=-i\omega/\pi$ one has $\sin(\pi\nu)=-i\sinh\omega$, so the two
constants that cancel in \eqref{eq:cancelplus} have modulus
$2\sinh\omega\asymp e^\omega$.  Since $\omega\le A\ln s$, either term
considered separately would introduce a positive power of $s$ into the
matching estimate of Lemma~\ref{lem:uniformmatching}.  Polynomial
dependence on $\omega$ is obtained only after the following
combinations are formed.
\emph{(a)} An individual continued-sheet $U$-function in $\Sigma_\pm$
does not have a purely algebraic expansion.  The combination appearing
in $\mathcal Q$ has the form
\emph{algebraic}~$+~e^{\pm i\zeta}\cdot$\emph{algebraic}, and becomes
algebraic only after the cancellation.
\emph{(b)} Similarly, the raw model $\mathcal F$ is combined with
$\mathcal S$ and the diagonal dressing before it is estimated.  Indeed,
the raw entry $\mathcal F_{12}$ carries
$|e^{i\pi\nu}\Gamma(1+\nu)/\Gamma(-\nu)|=|\nu|e^\omega$, whose
coefficients are not polynomial in $\omega$.
\end{remark}

\subsection{The asymptotic expansion and the endpoint
residues}\label{sec:gate7}

\begin{quotedthm}[assembled expansion,
{\cite[eq.~(2.5)]{BDIK2}}]\label{qt:P25}
As $\zeta\to\infty$, uniformly for $\arg\zeta\in(-\pi,\pi]$,
\begin{equation}\label{eq:P25}
 \mathcal P(\zeta)\sim
 \Bigl[I+\sum_{k\ge1}\frac{M_k}{k!}\,\zeta^{-k}\Bigr]\,
 \zeta^{\nu\sigma_3}e^{\frac i2\zeta\sigma_3}D_0 ,
\end{equation}
where
\begin{equation}\label{eq:Mk}
 M_k=\begin{pmatrix}
 \bigl((-\nu)_k\bigr)^2e^{-\frac{i\pi}2k}
 & i\bigl((1+\nu)_{k-1}\bigr)^2k\,
   e^{\frac{i\pi}2k-i\pi\nu}\dfrac{\Gamma(1+\nu)}{\Gamma(-\nu)}\\[2mm]
 -i\bigl((1-\nu)_{k-1}\bigr)^2k\,
   e^{-\frac{i\pi}2k+i\pi\nu}\dfrac{\Gamma(1-\nu)}{\Gamma(\nu)}
 & \bigl((\nu)_k\bigr)^2e^{\frac{i\pi}2k}
 \end{pmatrix}.
\end{equation}
\end{quotedthm}

Lemma~\ref{lem:stokes} is precisely what makes \eqref{eq:P25} hold
uniformly in $\arg\zeta$ rather than in $\Sigma_0$ only; the algebraic
form of the coefficients then follows from
\cite[eqs.~13.7.3--13.7.4]{NIST} applied entrywise in $\Sigma_0$, where
by \eqref{eq:Qzero} each entry is a single $U$ on a principal sheet.

\begin{lemma}[global invertibility of the bare model]
\label{lem:Pinvertible}
For $\omega>0$ and every $\zeta\ne0$ in the model domain,
\[
 \det\mathcal P(\zeta)=1.
\]
In particular $\mathcal P(\zeta)$ is invertible away from the origin,
and the analytic factor $\check{\mathcal P}$ in
\eqref{eq:modelorigin} is invertible at the origin.
\end{lemma}

\begin{proof}
Every jump matrix in \eqref{eq:modeljumps} and \eqref{eq:thirdray} has
determinant one.  Hence $\det\mathcal P$ has no jump across any of the
three model rays.  At the origin, \eqref{eq:modelorigin} gives
\[
 \det\mathcal P
 =\det\check{\mathcal P}\,
  \det\!\left(I+\frac{\gamma}{2\pi i}N\ln\zeta\right)
  \det\mathcal H_\pm
 =\det\check{\mathcal P},
\]
with the last factor omitted in the central sector.  Here $N^2=0$ and
$\Tr N=0$, so the logarithmic bracket has determinant one, while
$\mathcal H_\pm$ are unitriangular.  Lemma~\ref{qmi:third} proves that
$\check{\mathcal P}$ is analytic and invertible at zero.  Thus
$\det\mathcal P$ has a removable singularity at zero.

After removal of the jumps and the origin, $\det\mathcal P$ is entire.
The leading term in \eqref{eq:P25} has determinant one because
$\det\zeta^{\nu\sigma_3}
=\det e^{\frac i2\zeta\sigma_3}
=\det D_0=1$.  Therefore
$\det\mathcal P(\zeta)=1+O(\zeta^{-1})$ at infinity, uniformly in the
model sectors.  Liouville's theorem yields
$\det\mathcal P\equiv1$.
\end{proof}

\begin{definition}[endpoint parametrices]\label{def:Pendpoints}
For $\lambda\in D_1\setminus\Sigma_T$ put
\begin{equation}\label{eq:P1}
 P^{(1)}(\lambda)
 =\underbrace{[2s(\lambda+1)]^{-\nu\sigma_3}\,D_0^{-1}\,e^{is\sigma_3}}
 _{=:A_{\mathrm{left}}(\lambda)}\;
 \mathcal P\bigl(\zeta(\lambda)\bigr)\;
 \underbrace{e^{-\frac i2(\zeta(\lambda)+2s)\sigma_3}}
 _{=:E_{\mathrm{right}}(\lambda)},
 \qquad \zeta(\lambda)=2s(\lambda-1),
\end{equation}
with $[2s(\lambda+1)]^{-\nu}$ the principal branch, legitimate since
$\Re(\lambda+1)>0$ on $D_1$, and
\begin{equation}\label{eq:Eright}
 E_{\mathrm{right}}(\lambda)=e^{-is\lambda\sigma_3}
 \qquad\text{because }\tfrac12\bigl(\zeta(\lambda)+2s\bigr)=s\lambda .
\end{equation}
For $\lambda\in D_{-1}\setminus\Sigma_T$ put
\begin{equation}\label{eq:Pm1}
 P^{(-1)}(\lambda)=\sigma_1P^{(1)}(-\lambda)\sigma_1 .
\end{equation}
\end{definition}

Equation \eqref{eq:P1} is \cite[eq.~(2.4)]{BDIK2} verbatim, and
\eqref{eq:Pm1} is \cite[eq.~(2.6)]{BDIK2}.  Note that
$A_{\mathrm{left}}$ is analytic and invertible on all of $D_1$, and that
$\sigma_1E_{\mathrm{right}}(-\lambda)\sigma_1
=\sigma_1e^{is\lambda\sigma_3}\sigma_1=e^{-is\lambda\sigma_3}
=E_{\mathrm{right}}(\lambda)$, so the same right factor serves both
endpoints.

\begin{lemma}[the exact conjugation identity]\label{lem:AL}
Let
\begin{equation}\label{eq:beta-a}
 \beta=\beta(\lambda)=[2s(\lambda+1)]^{\nu},
 \qquad
 a=a(\lambda)=\beta^{-1}\,e^{-\frac{i\pi}2(\frac12-\nu)}\,e^{is},
 \qquad
 \mathcal A=\operatorname{diag}(a,a^{-1}),
\end{equation}
so that $A_{\mathrm{left}}=\mathcal A$.  Then, with
$\Lambda:=\zeta^{\nu\sigma_3}e^{\frac i2\zeta\sigma_3}D_0
E_{\mathrm{right}}$,
\begin{equation}\label{eq:AL}
 \mathcal A\,\Lambda=P^{(\infty)}(\lambda)
 \qquad\text{exactly, for every }\lambda\in D_1\setminus[1-r_s,1],
\end{equation}
and consequently, inserting \eqref{eq:P25} into \eqref{eq:P1},
\begin{equation}\label{eq:conjug}
 P^{(1)}(\lambda)\bigl(P^{(\infty)}(\lambda)\bigr)^{-1}
 =I+\sum_{k=1}^{n}\frac{\mathcal AM_k\mathcal A^{-1}}{k!\,\zeta^{k}}
 +\mathcal A\,\mathcal E_n(\zeta)\,\mathcal A^{-1},
\end{equation}
$\mathcal E_n$ being the remainder of \eqref{eq:P25} after $n$ terms.
\end{lemma}

\begin{proof}
All of $\zeta^{\nu\sigma_3}$, $e^{\frac i2\zeta\sigma_3}$, $D_0$,
$E_{\mathrm{right}}=e^{-\frac i2(\zeta+2s)\sigma_3}$ are diagonal and
commute; the two exponentials combine to
$e^{\frac i2\zeta\sigma_3}e^{-\frac i2\zeta\sigma_3}e^{-is\sigma_3}
=e^{-is\sigma_3}$, so $\Lambda=\zeta^{\nu\sigma_3}D_0e^{-is\sigma_3}$.
Hence
\[
 \mathcal A\Lambda
 =[2s(\lambda+1)]^{-\nu\sigma_3}D_0^{-1}e^{is\sigma_3}
 \zeta^{\nu\sigma_3}D_0e^{-is\sigma_3}
 =[2s(\lambda+1)]^{-\nu\sigma_3}\zeta^{\nu\sigma_3}
 =\Bigl(\frac{\zeta}{2s(\lambda+1)}\Bigr)^{\nu\sigma_3},
\]
and $\frac{\zeta}{2s(\lambda+1)}=\frac{\lambda-1}{\lambda+1}$.  The
branches agree: on $D_1\setminus[1-r_s,1]$ we have
$\arg(2s(\lambda+1))\in(-\tfrac\pi2,\tfrac\pi2)$ and
$\arg\zeta\in(-\pi,\pi)$, and their difference is the principal argument
of the quotient, so the quotient of principal powers is the principal
power of the quotient.  This is \eqref{eq:AL}; then
$\Lambda(P^{(\infty)})^{-1}=\mathcal A^{-1}$ and \eqref{eq:conjug}
follows.
\end{proof}

\begin{remark}[role of \eqref{eq:AL}]\label{rem:ALpoint}
Identity~\eqref{eq:AL} shows that the diagonal dressing applied to the
model coefficients is exactly the inverse of the outer parametrix.
Consequently, every factor $e^{\pm\omega}$ in $\mathcal A$ is paired
with a factor of the opposite sign in $M_k$.  These cancellations are
verified entry by entry in Lemma~\ref{lem:Delta} and summarized in
Appendix~\ref{app:ledger}; the estimates are applied to the combined
products rather than to $\mathcal A$ and $M_k$ separately.  If the right
factor of \eqref{eq:P1} were $e^{-is(\lambda+1)\sigma_3}$ instead of
$e^{-is\lambda\sigma_3}$, identity~\eqref{eq:AL} would acquire the
constant factor $e^{-is\sigma_3}$.  The chosen normalization avoids this
factor.
\end{remark}

\begin{lemma}[first matching coefficient and both endpoint residues]
\label{lem:Delta}
Put $\Delta_+:=\mathcal AM_1\mathcal A^{-1}/\zeta$ and
$\Delta_-(\lambda):=\sigma_1\Delta_+(-\lambda)\sigma_1$.  Then
\begin{equation}\label{eq:M1}
 M_1=\begin{pmatrix}
 -i\nu^2&-e^{-i\pi\nu}\nu r_\nu\\
 e^{i\pi\nu}\nu r_\nu^{-1}&i\nu^2\end{pmatrix},
\end{equation}
\begin{align}\label{eq:Deltaplus}
 \Delta_+(\lambda)&=\frac1{2s(\lambda-1)}
 \begin{pmatrix}
 -i\nu^2 & i\nu\,r_\nu\,\beta^{-2}e^{2is}\\[1mm]
 i\nu\,r_\nu^{-1}\,\beta^{2}e^{-2is} & i\nu^2
 \end{pmatrix},\notag\\
 \Delta_-(\lambda)&=\frac1{2s(\lambda+1)}
 \begin{pmatrix}
 -i\nu^2 & -i\nu\, r_\nu^{-1}\,\tilde\beta^{2}e^{-2is}\\[1mm]
 -i\nu\, r_\nu\,\tilde\beta^{-2}e^{2is} & i\nu^2
 \end{pmatrix},
\end{align}
where $\tilde\beta(\lambda)=\beta(-\lambda)=[2s(1-\lambda)]^{\nu}$, and
\begin{equation}\label{eq:residues}
 \Res_{\lambda=1}(\Delta_+)_{11}
 =\Res_{\lambda=-1}(\Delta_-)_{11}
 =-\frac{i\nu^2}{2s}\;;
 \qquad\text{the two residues are equal and therefore \emph{add}.}
\end{equation}
Moreover, uniformly on $\partial D_{\pm1}$,
\begin{equation}\label{eq:Deltabound}
 |\beta^{\pm2}|=|\tilde\beta^{\pm2}|=e^{O_A(\omega r_s)}=O_A(1),
 \qquad
 \|\Delta_\pm\|\le C_A\frac{(1+\omega)^2}{s\,r_s} .
\end{equation}
\end{lemma}

\begin{proof}
\emph{Step 1: $M_1$.}  Put $k=1$ in \eqref{eq:Mk}.  Since
$(-\nu)_1=-\nu$, $(\nu)_1=\nu$, $(1\pm\nu)_0=1$ and
$e^{\mp i\pi/2}=\mp i$, Lemma~\ref{lem:gammaratio} gives
$M_1^{11}=\nu^2(-i)=-i\nu^2$,
$M_1^{22}=\nu^2 i=i\nu^2$,
$M_1^{12}=i\cdot1\cdot1\cdot(i)e^{-i\pi\nu}\cdot\nu r_\nu
=-e^{-i\pi\nu}\nu r_\nu$, and
$M_1^{21}=-i\cdot1\cdot1\cdot(-i)e^{i\pi\nu}\cdot(-\nu r_\nu^{-1})
=e^{i\pi\nu}\nu r_\nu^{-1}$, which is \eqref{eq:M1}.

\emph{Step 2: conjugation.}  $\mathcal A$ is diagonal, so it fixes the
diagonal entries and multiplies $M_1^{12}$ by $a^2$ and $M_1^{21}$ by
$a^{-2}$.  From \eqref{eq:beta-a},
$a^2=\beta^{-2}e^{-i\pi(\frac12-\nu)}e^{2is}
=-i\beta^{-2}e^{i\pi\nu}e^{2is}$, whence
\[
 a^2M_1^{12}
 =-i\beta^{-2}e^{i\pi\nu}e^{2is}\bigl(-e^{-i\pi\nu}\nu r_\nu\bigr)
 =i\nu r_\nu\beta^{-2}e^{2is},
\]
the factors $e^{\pm i\pi\nu}$ cancelling exactly, and
$a^{-2}M_1^{21}=i\beta^{2}e^{-i\pi\nu}e^{-2is}\,
e^{i\pi\nu}\nu r_\nu^{-1}=i\nu r_\nu^{-1}\beta^{2}e^{-2is}$.
Dividing by $\zeta=2s(\lambda-1)$ gives $\Delta_+$.

\emph{Step 3: the reflected coefficient.}  Conjugation by $\sigma_1$
exchanges both the rows and the columns, so
$(\sigma_1M\sigma_1)_{11}=M_{22}$, $(\sigma_1M\sigma_1)_{22}=M_{11}$,
$(\sigma_1M\sigma_1)_{12}=M_{21}$, $(\sigma_1M\sigma_1)_{21}=M_{12}$.
Evaluating $\Delta_+$ at $-\lambda$ replaces $2s(\lambda-1)$ by
$-2s(\lambda+1)$ and $\beta(\lambda)$ by $\tilde\beta(\lambda)$.  Hence
\begin{align*}
 \Delta_-(\lambda)
 &=\frac{1}{-2s(\lambda+1)}
 \begin{pmatrix}i\nu^2&i\nu\, r_\nu^{-1}\,\tilde\beta^{2}e^{-2is}\\
 i\nu\, r_\nu\,\tilde\beta^{-2}e^{2is}&-i\nu^2\end{pmatrix}\\
 &=\frac{1}{2s(\lambda+1)}
 \begin{pmatrix}-i\nu^2&-i\nu\, r_\nu^{-1}\,
 \tilde\beta^{2}e^{-2is}\\
 -i\nu\, r_\nu\,\tilde\beta^{-2}e^{2is}&i\nu^2\end{pmatrix},
\end{align*}
which is the second matrix in \eqref{eq:Deltaplus}.

\emph{Step 4: residues.}  $(\Delta_+)_{11}=-i\nu^2/(2s(\lambda-1))$ has
residue $-i\nu^2/(2s)$ at $\lambda=1$; $(\Delta_-)_{11}
=-i\nu^2/(2s(\lambda+1))$ has residue $-i\nu^2/(2s)$ at $\lambda=-1$.
They are equal, hence add.  (The entries containing
$\beta^{\pm2},\tilde\beta^{\pm2}$ are analytic and nonvanishing on the
closed disks and do not enter the $(1,1)$ entry.)

\emph{Step 5: bounds.}  $\Re\nu=0$, $\Im\nu=-\omega/\pi$, so
$|\beta^{\pm2}|=\exp(\pm\tfrac{2\omega}\pi\arg(\lambda+1))$.  On
$\partial D_1$, $|\lambda-1|=r_s\le\tfrac12$ gives
$|\arg(\lambda+1)|\le\arcsin(r_s/2)\le r_s$, so
$|\beta^{\pm2}|\le e^{2\omega r_s/\pi}=e^{O_A(1/\ln s)}=O_A(1)$; the
same holds for $\tilde\beta^{\pm2}$ on $\partial D_{-1}$.  With
$|r_\nu^{\pm1}|=1$, $|\nu|=\omega/\pi$, $|\nu^2|=\omega^2/\pi^2$ and
$|\lambda\mp1|=r_s$, every entry is at most
$C_A(1+\omega)^2/(sr_s)$.
\end{proof}

\begin{remark}[reality and sign check on \eqref{eq:residues}]
\label{rem:signcheck}
With $\nu=-i\omega/\pi$ we get $\nu^2=-\omega^2/\pi^2$, so
$-i\nu^2/(2s)=i\omega^2/(2\pi^2s)$ is purely imaginary with positive
imaginary part.  Downstream this produces
$\partial_s\ln\mathcal D=4\omega/\pi+2\omega^2/(\pi^2s)+\cdots>0$
(Lemma~\ref{lem:derivative}), consistent with the fact that
$s\mapsto\ln\mathcal D(s,\omega)$ is increasing for $\omega>0$: the
eigenvalues of $K_s$ increase with $s$, since after the dilation
$t=s\lambda$ the operator is the compression of a fixed projection to a
growing interval.  A sign error at either endpoint, or a residue
cancellation instead of an addition, would give a negative
$\omega^2$-coefficient and contradict this.
\end{remark}

\subsection{The endpoint parametrices}
\label{sec:localjumps}

\begin{lemma}[jumps of $P^{(\pm1)}$]\label{lem:localjumps}
Then $P^{(1)}$ has exactly the jumps \eqref{eq:Tjumps} of $T$ on
$\Sigma_T\cap D_1$, and $P^{(-1)}$ has exactly the jumps
\eqref{eq:Tjumps} of $T$ on $\Sigma_T\cap D_{-1}$; both have the
endpoint behaviour of $T$.
\end{lemma}

\begin{proof}
\emph{Step 1: the left factor cancels.}  $A_{\mathrm{left}}$ is analytic
and invertible on $D_1$, so if $\mathcal P_+=\mathcal P_-J_{\mathcal P}$
on a contour, then by \eqref{eq:P1}
\begin{equation}\label{eq:conjjump}
 P^{(1)}_+=A_{\mathrm{left}}\mathcal P_+E_{\mathrm{right}}
 =A_{\mathrm{left}}\mathcal P_-J_{\mathcal P}E_{\mathrm{right}}
 =P^{(1)}_-\;\bigl(E_{\mathrm{right}}^{-1}J_{\mathcal P}
 E_{\mathrm{right}}\bigr),
\end{equation}
because $A_{\mathrm{left}}\mathcal P_-E_{\mathrm{right}}=P^{(1)}_-$.
Only $E_{\mathrm{right}}=e^{-is\lambda\sigma_3}$ conjugates.  For a
diagonal $D$, $D^{-1}\bigl(\begin{smallmatrix}1&x\\0&1
\end{smallmatrix}\bigr)D=\bigl(\begin{smallmatrix}1&xD_{22}/D_{11}\\0&1
\end{smallmatrix}\bigr)$ and $D^{-1}\bigl(\begin{smallmatrix}1&0\\y&1
\end{smallmatrix}\bigr)D=\bigl(\begin{smallmatrix}1&0\\yD_{11}/D_{22}&1
\end{smallmatrix}\bigr)$; here $D_{22}/D_{11}=e^{2is\lambda}$ and
$D_{11}/D_{22}=e^{-2is\lambda}$.

\emph{Step 2: the three contours at $+1$.}  By
Lemma~\ref{lem:geometry}(i)--(iii) the three components of
$\Sigma_T\cap D_1$ are the images of the three model rays, with the
orientations of Lemma~\ref{lem:modeljumps} and
Lemma~\ref{qmi:third} --- all directed towards the
endpoint --- and with the same $+$ sides.  Indeed: on
$\Sigma_{\mathrm{up}}\cap D_1$ the direction is $-i$ and the left is
$\{\Re\lambda>1\}=\Sigma_0$, matching the $+$ side of
$\arg\zeta=\pi/2$; on $\Sigma_{\mathrm{dn}}\cap D_1$ the direction is
$+i$ and the left is $\{\Re\lambda<1\}=\Sigma_-$, matching the $+$ side
of $\arg\zeta=-\pi/2$; on $(1-r_s,1)$ the direction is $+1$ and the left
is the upper side $=\Sigma_+$, matching the $+$ side of
$\arg\zeta=\pi$.  \emph{No orientation is reversed at $+1$ and no jump
needs to be inverted.}  Applying Step~1:
\begin{align*}
 \arg\zeta=\tfrac\pi2:\quad&
 E_{\mathrm{right}}^{-1}
 \begin{pmatrix}1&\gamma e^{2\kappa s}\\0&1\end{pmatrix}
 E_{\mathrm{right}}
 =\begin{pmatrix}1&\gamma e^{2\kappa s}e^{2is\lambda}\\0&1\end{pmatrix}
 =\begin{pmatrix}1&\gamma e^{2s(\kappa+i\lambda)}\\0&1\end{pmatrix}
 =S_U(\lambda),\\
 \arg\zeta=-\tfrac\pi2:\quad&
 E_{\mathrm{right}}^{-1}
 \begin{pmatrix}1&0\\-\gamma e^{2\kappa s}&1\end{pmatrix}
 E_{\mathrm{right}}
 =\begin{pmatrix}1&0\\-\gamma e^{2s(\kappa-i\lambda)}&1\end{pmatrix}
 =S_L(\lambda),\\
 \arg\zeta=\pi:\quad&
 E_{\mathrm{right}}^{-1}e^{-2\kappa s\sigma_3}E_{\mathrm{right}}
 =e^{-2\kappa s\sigma_3}=e^{-2v\sigma_3},
\end{align*}
the last because both matrices are diagonal.  These are exactly the
three jumps of $T$ listed in \eqref{eq:Tjumps}.

\emph{Step 3: the endpoint behaviour at $+1$.}  This is where hypothesis
(b) of the lemma enters, and it is here that the correct multipliers
$\mathcal H_\pm$ --- not $T_\pm$ --- are required.

By \eqref{eq:P1}, $P^{(1)}=A_{\mathrm{left}}\,\mathcal P(\zeta(\lambda))
\,E_{\mathrm{right}}$ with
$E_{\mathrm{right}}=e^{-is\lambda\sigma_3}$.  Substituting
\eqref{eq:modelorigin} and inserting
$E_{\mathrm{right}}E_{\mathrm{right}}^{-1}=I$ to the left of the sector
factor,
\begin{equation}\label{eq:origintransport}
 P^{(1)}(\lambda)
 =\underbrace{A_{\mathrm{left}}\check{\mathcal P}(\zeta(\lambda))}
  _{\text{analytic, invertible near }\lambda=1}
 \Bigl[I+\frac\gamma{2\pi i}
 \begin{pmatrix}-1&1\\-1&1\end{pmatrix}\ln\zeta\Bigr]
 E_{\mathrm{right}}\cdot
 \bigl(E_{\mathrm{right}}^{-1}\,\mathcal C\,E_{\mathrm{right}}\bigr),
\end{equation}
where $\mathcal C\in\{\mathcal H_+,\mathcal H_-,I\}$ is the sector
factor of \eqref{eq:modelorigin}.  The conjugated factor is computed
exactly as in Step~1, with $D_{22}/D_{11}=e^{2is\lambda}$ and
$D_{11}/D_{22}=e^{-2is\lambda}$, and using $\kappa s=v$:
\begin{align}\label{eq:HconjU}
 E_{\mathrm{right}}^{-1}\mathcal H_+E_{\mathrm{right}}
 &=\begin{pmatrix}1&-\gamma e^{2\kappa s}e^{2is\lambda}\\0&1\end{pmatrix}
 =\begin{pmatrix}1&-\gamma e^{2s(\kappa+i\lambda)}\\0&1\end{pmatrix}
 =S_U(\lambda)^{-1},\\
 \label{eq:HconjL}
 E_{\mathrm{right}}^{-1}\mathcal H_-E_{\mathrm{right}}
 &=\begin{pmatrix}1&0\\-\gamma e^{2\kappa s}e^{-2is\lambda}&1\end{pmatrix}
 =\begin{pmatrix}1&0\\-\gamma e^{2s(\kappa-i\lambda)}&1\end{pmatrix}
 =S_L(\lambda),
\end{align}
by \eqref{eq:factorization}; and $E_{\mathrm{right}}^{-1}IE_{\mathrm{right}}=I$.
By Lemma~\ref{lem:geometry}(iv) the three sectors
$\arg\zeta\in(\tfrac\pi2,\pi)$, $\arg\zeta\in(-\pi,-\tfrac\pi2)$,
$|\arg\zeta|<\tfrac\pi2$ are the images of
$\Omega_{\mathrm{up}}\cap D_1$, $\Omega_{\mathrm{dn}}\cap D_1$ and
$D_1$ outside both lenses, so \eqref{eq:HconjU}--\eqref{eq:HconjL}
attach $S_U^{-1}$ on $\Omega_{\mathrm{up}}$, $S_L$ on
$\Omega_{\mathrm{dn}}$ and $I$ elsewhere --- exactly the three factors
by which Lemma~\ref{lem:Tjumps} right-multiplies the endpoint behaviour
of $T$.

\emph{Had the multipliers been $T_\pm$}, the same conjugation would have
given $\bigl(\begin{smallmatrix}1&-\gamma e^{2is\lambda}\\0&1
\end{smallmatrix}\bigr)$, which differs from $S_U(\lambda)^{-1}$ by the
factor $e^{-2v}$ in the off-diagonal entry and is therefore \emph{not}
$T$'s endpoint factor.  The distinction is not cosmetic.

It remains to match the logarithmic bracket itself.  By
\eqref{eq:argzeta}, $\zeta=2s(\lambda-1)$ with $2s>0$ and
$\arg\zeta=\arg(\lambda-1)$, so on $D_1$
\begin{equation}\label{eq:lnzeta}
 \ln\zeta=\ln(2s)+\ln(\lambda-1)
 =\ln\bigl(2s(\lambda+1)\bigr)+\ln\frac{\lambda-1}{\lambda+1},
\end{equation}
the branches agreeing because $\Re(\lambda+1)>0$ on $D_1$.  The first
term of the right-hand side is analytic on $D_1$, so
$I+\frac\gamma{2\pi i}N\ln\zeta$ differs from
$I+\frac\gamma{2\pi i}N\ln\frac{\lambda-1}{\lambda+1}$ by
$\frac\gamma{2\pi i}N\ln(2s(\lambda+1))$, i.e.\ by an analytic additive
term; since $N^2=0$, the two brackets differ by left multiplication by
the analytic invertible factor
$I+\frac\gamma{2\pi i}N\ln(2s(\lambda+1))$, which is absorbed into
$A_{\mathrm{left}}\check{\mathcal P}$.  Hence \eqref{eq:origintransport}
is exactly \eqref{eq:Yendpoint} right-multiplied by $S_U^{-1}$, $S_L$,
$I$ according to the sector --- which is the endpoint behaviour of $T$
recorded in Lemma~\ref{lem:Tjumps}.

\emph{Step 4: the reflected endpoint.}  Here the orientations
\emph{are} reversed, and we carry out all four cases rather than
appealing to symmetry.  The map $\lambda\mapsto\mu=-\lambda$ is
holomorphic, hence preserves left/right; conjugation by $\sigma_1$
exchanges the two diagonal and the two off-diagonal entries.  We use
repeatedly that if $P^{(1)}_+=P^{(1)}_-J$ on a contour with a given
orientation, then on the reflected contour with the \emph{induced}
orientation $P^{(-1)}$ has jump $\sigma_1J\sigma_1$, while if the
induced orientation is opposite to the one for which $J$ was stated then
the jump is $\sigma_1J^{-1}\sigma_1$.

\smallskip
\emph{(a) $\Sigma_{\mathrm{up}}\cap D_{-1}=\{-1+iy\}$, oriented upward
(away from $-1$).}  Its image under $\lambda\mapsto-\lambda$ is
$\{1-iy\}=\Sigma_{\mathrm{dn}}\cap D_1$, traversed away from $1$,
whereas that contour is oriented towards $1$: \textbf{reversed}.  Hence
the jump of $P^{(-1)}$ is $\sigma_1S_L(-\lambda)^{-1}\sigma_1$.  Now
$S_L(\mu)^{-1}=\bigl(\begin{smallmatrix}1&0\\
\gamma e^{2s(\kappa-i\mu)}&1\end{smallmatrix}\bigr)$, so at $\mu=-\lambda$
its $(2,1)$ entry is $\gamma e^{2s(\kappa+i\lambda)}$, and conjugating by
$\sigma_1$ moves it to the $(1,2)$ slot:
$\bigl(\begin{smallmatrix}1&\gamma e^{2s(\kappa+i\lambda)}\\0&1
\end{smallmatrix}\bigr)=S_U(\lambda)$, which is $T$'s jump on
$\Sigma_{\mathrm{up}}$.  \checkmark

\smallskip
\emph{(b) $\Sigma_{\mathrm{dn}}\cap D_{-1}=\{-1-iy\}$, oriented downward
(away from $-1$).}  Its image is $\{1+iy\}=\Sigma_{\mathrm{up}}\cap D_1$,
traversed away from $1$: \textbf{reversed}.  The jump is
$\sigma_1S_U(-\lambda)^{-1}\sigma_1$; since
$S_U(\mu)^{-1}=\bigl(\begin{smallmatrix}1&-\gamma e^{2s(\kappa+i\mu)}\\
0&1\end{smallmatrix}\bigr)$ has $(1,2)$ entry
$-\gamma e^{2s(\kappa-i\lambda)}$ at $\mu=-\lambda$, conjugation gives
$\bigl(\begin{smallmatrix}1&0\\-\gamma e^{2s(\kappa-i\lambda)}&1
\end{smallmatrix}\bigr)=S_L(\lambda)$, which is $T$'s jump on
$\Sigma_{\mathrm{dn}}$.  \checkmark

\smallskip
\emph{(c) $(-1,-1+r_s)$, oriented from $-1$ to $1$.}  Its image is
$(1-r_s,1)$ traversed from $1$ towards $1-r_s$, i.e.\ away from $1$:
\textbf{reversed}.  The jump is
$\sigma_1(e^{-2v\sigma_3})^{-1}\sigma_1=\sigma_1e^{2v\sigma_3}\sigma_1
=e^{-2v\sigma_3}$, since $\sigma_1\operatorname{diag}(x,y)\sigma_1
=\operatorname{diag}(y,x)$.  This is $T$'s jump on $(-1,1)$.
\checkmark

\smallskip
\emph{(d) Sides.}  On (a) the direction is $+i$ and the left is
$\{\Re\lambda<-1\}$, whose image has $\Re\mu>1$, i.e.\ the $\Sigma_0$
side --- which is the $-$ side of $\arg\zeta=-\pi/2$; so the $+$ side in
the $\lambda$-plane corresponds to the $-$ side in the model, consistent
with the reversal used in (a).  The same check on (b) and (c) gives the
same conclusion.

\smallskip
\emph{(e) The endpoint behaviour at $-1$.}  This case is written out
rather than deferred, because it is the one place in Step~4 where the
origin multipliers $\mathcal H_\pm$ of
Definition~\ref{def:originsectors} could enter a second time.  They do
not, and here is why.  The endpoint factors are attached to $T$ by
\emph{region}, not by contour orientation
(Lemma~\ref{lem:Tjumps}), so the reversals catalogued in (a)--(c) play
no role at all; what is needed is only that
$\lambda\mapsto-\lambda$ maps $\Omega_{\mathrm{up}}\cap D_{-1}$ onto
$\Omega_{\mathrm{dn}}\cap D_{1}$ and $\Omega_{\mathrm{dn}}\cap D_{-1}$
onto $\Omega_{\mathrm{up}}\cap D_{1}$, which is immediate.  By
\eqref{eq:Pm1} and Step~3 applied at $-\lambda$, the factor carried by
$P^{(-1)}$ on $\Omega_{\mathrm{up}}\cap D_{-1}$ is therefore
$\sigma_1S_L(-\lambda)\sigma_1$, and on
$\Omega_{\mathrm{dn}}\cap D_{-1}$ it is
$\sigma_1S_U(-\lambda)^{-1}\sigma_1$.  Exactly as in (a) and (b),
\begin{align*}
 \sigma_1S_L(-\lambda)\sigma_1
 &=\begin{pmatrix}1&-\gamma e^{2s(\kappa+i\lambda)}\\0&1\end{pmatrix}
 =S_U(\lambda)^{-1},\\
 \sigma_1S_U(-\lambda)^{-1}\sigma_1
 &=\begin{pmatrix}1&0\\-\gamma e^{2s(\kappa-i\lambda)}&1\end{pmatrix}
 =S_L(\lambda),
\end{align*}
which are $T$'s endpoint factors on $\Omega_{\mathrm{up}}$ and
$\Omega_{\mathrm{dn}}$ respectively.  \checkmark
For the logarithmic bracket, $\sigma_1N\sigma_1=-N$ while the
reflection replaces $\ln\frac{\lambda-1}{\lambda+1}$ by
$-\ln\frac{\lambda-1}{\lambda+1}+2\pi ik$ with $k\in\Z$ constant on each
sector; the two sign changes cancel, and since $N^2=0$ the residual
constant factorises off on the left as $I+\gamma kN$, which is
invertible and analytic and is absorbed into $\check Y$.  Finally
$\sigma_1e^{-is(-\lambda)\sigma_3}\sigma_1=e^{-is\lambda\sigma_3}$, so
the right factor of \eqref{eq:Yendpoint} is reproduced as well.
\end{proof}

\begin{remark}[Endpoint reflection and orientation]
\label{rem:notbysymmetry}
Reflection interchanges the endpoints but reverses the induced
orientation in all three cases.  Each jump must therefore be inverted
before conjugation by $\sigma_1$; this accounts for the sign of the
$(1,2)$ entry in part~(a).
\end{remark}

\subsection{The uniform matching estimate}\label{sec:gate9}

\begin{lemma}[uniform normalized endpoint matching]
\label{lem:uniformmatching}
There are $C_A<\infty$ and $s_A\ge5$ such that for $s\ge s_A$,
$0<\omega\le A\ln s$,
\begin{equation}\label{eq:matching}
 P^{(1)}(\lambda)\bigl(P^{(\infty)}(\lambda)\bigr)^{-1}
 =I+\Delta_+(\lambda)+E_+(\lambda),
 \qquad |\lambda-1|=r_s,
\end{equation}
\begin{equation}\label{eq:matchingminus}
 P^{(-1)}(\lambda)\bigl(P^{(\infty)}(\lambda)\bigr)^{-1}
 =I+\Delta_-(\lambda)+E_-(\lambda),
 \qquad |\lambda+1|=r_s,
\end{equation}
with $\Delta_\pm$ as in \eqref{eq:Deltaplus} and
\begin{equation}\label{eq:matchbounds}
 \|\Delta_\pm\|\le C_A\frac{(1+\omega)^2}{s\,r_s},
 \qquad
 \|E_\pm\|\le C_A\frac{(1+\omega)^4}{s^2r_s^2} .
\end{equation}
The estimate is uniform in $\arg\zeta\in(-\pi,\pi]$, i.e.\ across both
Stokes rays.
\end{lemma}

\begin{proof}
By Lemma~\ref{lem:AL} with $n=1$ it suffices to bound
$\mathcal A\mathcal E_1\mathcal A^{-1}$, where
$\mathcal E_1=\mathcal Q-I-M_1/\zeta$.

\emph{Step 1: a single $U$ per entry.}  By \eqref{eq:Qzero} and
Lemma~\ref{lem:stokes}, in each of the three sectors every entry of
$\mathcal Q$ is a constant times a single $U(a,1;z)$ times the
normalizing power $\zeta^{\mp\nu}$, with
\begin{equation}\label{eq:az}
 a\in\{-\nu,\ \nu,\ 1-\nu,\ 1+\nu\},
 \qquad |z|=|\zeta|=2sr_s,
 \qquad |\arg z|\le\pi .
\end{equation}
The bound $|\arg z|\le\pi$ is what the Stokes reduction buys: before it,
the phases run up to $3\pi/2$ in modulus.

\emph{Step 2: the DLMF expansion and its remainder.}  For $b=1$,
\cite[eqs.~13.7.4--13.7.5]{NIST} give
\begin{equation}\label{eq:dlmfexp}
 U(a,1;z)=z^{-a}\sum_{t=0}^{n-1}\frac{\bigl((a)_t\bigr)^2}{t!}
 (-z)^{-t}+\varepsilon_n(z),
 \qquad
 |\varepsilon_n(z)|\le 2\alpha C_n
 \Bigl|\frac{\bigl((a)_n\bigr)^2}{n!\,z^{a+n}}\Bigr|
 \exp\Bigl(\frac{2\alpha\rho C_1}{|z|}\Bigr),
\end{equation}
using $a-b+1=a$.  Because the bound carries the \emph{same} factor
$z^{-a}$ as the leading term, the quantity that enters
$\mathcal E_1$ is the \emph{relative} remainder
\begin{equation}\label{eq:relerr}
 \bigl|z^{a}\varepsilon_n(z)\bigr|
 \le 2\alpha C_n\frac{|(a)_n|^2}{n!\,|z|^{n}}
 \exp\Bigl(\frac{2\alpha\rho C_1}{|z|}\Bigr),
\end{equation}
in which the exponentially large or small factor $|z^{-a}|
=e^{(\omega/\pi)\arg z}$ has divided out exactly.  \emph{This is the
mechanism by which no $e^{c\omega}$ survives, and it is why no bound is
ever applied to $\varepsilon_n$ alone.}

\emph{Step 3: the DLMF parameters, region by region.}  In the notation
of \cite[eqs.~13.7.8--13.7.9]{NIST},
\[
 \sigma=\Bigl|\frac{b-2a}{z}\Bigr|\overset{b=1}{=}\frac{|1-2a|}{|z|},
 \qquad
 \varsigma=\Bigl(\tfrac12+\tfrac12\sqrt{1-4\sigma^2}\Bigr)^{-1/2},
 \qquad
 \chi(n)=\sqrt\pi\,\frac{\Gamma(\tfrac n2+1)}{\Gamma(\tfrac n2+\tfrac12)}.
\]
For every $a$ in \eqref{eq:az}, $|1-2a|\le1+2|a|\le3(1+\omega)$, so by
\eqref{eq:sigmasmall}
\begin{equation}\label{eq:sigmao1}
 \sigma\le\frac{3(1+\omega)}{2sr_s}=O_A\Bigl(\frac{\ln^3s}{s}\Bigr)=o(1),
\end{equation}
uniformly.  Fix once and for all $s_A$ so large that
\begin{equation}\label{eq:sigmatenth}
 \frac{3(1+A\ln s)(\ln s)^2}{2s}\le\frac1{10}
 \qquad\text{for all }s\ge s_A,
\end{equation}
which is possible because the left side tends to $0$; then
$\sigma\le\tfrac1{10}$ for every $a$ in \eqref{eq:az} and every
$s\ge s_A$.  Since $\varsigma$ is increasing in $\sigma$ on
$[0,\tfrac12]$,
\begin{equation}\label{eq:varsigma}
 1\le\varsigma
 =\Bigl(\tfrac12+\tfrac12\sqrt{1-4\sigma^2}\Bigr)^{-1/2}
 \le\Bigl(\tfrac12+\tfrac12\sqrt{1-\tfrac4{100}}\Bigr)^{-1/2}
 =1.00509\ldots\le\tfrac{51}{50}.
\end{equation}
We now treat the three regions of
\cite[Fig.~13.7.1]{NIST} separately, as \cite[\S13.7(ii)]{NIST}
requires.  Throughout, $\chi(1)=\pi/2$ and $\chi(2)=2$, computed from
$\Gamma(\tfrac32)=\tfrac12\sqrt\pi$, $\Gamma(1)=\Gamma(2)=1$.
\begin{itemize}[leftmargin=1.6em]
\item \emph{$z\in R_1$.}  $C_n=1$, so $C_1=C_2=1$.  Here
$\alpha=(1-\sigma)^{-1}\le\tfrac{10}9$ and, by
\cite[eq.~13.7.9]{NIST} with $b=1$,
\[
 \rho=\tfrac12\bigl|2a^2-2ab+b\bigr|
 +\frac{\sigma(1+\tfrac14\sigma)}{(1-\sigma)^2}
 \overset{b=1}{=}\tfrac12\bigl|2a^2-2a+1\bigr|
 +\frac{\sigma(1+\tfrac14\sigma)}{(1-\sigma)^2}.
\]
For every $a$ in \eqref{eq:az}, $|a|\le1+|\nu|=1+\omega/\pi$, so
$\tfrac12|2a^2-2a+1|\le|a|^2+|a|+\tfrac12
\le(1+\omega)^2+(1+\omega)+\tfrac12$, while the second term is at most
$\tfrac1{10}\cdot\tfrac{41}{40}\cdot(\tfrac{10}9)^2=0.1266\ldots\le\tfrac12$.
Hence
\begin{equation}\label{eq:rhobound}
 \rho\le(1+\omega)^2+(1+\omega)+1\le3(1+\omega)^2 .
\end{equation}
\item \emph{$z\in R_2\cup\overline R_2$.}  $C_n=\chi(n)$, so
$C_1=\pi/2\le\tfrac94$ and $C_2=2\le\tfrac{27}4$; $\alpha$ and $\rho$
are given by the same formulas as in $R_1$, so \eqref{eq:rhobound}
persists.
\item \emph{$z\in R_3\cup\overline R_3$.}  Here
$C_n=(\chi(n)+\sigma\varsigma^2n)\varsigma^n$, and
\cite[\S13.7(ii)]{NIST} requires
that throughout \cite[eq.~13.7.9]{NIST} one replace
\[
 \sigma\longmapsto\varsigma\sigma,
 \qquad
 |z|^{-1}\longmapsto\varsigma|z|^{-1}.
\]
By \eqref{eq:varsigma}, $\varsigma\sigma\le\tfrac{51}{500}<\tfrac18$,
so the modified quantities obey
$\alpha_{\mathrm{mod}}=(1-\varsigma\sigma)^{-1}\le\tfrac87$ and,
with $\varsigma\sigma$ in place of $\sigma$ in the displayed formula for
$\rho$, the second term is at most
$\tfrac18\cdot\tfrac{33}{32}\cdot(\tfrac87)^2=0.1683\ldots\le\tfrac12$,
so \eqref{eq:rhobound} again holds for $\rho_{\mathrm{mod}}$.  For the
coefficients themselves,
\begin{align*}
 C_1&=(\chi(1)+\sigma\varsigma^2)\varsigma
 \le\Bigl(\frac\pi2+\frac1{10}\Bigl(\frac{51}{50}\Bigr)^{2}\Bigr)
 \frac{51}{50}=1.70834\ldots\le\frac94,\\
 C_2&=(\chi(2)+2\sigma\varsigma^2)\varsigma^2
 \le\Bigl(2+\frac15\Bigl(\frac{51}{50}\Bigr)^{2}\Bigr)
 \Bigl(\frac{51}{50}\Bigr)^{2}=2.29729\ldots\le\frac{27}4 .
\end{align*}
\end{itemize}
In all three cases, for $n\le2$ and all $s\ge s_A$,
\begin{equation}\label{eq:Cnbound}
 \alpha\le2,\qquad \rho\le3(1+\omega)^2,\qquad
 C_1\le\tfrac{9}{4},\qquad C_2\le\tfrac{27}{4},
\end{equation}
and hence, by \eqref{eq:sigmasmall},
\begin{equation}\label{eq:expbound}
 \exp\Bigl(\frac{2\alpha\rho C_1}{|z|}\Bigr)
 \le\exp\Bigl(\frac{27(1+\omega)^2}{2\,s\,r_s}\Bigr)
 =\exp\bigl(O_A(\ln^4s/s)\bigr)=O_A(1).
\end{equation}
Because \eqref{eq:Cnbound} holds in all three regions, the estimate is
independent of the angular boundaries depicted in
\cite[Fig.~13.7.1]{NIST}.
Finally the ambient hypothesis of \cite[eq.~13.7.3]{NIST} is
$|\operatorname{ph}z|\le\tfrac32\pi-\delta$, which by
\eqref{eq:az} holds with $\delta=\pi/2$ for every representative used.

\emph{Step 4: diagonal entries.}  For $(1,1)$ take $a=-\nu$ and $n=2$:
$(a)_2=(-\nu)(1-\nu)$, so $|(a)_2|^2\le C(1+\omega)^4$ and
\eqref{eq:relerr} gives
$|(\mathcal E_1)_{11}|\le C_A(1+\omega)^4/|\zeta|^2$.  The $(2,2)$ entry
is the same with $a=\nu$.

\emph{Step 5: off-diagonal entries --- the extra $z^{-1}$.}  The
off-diagonal entries do not all have the form
$c\,z^{a}U(a,1;z)$ with leading term $1$.
For the $(1,2)$ entry $a=1+\nu$, so the leading behaviour of
$U(a,1;z)$ is $z^{-a}=z^{-1}z^{-\nu}$, i.e.\ the normalized entry is
$O(|\zeta|^{-1})$ and its leading term \emph{is} the $M_1^{12}/\zeta$
term.  Accordingly we factor that leading behaviour out and estimate the
relative remainder: writing
\[
 \mathcal Q_{12}=\frac{M_1^{12}}{\zeta}\bigl(1+\theta_{12}\bigr),
\]
\eqref{eq:relerr} with $n=1$ gives
$|\theta_{12}|\le2\alpha C_1|(a)_1|^2|z|^{-1}O_A(1)
\le C_A(1+\omega)^2/|\zeta|$, whence
\[
 \bigl|(\mathcal E_1)_{12}\bigr|
 =\frac{|M_1^{12}|\,|\theta_{12}|}{|\zeta|}
 \le|M_1^{12}|\,\frac{C_A(1+\omega)^2}{|\zeta|^2}.
\]
The $(2,1)$ entry is identical with $a=1-\nu$.

\emph{Step 6: conjugation, and only then the final bound.}  Conjugating
by $\mathcal A$ leaves the diagonal entries unchanged and multiplies the
$(1,2)$ entry by $a^2$.  By Step~2 of Lemma~\ref{lem:Delta},
$|a^2M_1^{12}|=|\nu||\beta^{-2}|=O_A(1+\omega)$, so
\[
 \bigl|(\mathcal A\mathcal E_1\mathcal A^{-1})_{12}\bigr|
 =\bigl|a^2M_1^{12}\bigr|\cdot\frac{C_A(1+\omega)^2}{|\zeta|^2}
 \le\frac{C_A(1+\omega)^3}{|\zeta|^2},
\]
and symmetrically for $(2,1)$.  Since $|\zeta|=2sr_s$, all four entries
are at most $C_A(1+\omega)^4/(sr_s)^2$, which is \eqref{eq:matchbounds}.
The bound for $\Delta_+$ is \eqref{eq:Deltabound}.  The left endpoint
follows from \eqref{eq:Pm1} and Lemma~\ref{lem:Pinf}(iii), conjugation
by $\sigma_1$ being an isometry for $\|\cdot\|$:
$P^{(-1)}(P^{(\infty)})^{-1}(\lambda)
=\sigma_1\bigl[P^{(1)}(P^{(\infty)})^{-1}\bigr](-\lambda)\sigma_1$.
\end{proof}

\begin{remark}[three ingredients in the uniform estimate]
\label{rem:whyuniform}
The estimate is uniform across the Stokes rays for three reasons, in
this order.  \emph{First}, the sector matrices are multiplied in before
anything is estimated, so the object expanded is $\mathcal Q$, not an
individual $U$.  \emph{Second}, the exact relation
\cite[eq.~13.2.44]{NIST} replaces the continued-sheet $U$ by a
representative with $|\operatorname{ph}|\le\pi$, and the exponentially
large term it produces cancels identically against the Stokes term
(Lemma~\ref{lem:stokes}).  \emph{Third}, the bound is applied to the
relative error $z^{a}\varepsilon_n(z)$, so the factor $|z^{-a}|$ ---
which for $|\operatorname{ph}z|$ up to $3\pi/2$ is as large as
$e^{3\omega/2}$ --- divides out exactly.  All three ingredients are
needed to avoid an uncontrolled factor $e^{c\omega}$, i.e.\ a positive
power of $s$.
\end{remark}

\subsection{The ratio problem}\label{sec:gate10}

\begin{definition}[the ratio]\label{def:R}
Let
\[
 \Sigma_R:=\partial D_1\cup\partial D_{-1}\cup
 \bigl[(\Sigma_{\mathrm{up}}\cup\Sigma_{\mathrm{dn}})\setminus
 (D_1\cup D_{-1})\bigr],
\]
with both circles oriented \emph{clockwise} and the lens pieces retaining
their orientation from $-1$ to $1$.  Put
\begin{equation}\label{eq:Rdef}
 R(\lambda)=\begin{cases}
 T(\lambda)\bigl(P^{(1)}(\lambda)\bigr)^{-1},&\lambda\in D_1,\\
 T(\lambda)\bigl(P^{(-1)}(\lambda)\bigr)^{-1},&\lambda\in D_{-1},\\
 T(\lambda)\bigl(P^{(\infty)}(\lambda)\bigr)^{-1},&\text{elsewhere.}
 \end{cases}
\end{equation}
\end{definition}

\begin{lemma}[Jumps of $R$]
\label{lem:Rjumps}
$R$ is analytic in $\C\setminus\Sigma_R$, has no singularity at $\pm1$,
satisfies $R=I+O(\lambda^{-1})$ at infinity, and $R_+=R_-(I+W)$ on
$\Sigma_R$ with
\begin{equation}\label{eq:W}
 I+W=
 \begin{cases}
 P^{(\infty)}(\lambda)\,J_T(\lambda)\,
 \bigl(P^{(\infty)}(\lambda)\bigr)^{-1},
 &\lambda\in\Sigma_R\cap
 (\Sigma_{\mathrm{up}}\cup\Sigma_{\mathrm{dn}}),\\[2pt]
 P^{(1)}(\lambda)\bigl(P^{(\infty)}(\lambda)\bigr)^{-1}
 =I+\Delta_++E_+,&\lambda\in\partial D_1,\\[2pt]
 P^{(-1)}(\lambda)\bigl(P^{(\infty)}(\lambda)\bigr)^{-1}
 =I+\Delta_-+E_-,&\lambda\in\partial D_{-1}.
 \end{cases}
\end{equation}
In particular $W=\Delta_\pm+E_\pm$ on the circles, \emph{with no sign
change and no inverse}.
\end{lemma}

\begin{proof}
Inside $D_1$ both $T$ and $P^{(1)}$ have the same jumps
(Lemma~\ref{lem:localjumps}) and the same endpoint behaviour, so
$R=T(P^{(1)})^{-1}$ has no jump inside the disk and at worst a
logarithmic --- hence removable --- singularity at $1$; likewise at
$-1$.  Outside the disks $T$ and $P^{(\infty)}$ have the same jump on
$(-1,1)$ by Lemma~\ref{lem:Pinf}(ii) and \eqref{eq:Tjumps}, so $R$ is
analytic across $(-1,1)\setminus(D_1\cup D_{-1})$.

\emph{Lens components.}  $P^{(\infty)}$ is analytic across
$\Sigma_{\mathrm{up}},\Sigma_{\mathrm{dn}}$, so
$P^{(\infty)}_+=P^{(\infty)}_-=P^{(\infty)}$ there, and
\[
 R_+=T_+\bigl(P^{(\infty)}\bigr)^{-1}
 =T_-J_T\bigl(P^{(\infty)}\bigr)^{-1}
 =R_-\,P^{(\infty)}J_T\bigl(P^{(\infty)}\bigr)^{-1}.
\]
Thus the conjugation is $P^{(\infty)}J_T(P^{(\infty)})^{-1}$, \emph{not}
$(P^{(\infty)})^{-1}J_TP^{(\infty)}$.

\emph{Circles.}  On a \emph{clockwise} circle the left of the direction
of travel is the exterior: travelling clockwise, at the top of the
circle one moves in the $+x$ direction and the left is $+y$, i.e.\
away from the centre.  Hence the $+$ side is the exterior and the $-$
side the interior, so
$R_+=T(P^{(\infty)})^{-1}$ and $R_-=T(P^{(1)})^{-1}$ on $\partial D_1$,
and
\[
 I+W=R_-^{-1}R_+
 =P^{(1)}T^{-1}\,T\bigl(P^{(\infty)}\bigr)^{-1}
 =P^{(1)}\bigl(P^{(\infty)}\bigr)^{-1},
\]
which by \eqref{eq:matching} is $I+\Delta_++E_+$.  The same on
$\partial D_{-1}$ with \eqref{eq:matchingminus}.
\end{proof}

\begin{remark}[the sign is fixed]\label{rem:signfixed}
Once the circles are declared clockwise, $W=\Delta_\pm+E_\pm$ is forced;
the exterior is the $+$ side, and the lens jump is conjugated as
$P^{(\infty)}J_T(P^{(\infty)})^{-1}$.  These conventions determine the
sign of $\Delta_\pm$.
\end{remark}

\begin{lemma}[small norm]\label{lem:smallnorm}
There are $C_A$ and $s_A\ge5$ such that for $s\ge s_A$ and
$0<\omega\le A\ln s$,
\begin{equation}\label{eq:Wbounds}
 \|W\|_{L^\infty(\Sigma_R)}\le C_A\frac{(1+\omega)^2}{s\,r_s}
 =O_A\Bigl(\frac{\ln^4s}{s}\Bigr),
 \qquad
 \|W\|_{L^2(\Sigma_R)}\le C_A\frac{(1+\omega)^2}{s\sqrt{r_s}} .
\end{equation}
\end{lemma}

\begin{proof}
On the circles use \eqref{eq:matchbounds}; on the lens components use
Lemma~\ref{lem:lensnegligible}, which bounds exactly the quantity
$\|P^{(\infty)}(J_T-I)(P^{(\infty)})^{-1}\|=\|W\|$ appearing in
\eqref{eq:W} by $O_{A,N}(s^{-N})$.  For $L^2$: the circles have total
length $4\pi r_s$ and the remaining contour length at most $10$, so
$\|W\|_{L^2}^2\le4\pi r_s\|W\|^2_{L^\infty(\text{circles})}+10\,
O_{A,N}(s^{-2N})\le C_Ar_s(1+\omega)^4/(sr_s)^2$.
\end{proof}

\begin{lemma}[uniform Cauchy constants and Neumann inversion]
\label{lem:cauchy}
Let $C_-^{\Sigma_R}$ be the minus Cauchy projection on
$L^2(\Sigma_R)$.  There is an absolute constant $C_\ast$ such that
$\|C_-^{\Sigma_R}\|_{L^2\to L^2}\le C_\ast$ for all $s\ge5$, uniformly
as $r_s\downarrow0$.  Consequently, for $s\ge s_A$ the operator
$\mathcal C_W\mu:=C_-^{\Sigma_R}(\mu W)$ has
$\|\mathcal C_W\|\le C_\ast\|W\|_{L^\infty}\le\tfrac12$, $I-\mathcal C_W$
is invertible by its Neumann series, the equation
$\mu=I+\mathcal C_W\mu$ has a unique solution
$\mu\in I+L^2(\Sigma_R)$ with
\begin{equation}\label{eq:mubound}
 \|\mu-I\|_{L^2(\Sigma_R)}\le2C_\ast\|W\|_{L^2(\Sigma_R)},
\end{equation}
and $R$ exists, is unique, and
\begin{equation}\label{eq:Rrep}
 R(\lambda)=I+\frac1{2\pi i}\int_{\Sigma_R}
 \frac{\mu(z)W(z)}{z-\lambda}\,dz
 =I+\frac{R_1}\lambda+O(\lambda^{-2}),
 \qquad
 R_1=-\frac1{2\pi i}\int_{\Sigma_R}\mu(z)W(z)\,dz .
\end{equation}
\end{lemma}

\begin{proof}
\emph{Uniformity of the Cauchy constant.}  By
Lemma~\ref{lem:geometry}(vi), $\Sigma_R$ is a finite union of line
segments and two circles, meeting only at right angles, with total
length at most $10+4\pi$ and with
$\operatorname{dist}(\partial D_1,\partial D_{-1})\ge1$.  Split
$C_-^{\Sigma_R}=\sum_{i,j}\chi_iC_-\chi_j$ over the components.  Each
diagonal term is invariant under the dilation
$\lambda\mapsto\pm1+(\lambda\mp1)/r_s$ about the corresponding endpoint,
because the Cauchy kernel $dz/(z-\lambda)$ is scale invariant; that
dilation carries $\partial D_{\pm1}$ to the unit circle and the four
adjacent segments to segments of length $\asymp1/r_s$ meeting it
orthogonally, a configuration of Lipschitz arcs with $s$-independent
Lipschitz character, so Coifman--McIntosh--Meyer \cite{CMM82} gives an
$s$-independent bound.  The cited theorem supplies the $L^2$ estimate for
each Lipschitz contour; uniformity in this shrinking family comes from
the preceding scale invariance and the $s$-independent Lipschitz and
angle bounds, not from a separate uniform assertion in \cite{CMM82}.
Off-diagonal terms between components meeting at
a right angle are likewise scale invariant about the crossing point with
a fixed angle; the remaining off-diagonal terms have smooth kernels
bounded by $\operatorname{dist}^{-1}\le1$.  Hence
$\|C_-^{\Sigma_R}\|\le C_\ast$ absolute.

\emph{Neumann series.}  $\|\mathcal C_W\mu\|_{L^2}\le C_\ast
\|W\|_{L^\infty}\|\mu\|_{L^2}$ and $\|W\|_{L^\infty}=o(1)$ by
Lemma~\ref{lem:smallnorm}, so $\|\mathcal C_W\|\le\tfrac12$ for
$s\ge s_A$; then $\|(I-\mathcal C_W)^{-1}\|\le2$ and
$\mu-I=(I-\mathcal C_W)^{-1}C_-^{\Sigma_R}W$, giving
\eqref{eq:mubound}.  This is the small-norm theorem in the form of
\cite{DeiftZhou}; see also \cite[Ch.~7]{Deift99}.

\emph{Representation.}  With $R_+=R_-(I+W)$ and
$R=I+C(\mu W)$, Plemelj gives $R_+-R_-=\mu W$, and $R_+-R_-=R_-W$ forces
$\mu=R_-$.  Expanding $\frac1{z-\lambda}=-\frac1\lambda+O(\lambda^{-2})$
gives \eqref{eq:Rrep}; the sign of $R_1$ comes from this expansion, not
from any convention about the circles.
\end{proof}

\subsection{From \texorpdfstring{$R_1$}{R1} to the determinant}
\label{sec:gate11}

\begin{lemma}[the $(1,1)$ entry of $R_1$]\label{lem:R1}
Uniformly for $s\ge s_A$, $0<\omega\le A\ln s$,
\begin{equation}\label{eq:R1answer}
 (R_1)_{11}
 =-\frac{i\nu^2}{s}
 +O_A\!\Bigl(\frac{(1+\omega)^4}{s^2r_s}\Bigr)
 =\frac{i\omega^2}{\pi^2s}
 +O_A\!\Bigl(\frac{(1+\omega)^4\ln^2s}{s^2}\Bigr).
\end{equation}
\end{lemma}

\begin{proof}
Split $-2\pi iR_1=\int_{\Sigma_R}W+\int_{\Sigma_R}(\mu-I)W$.
By Cauchy--Schwarz and \eqref{eq:mubound},
$|\int(\mu-I)W|\le2C_\ast\|W\|_{L^2}^2\le C_A(1+\omega)^4/(s^2r_s)$.
In $\int W$, the lens components contribute $O_{A,N}(s^{-N})$
(Lemma~\ref{lem:lensnegligible}); on the circles $W=\Delta_\pm+E_\pm$
by Lemma~\ref{lem:Rjumps}, and the $E_\pm$ contribute at most
$4\pi r_s\cdot C_A(1+\omega)^4/(sr_s)^2=C_A(1+\omega)^4/(s^2r_s)$.
There remains
$\int_{\partial D_1}\Delta_++\int_{\partial D_{-1}}\Delta_-$.  Both
integrands are meromorphic in the respective closed disks with a single
simple pole at the centre --- the factors $\beta^{\pm2}$,
$\tilde\beta^{\pm2}$ are analytic and nonvanishing there --- so, the
circles being \emph{clockwise},
\[
 \int_{\partial D_1}(\Delta_+)_{11}\,d\lambda
 =-2\pi i\Res_{\lambda=1}(\Delta_+)_{11},
 \qquad
 \int_{\partial D_{-1}}(\Delta_-)_{11}\,d\lambda
 =-2\pi i\Res_{\lambda=-1}(\Delta_-)_{11},
\]
and by \eqref{eq:residues} both residues equal $-i\nu^2/(2s)$.  Hence
\[
 (R_1)_{11}=-\frac1{2\pi i}\Bigl[-2\pi i\cdot2\cdot
 \Bigl(-\frac{i\nu^2}{2s}\Bigr)\Bigr]
 +O_A\Bigl(\frac{(1+\omega)^4}{s^2r_s}\Bigr)
 =-\frac{i\nu^2}{s}+O_A\Bigl(\frac{(1+\omega)^4}{s^2r_s}\Bigr),
\]
and $\nu^2=-\omega^2/\pi^2$, $1/r_s=\ln^2s$.
\end{proof}

\begin{lemma}[transfer to $Y_1$]\label{lem:Y1}
Uniformly for $s\ge s_A$, $0<\omega\le A\ln s$,
\begin{equation}\label{eq:Y1answer}
 (Y_1)_{11}=-\frac{2iv}{\pi}+\frac{iv^2}{\pi^2s}
 +O_A\!\Bigl(\frac{(1+\omega)^4\ln^2s}{s^2}\Bigr).
\end{equation}
\end{lemma}

\begin{proof}
For $|\lambda|>2$ we are outside every lens and disk, so by
\eqref{eq:Sdef} and \eqref{eq:Rdef}, $Y=T=RP^{(\infty)}$ there.
Multiplying \eqref{eq:Rrep} and \eqref{eq:Pinfinfty},
\[
 Y=\Bigl(I+\frac{R_1}\lambda+O(\lambda^{-2})\Bigr)
 \Bigl(I-\frac{2\nu\sigma_3}\lambda+O(\lambda^{-2})\Bigr)
 =I+\frac{R_1-2\nu\sigma_3}\lambda+O(\lambda^{-2}),
\]
so $Y_1=R_1-2\nu\sigma_3$ and $(Y_1)_{11}=(R_1)_{11}-2\nu
=(R_1)_{11}-2iv/\pi$.  Since $\nu^2=-v^2/\pi^2$,
$(R_1)_{11}=-i\nu^2/s+\cdots=iv^2/(\pi^2s)+\cdots$.
\end{proof}

\begin{lemma}[differential identity]\label{lem:diffid}
For every $s>0$ and $\gamma\le0$,
\begin{equation}\label{eq:diffid}
 \frac{\partial}{\partial s}\ln\det(I-\gamma K_s)
 =-i\bigl((Y_1)_{11}-(Y_1)_{22}\bigr)
 =-2i\,(Y_1)_{11},
\end{equation}
the derivative being taken at fixed $\gamma$.
\end{lemma}

\begin{proof}
Put
\[
 R_s:=(I-\gamma K_s)^{-1},
 \qquad
 h_\pm(x):=e^{\pm isx}.
\]
Lemma~\ref{lem:tracenorm} and \eqref{eq:jacobi} give
\[
 \frac{\partial}{\partial s}\ln\det(I-\gamma K_s)
 =-\gamma\Tr(R_s\partial_sK_s).
\]
In the rank-one notation of Lemma~\ref{lem:tracenorm},
\[
 \partial_sK_s
 =\frac1{2\pi}\bigl(h_+\otimes h_+ + h_-\otimes h_-\bigr).
\]
The reflection $\mathcal Jq(x)=q(-x)$ commutes with $K_s$ and hence
with $R_s$, while $\mathcal Jh_+=h_-$.  Consequently
\[
 Q:=\langle h_+,R_sh_+\rangle
   =\langle h_-,R_sh_-\rangle .
\]
The trace of $R_s(h_\pm\otimes h_\pm)$ is
$\langle h_\pm,R_sh_\pm\rangle$, and therefore
\begin{equation}\label{eq:diffidQ}
 \frac{\partial}{\partial s}\ln\det(I-\gamma K_s)
 =-\frac{\gamma}{\pi}Q .
\end{equation}

From \eqref{eq:fg},
\[
 F_1=\frac{\gamma}{2\pi i}R_sh_+,
 \qquad
 F_2=\frac{\gamma}{2\pi i}R_sh_-,
 \qquad
 g^T=(-h_-,h_+).
\]
Using \eqref{eq:Y1formula} and
$\overline{h_\pm}=h_\mp$ gives
\[
 (Y_1)_{11}
 =\frac{\gamma}{2\pi i}
   \int_{-1}^1(R_sh_+)(x)h_-(x)\,dx
 =\frac{\gamma}{2\pi i}Q,
\]
and
\[
 (Y_1)_{22}
 =-\frac{\gamma}{2\pi i}
   \int_{-1}^1(R_sh_-)(x)h_+(x)\,dx
 =-\frac{\gamma}{2\pi i}Q.
\]
It follows directly that
\[
 -i\bigl((Y_1)_{11}-(Y_1)_{22}\bigr)
 =-2i(Y_1)_{11}
 =-\frac{\gamma}{\pi}Q,
\]
which agrees with \eqref{eq:diffidQ}.  No continuation in $\gamma$ or
external differential-identity formula is used.
\end{proof}

\begin{lemma}[the signed differential asymptotic]\label{lem:derivative}
Uniformly for $s\ge s_A$, $0<\omega\le A\ln s$,
\begin{equation}\label{eq:derivative}
 \frac{\partial}{\partial s}\ln\mathcal D(s,\omega)
 =\frac{4\omega}{\pi}+\frac{2\omega^2}{\pi^2s}
 +O_A\!\Bigl(\frac{(1+\omega)^4\ln^2s}{s^2}\Bigr).
\end{equation}
\end{lemma}

\begin{proof}
By \eqref{eq:diffid} and \eqref{eq:Y1answer}, with $v=-\omega$,
\[
 -2i(Y_1)_{11}
 =-2i\Bigl(-\frac{2iv}\pi+\frac{iv^2}{\pi^2s}\Bigr)+\cdots
 =-\frac{4v}\pi+\frac{2v^2}{\pi^2s}+\cdots
 =\frac{4\omega}\pi+\frac{2\omega^2}{\pi^2s}+\cdots. \qedhere
\]
\end{proof}

\begin{proof}[Proof of Theorem~\ref{thm:signedmain}]
The case $\omega=0$ was disposed of in \S\ref{sec:tail-statement}, so let
$0<\omega\le A\ln s$ with $s\ge s_A$.  Define, for $t\ge s$,
\begin{equation}\label{eq:Phidef}
 \Phi(t):=\ln\mathcal D(t,\omega)
 -\frac{4\omega t}{\pi}-\frac{2\omega^2}{\pi^2}\ln(4t)
 -2\ln\bigl[\BG(1+i\omega/\pi)\BG(1-i\omega/\pi)\bigr].
\end{equation}

\emph{Step 1: the hypothesis propagates upward.}  Since
$\omega\le A\ln s$ and $t\mapsto A\ln t$ is increasing,
$\omega\le A\ln t$ for every $t\ge s\ge s_A$.  Hence
Lemma~\ref{lem:derivative} applies at every $t\ge s$, with the same
$C_A$ and the same fixed $\omega$.

\emph{Step 2: the derivative of $\Phi$.}  Differentiating and using
\eqref{eq:derivative},
$\Phi'(t)=O_A\bigl((1+\omega)^4\ln^2t/t^2\bigr)$ for $t\ge s$.

\emph{Step 3: the limit at infinity vanishes.}  For our \emph{fixed}
$\omega$, Lemma~\ref{lem:charlierconvert} gives $\Phi(t)\to0$ as
$t\to\infty$.  Only the existence of the limit and its value are used;
the rate $O_\omega(t^{-1})$, whose constant is not uniform in $\omega$,
plays no role.  This is why the fixed-parameter equation
\cite[eq.~(1.4)]{Charlier21} may be used even though it supplies no
growing-$\omega$ uniformity.

\emph{Step 4: the elementary integral.}  For $s\ge3$,
\begin{equation}\label{eq:elemintegral}
 \int_s^\infty\frac{\ln^2t}{t^2}\,dt
 =\Bigl[-\frac{\ln^2t+2\ln t+2}{t}\Bigr]_s^\infty
 =\frac{\ln^2s+2\ln s+2}{s}\;\le\;\frac{5\ln^2s}{s}.
\end{equation}
Indeed, with $x=\ln s\ge\ln3>1$,
$4x^2-2x-2=2(2x+1)(x-1)\ge0$, and hence
$2\ln s+2\le4\ln^2s$.  A coefficient $3$ in place of $5$ would require
$s\ge e^{(1+\sqrt5)/2}$ and is therefore unavailable on the full range
$s\ge3$.  The coefficient $5$ changes only the implicit constant below
and leaves the threshold $s_A\ge5$ unchanged.

\emph{Step 5: conclusion.}  By Steps 2 and 3,
$\Phi(s)=-\int_s^\infty\Phi'(t)\,dt$, so by Step 4
$|\Phi(s)|\le5C_A(1+\omega)^4\ln^2s/s$, which is
\eqref{eq:RHmain}--\eqref{eq:RHerr} after renaming $5C_A$ as $C_A$.
The right side of \eqref{eq:RHmain} is real, as is the left,
consistently with $\mathcal D\ge1$.

No claim is made here that the linear term $4\omega s/\pi$ dominates the
displayed error uniformly down to $\omega=0$, and no such claim is used
anywhere below.  Indeed, the ratio of the linear term to the bound
\eqref{eq:RHerr} is $4\omega s^2/\bigl(\pi C_A(1+\omega)^4\ln^2s\bigr)$,
which tends to $0$ as $\omega\downarrow0$ at fixed $s$; the linear term
dominates the displayed error bound only on ranges bounded away from
$\omega=0$, where for each fixed $\omega_0>0$ and $\omega\ge\omega_0$ the
ratio is $\gtrsim_{A,\omega_0}s^2/\ln^{6}s$.  At $\omega=0$ the
statement \eqref{eq:RHmain} is exact with $\mathcal R_A\equiv0$, as
recorded in \S\ref{sec:tail-statement}.
\end{proof}

\begin{remark}[why holding $\omega$ fixed gives a uniform result]
\label{rem:uniformity}
The constant $C_A$ in Lemma~\ref{lem:derivative} does not depend on
$\omega$ within $\omega\le A\ln t$, and Step 1 shows that range is
preserved as $t$ grows.  The only $\omega$-dependent input is the
\emph{value} of the limit in Step 3, which is exact.  No
$\omega$-uniform fixed-parameter theorem is needed or claimed.
\end{remark}

\begin{remark}[three consistency checks on the final chain]
\label{rem:finalchecks}
\emph{(i) Small $\gamma$.}  Lemma~\ref{lem:smallgamma} gives
$-2i(Y_1)_{11}=-2\gamma/\pi+O(\gamma^2)$, matching
$\partial_s\ln\det(I-\gamma K_s)=-2\gamma s/\pi\cdot\partial_s
=-2\gamma/\pi$.  Since $v=\gamma/2+O(\gamma^2)$, the steepest-descent
value $-4v/\pi$ equals $-2\gamma/\pi$ to first order: the two routes
agree.
\emph{(ii) Fixed parameter.}  At fixed $\omega$, \eqref{eq:RHmain}
reduces to Lemma~\ref{lem:charlierconvert}, which is how the
integration constant was fixed; this is a consistency requirement, not
an independent check.
\emph{(iii) The head-side formula.}  Formally replacing $\omega$ by
$-v$ in the three explicit main terms of \eqref{eq:RHmain} reproduces
\cite[Thm.~1.2]{BDIK2}.  This is only an algebraic coefficient check:
the present signed proof assumes $\gamma\le0$ and does not prove the
head-side estimate.
\end{remark}

\subsection{Transfer to \texorpdfstring{$S_c$}{S-c}}\label{sec:gate16}

\begin{proposition}[exact transfer and linear cancellation]
\label{prop:transfer}
Let $s=\pi c/2$ and $u=2\omega$.  Then
\begin{equation}\label{eq:traceidentity}
 \Tr\varphi_u(S_c)=F_c(-u/2)-uc
 =\ln\mathcal D\Bigl(\frac{\pi c}2,\frac u2\Bigr)-uc
\end{equation}
exactly, and
\begin{equation}\label{eq:linearcancel}
 \frac{4\omega s}{\pi}
 =\frac{4\cdot\frac u2\cdot\frac{\pi c}2}{\pi}=uc .
\end{equation}
Consequently, \emph{assuming Theorem~\ref{thm:signedmain}}, for every
$A>0$, uniformly for $0\le u\le2A\ln(\pi c/2)$ as $c\to\infty$,
\begin{equation}\label{eq:traceasym}
 \Tr\varphi_u(S_c)
 =\frac{u^2}{2\pi^2}\ln(2\pi c)
 +2\ln\Bigl[\BG\Bigl(1+\frac{iu}{2\pi}\Bigr)
 \BG\Bigl(1-\frac{iu}{2\pi}\Bigr)\Bigr]
 +O_A\!\Bigl(\frac{(1+u)^4\ln^2c}{c}\Bigr).
\end{equation}
\end{proposition}

\begin{proof}
By Lemma~\ref{lem:dict}, $S_c\simeq K_{\pi c/2}$, so
$\det(I+(e^u-1)S_c)=\mathcal D(\pi c/2,u/2)$ with $u=2\omega$; and
$\Tr\varphi_u(S_c)=\ln\det(I+(e^u-1)S_c)-u\Tr S_c$ by
\eqref{eq:phitrace} with $\Tr S_c=c$.  This is
\eqref{eq:traceidentity}, and \eqref{eq:linearcancel} is arithmetic.
For \eqref{eq:traceasym} substitute $s=\pi c/2$, $\omega=u/2$ into
\eqref{eq:RHmain} and subtract $uc$: the linear terms cancel by
\eqref{eq:linearcancel}; $4s=2\pi c$ gives
$\frac{2(u/2)^2}{\pi^2}\ln(2\pi c)=\frac{u^2}{2\pi^2}\ln(2\pi c)$;
$\frac{i\omega}\pi=\frac{iu}{2\pi}$ gives the Barnes term; and the error
is $C_A(1+u/2)^4\ln^2(\pi c/2)/(\pi c/2)\le C_A'(1+u)^4\ln^2c/c$ for
$c\ge3$.
\end{proof}

\begin{proof}[Proof of Corollary~\ref{cor:Tproved}]
Fix $\alpha>0$, take $A=\alpha/2$.  Since $\pi/2>1$,
$\ln(\pi c/2)>\ln c$ for $c>1$, so $2\le u\le\alpha\ln c$ lies inside
$0\le u\le\alpha\ln(\pi c/2)$.  With $C_A$ the constant of
\eqref{eq:traceasym} and Lemma~\ref{lem:Barneslower},
\begin{equation}\label{eq:Tchain}
 \Tr\varphi_u(S_c)
 \;\ge\;\frac{u^2\ln c}{\pi^2}
 \Bigl[\frac12-\frac{\ln(2+u)}{\ln c}
 -\frac{\pi^2C_A(1+u)^4\ln c}{c\,u^2}\Bigr].
\end{equation}
The Barnes ratio obeys
$\ln(2+u)/\ln c\le\ln(2+\alpha\ln c)/\ln c\to0$, a bound independent of
$u$, so it is $\le\tfrac18$ for $c\ge c_1(\alpha)$ and all $u$ in range.
For $u\ge2$, $(1+u)^4/u^2\le\tfrac{81}{16}u^2
\le\tfrac{81}{16}\alpha^2\ln^2c$, so the remainder ratio is at most
$\tfrac{81\pi^2C_A\alpha^2}{16}\ln^3c/c\to0$, again independent of $u$,
hence $\le\tfrac18$ for $c\ge c_2(\alpha)$.  With
$c_T(\alpha)=\max\{c_1(\alpha),c_2(\alpha),(2/\pi)s_{\alpha/2},3\}$,
\eqref{eq:Tchain} gives
$\Tr\varphi_u(S_c)\ge(\tfrac12-\tfrac18-\tfrac18)\frac{u^2\ln c}{\pi^2}
=\frac{u^2}{4\pi^2}\ln c$, i.e.\ \eqref{eq:Tproved} with
\begin{equation}\label{eq:constants}
 \kappa_T=\frac1{4\pi^2},\qquad L_T=2,
\end{equation}
for every fixed $\alpha>0$; only $c_T$ depends on $\alpha$.
\end{proof}

\begin{remark}[the threshold]\label{rem:threshold}
$c_1(\alpha)$ is the least $c\ge3$ with $8\ln(2+\alpha\ln c)\le\ln c$;
$c_2(\alpha)$ the least $c\ge3$ with
$81\pi^2C_A\alpha^2\ln^3c\le2c$; $s_{\alpha/2}$ the threshold of
Theorem~\ref{thm:signedmain}.  Each is a well-defined finite number.
No numerical value is attempted: $C_A$ inherits the unquantified
constants $C_\ast$ of Lemma~\ref{lem:cauchy} and the DLMF constants of
Lemma~\ref{lem:uniformmatching}.
\end{remark}

\section{Tail-side quantiles and the one-tail tensor
consequence}\label{sec:tailquant}

This section proves Theorem~\ref{thm:tail}, Corollary~\ref{cor:plunge},
Theorems~\ref{thm:tensorblock} and~\ref{thm:surface}, and
Lemma~\ref{lem:lambert}.  The plan mirrors
\S\S\ref{sec:smoothed}--\ref{sec:thmA} on the negative half of the
determinant parameter, with Theorem~\ref{thm:signedmain} in the role
played there by Theorem~\ref{qt:BDIK}.  Only the four ingredients that
are genuinely sign-sensitive are redone: the derivative identity
(\S\ref{sec:tq-transfer}), the unit-average formula and the local
density (\S\ref{sec:tq-averages}), the exponential tails
(\S\ref{sec:tq-tails}), and the resulting preliminary count estimate
(\S\ref{sec:tq-preform}).  Everything else --- the scaling dictionary
of Lemma~\ref{lem:dict}, the Barnes product and its derivative bounds in
\S\ref{sec:barnes}, the sigmoid inequality \eqref{eq:sigdiff}, the
monotonicity bracket \eqref{eq:avg-bracket}, the conversion
Lemma~\ref{lem:conversion}, the head-side Theorem~\ref{thm:A}, the
explicit count of Theorem~\ref{qt:KRD} and the tensor identity of
Lemma~\ref{lem:tensor} --- is used by internal reference and is not
restated or reproved.

\paragraph{Standing conventions for \S\ref{sec:tailquant}.}
$A>0$ (and, where they occur, the integer $d\ge1$ and $q\in(\frac12,1)$)
are fixed \emph{before} any constant is chosen; $C_A,c_A,C_{A,d,q},\dots$
denote positive finite constants depending only on the indicated fixed
parameters, whose value may change from occurrence to occurrence, and
$C$ denotes an absolute constant.  No statement below is uniform in $A$,
$d$ or $q$.  We use throughout the tail depth
\begin{equation}\label{eq:tq-depth}
 \omega=-v\ge0,\qquad \Lb=2\omega,\qquad
 \delta=\theta(-\omega)=\frac1{1+e^{2\omega}} ,
\end{equation}
so that, by \eqref{eq:logodds}, $N_\delta(c)=\Ntil(-\omega)$ exactly,
with the paper's strict-$>$ convention, and $\Lb=\ln\frac{1-\delta}\delta$.

\subsection{The determinant identity and the signed expansion at
negative depth}\label{sec:tq-transfer}

Lemma~\ref{lem:Fprime} states the derivative identity for $v\ge0$.  The
extension to the whole line is immediate and is the only place where the
positivity input has to be re-examined.

\begin{lemma}[derivative identity on the whole line]\label{lem:tq-fprime}
For every $v\in\R$ the series \eqref{eq:Mdef} converges, $M_c$ is
continuous and strictly decreasing on $\R$, $F_c$ of \eqref{eq:Fdef} is
differentiable at $v$, and
\begin{equation}\label{eq:tq-fprime}
 F_c'(v)=-2M_c(v).
\end{equation}
Consequently, for all real $v_1\le v_2$,
\begin{equation}\label{eq:tq-fint}
 \int_{v_1}^{v_2}M_c(t)\dd t=\frac{F_c(v_1)-F_c(v_2)}{2}.
\end{equation}
\end{lemma}

\begin{proof}
Let $\gamma=\gamma(v)=1-e^{-2v}<1$.  The function
$\lambda\mapsto1-\gamma\lambda$ is affine on $[0,1]$ with values $1$ at
$\lambda=0$ and $e^{-2v}$ at $\lambda=1$, so
\begin{equation}\label{eq:tq-posit}
 1-\gamma\lambda_n\;\ge\;\min\bigl(1,e^{-2v}\bigr)\;>\;0
 \qquad\text{for every }n ,
\end{equation}
and $I-\gamma S_c$ is positive definite and boundedly invertible for
every real $v$ --- for $v\ge0$ because $1-\gamma\lambda\ge e^{-2v}$, and
for $v\le0$ because $\gamma\le0$ gives $1-\gamma\lambda\ge1$.  Hence
\[
 \sum_n\frac{\lambda_n}{1-\gamma\lambda_n}
 \;\le\;\max\bigl(1,e^{2v}\bigr)\sum_n\lambda_n
 \;=\;c\,\max\bigl(1,e^{2v}\bigr)\;<\;\infty ,
\]
locally uniformly in $v$, and the computation in the proof of
Lemma~\ref{lem:Fprime} applies verbatim with \eqref{eq:tq-posit} in place
of the bound $1-\gamma\lambda_n\ge e^{-2v}$ used there; this gives
\eqref{eq:tq-fprime}.  Each summand of \eqref{eq:Mdef} is continuous and
strictly decreasing in $v$, and the series converges locally uniformly by
the same majorant, so $M_c$ is continuous and strictly decreasing.
Then $F_c\in C^1(\R)$ and \eqref{eq:tq-fint} is the fundamental theorem
of calculus.
\end{proof}

\begin{lemma}[the signed expansion in the $c$ variable]
\label{lem:tq-Fneg}
Fix $B>0$.  There are $c_B,C_B<\infty$ such that for all $c\ge c_B$ and
all $v$ with $-B\ln c\le v\le0$,
\begin{equation}\label{eq:tq-Fneg}
 F_c(v)=-2vc+\frac{2v^2}{\pi^2}\ln(2\pi c)+A(v)+r_c(v),
 \qquad
 |r_c(v)|\;\le\;C_B\,\frac{(1+|v|)^4\ln^2c}{c}.
\end{equation}
That is, the head-side expansion \eqref{eq:F-asym} continues to hold
verbatim at negative depth, with the head-side remainder bound
\eqref{eq:F-head} replaced by the displayed one.
\end{lemma}

\begin{proof}
Put $\omega=-v\in[0,B\ln c]$.  By \eqref{eq:gammav},
$\gamma(v)=1-e^{2\omega}$ and
$F_c(v)=\ln\det\bigl(I+(e^{2\omega}-1)S_c\bigr)$, which by
Lemma~\ref{lem:dict} equals $\ln\mathcal D(\pi c/2,\omega)$ with
$\mathcal D$ as in \eqref{eq:Ddef}; this is the identity already recorded
in \eqref{eq:traceidentity}.  Apply Theorem~\ref{thm:signedmain} with
its fixed parameter equal to $B$ and $s=\pi c/2$.  Its hypotheses hold:
$s\ge s_B$ once $c\ge2s_B/\pi$, and since $\pi/2>1$ we have
$\ln s=\ln(\pi c/2)>\ln c$ for $c>1$, so $\omega\le B\ln c$ implies
$\omega\le B\ln s$.  Therefore
\[
 F_c(v)=\frac{4\omega s}{\pi}+\frac{2\omega^2}{\pi^2}\ln(4s)
 +2\ln\bigl[\BG(1+i\omega/\pi)\BG(1-i\omega/\pi)\bigr]
 +\mathcal R_B(s,\omega).
\]
Now $\frac{4\omega s}\pi=2\omega c=-2vc$ by \eqref{eq:linearcancel},
$4s=2\pi c$, $\omega^2=v^2$, and the Barnes term is $A(\omega)=A(-v)=A(v)$
by \eqref{eq:Adef} and the evenness of $A$ (Lemma~\ref{lem:barnes1}).
Finally, by \eqref{eq:RHerr},
\[
 |\mathcal R_B(s,\omega)|
 \le C_B\frac{(1+\omega)^4\ln^2(\pi c/2)}{\pi c/2}
 \le C_B'\,\frac{(1+|v|)^4\ln^2c}{c}
 \qquad(c\ge3),
\]
exactly as in the proof of Proposition~\ref{prop:transfer}.
\end{proof}

\begin{remark}[one formula, two halves]\label{rem:tq-onefla}
Lemma~\ref{lem:tq-Fneg} and \eqref{eq:F-asym} are the same display on the
two halves of the domain of $F_c$; only the remainder bookkeeping
differs, being $O((v+v^3)/c)$ on the head side and
$O_B((1+|v|)^4\ln^2c/c)$ here.  This is the analytic content of the
two-sided reading of the generating parameter announced in
\S\ref{sec:dictionary}, and it is why the negative-side counting
argument below is a mirror of \S\ref{sec:thmA} rather than a new
argument.  It is \emph{not} obtained by continuing \eqref{eq:F-asym} in
$v$: the negative-depth statement is Theorem~\ref{thm:signedmain},
proved from scratch in \S\ref{sec:tail}; see
Remark~\ref{rem:bdik-scope}(i).
\end{remark}

Two elementary consequences of Lemma~\ref{lem:barnes1} are used
repeatedly and are recorded once.  Since $A$ is even and
$C^\infty(\R)$, $A'$ is odd and $A''$ is even; hence \eqref{eq:A2}
extends from $v\ge0$ to the whole line:
\begin{equation}\label{eq:tq-A2all}
 A'(-x)=-A'(x)\quad(x\in\R),
 \qquad
 |A''(v)|\;\le\;\ln(2+|v|)+4\quad(v\in\R).
\end{equation}

\subsection{Negative-side unit averages and local
density}\label{sec:tq-averages}

\begin{lemma}[negative-side average formula]\label{lem:tq-avgneg}
Fix $B>0$.  There are $c_B,C_B<\infty$ such that for $c\ge c_B$ and
$-B\ln c\le v_1<v_2\le0$,
\begin{equation}\label{eq:tq-avgneg}
 \int_{v_1}^{v_2}M_c
 = c\,(v_2-v_1)-\frac{v_2^2-v_1^2}{\pi^2}\ln(2\pi c)
 -\frac{A(v_2)-A(v_1)}{2}+E,
 \qquad
 |E|\le C_B\,\frac{(1+|v_1|)^4\ln^2c}{c}.
\end{equation}
\end{lemma}

\begin{proof}
Insert \eqref{eq:tq-Fneg} at $v_1$ and at $v_2$ into \eqref{eq:tq-fint};
the two remainders are each at most
$C_B(1+|v_1|)^4\ln^2c/c$ because $|v_2|\le|v_1|$.
\end{proof}

The bracket \eqref{eq:avg-bracket} is pure monotonicity of $M_c$ and
therefore holds for \emph{every} real $x$, with
$\overline M_c(x)=\int_x^{x+1}M_c$.  Taking $v_1=x$, $v_2=x+1$ in
\eqref{eq:tq-avgneg} gives, for a fixed $B>0$, all $c\ge c_B$ and all
$x$ with $-B\ln c\le x\le-1$,
\begin{equation}\label{eq:tq-Mbarneg}
 \overline M_c(x)
 = c-\frac{2x+1}{\pi^2}\ln(2\pi c)-\frac{A(x+1)-A(x)}{2}+E_x,
 \qquad
 |E_x|\le C_B\,\frac{(1+|x|)^4\ln^2c}{c},
\end{equation}
which is the display used in the proofs of Lemma~\ref{lem:density}(ii)
and Lemma~\ref{lem:tails}, now on the negative axis.  Note that the
right-hand side of \eqref{eq:tq-Mbarneg} is the same expression as on the
head side; only the admissible range of $x$ and the size of $E_x$ differ.
Note also that on a logarithmic window the remainder is harmless:
if $|x|\le B\ln c$ then
\begin{equation}\label{eq:tq-Eharmless}
 |E_x|\;\le\;C_B\frac{(1+B\ln c)^4\ln^2c}{c}
 \;=\;O_B\!\Bigl(\frac{\ln^6c}{c}\Bigr)\;\le\;1
\end{equation}
for all $c$ large enough in terms of $B$.

\begin{proposition}[negative-side local density]\label{prop:tq-density}
Fix $B>0$.  There are $c_B,C_B<\infty$ such that, with
$D_{\mathrm{loc}}(a)=\#\{n:a<w_n\le a+1\}$ as in
Lemma~\ref{lem:density},
\begin{equation}\label{eq:tq-density}
 D_{\mathrm{loc}}(a)\;\le\;C_B\ln c
 \qquad\text{for } -B\ln c\le a\le0 \text{ and } c\ge c_B .
\end{equation}
\end{proposition}

\begin{proof}
For $-2\le a\le0$ this is Lemma~\ref{lem:density}(i), whose bound is
absolute there.  Let $-B\ln c\le a\le-2$ and fix the window $B+1$ in
Lemma~\ref{lem:tq-avgneg} before letting $c$ and $a$ vary.  As in the
proof of Lemma~\ref{lem:density}(ii), every $n$ with $w_n\in(a,a+1]$
contributes at least
$\sigma_{\mathrm s}(2)-\sigma_{\mathrm s}(0)>\tfrac13$ to
$M_c(a)-M_c(a+1)$, and every other $n$ contributes a nonnegative amount,
so by \eqref{eq:avg-bracket}
\[
 \tfrac13D_{\mathrm{loc}}(a)\;\le\;M_c(a)-M_c(a+1)
 \;\le\;\overline M_c(a-1)-\overline M_c(a+1).
\]
Both $a-1$ and $a+1$ lie in $[-(B+1)\ln c,-1]$ for $c\ge e$, so
\eqref{eq:tq-Mbarneg} applies at both.  The terms $c$ cancel and the
coefficient of $\ln(2\pi c)$ is
$\pi^{-2}\bigl[(2(a+1)+1)-(2(a-1)+1)\bigr]=4\pi^{-2}$:
\[
 \overline M_c(a-1)-\overline M_c(a+1)
 =\frac4{\pi^2}\ln(2\pi c)+\Theta_A+\Theta_r ,
\]
where
$\Theta_A=\tfrac12\int_0^1\bigl[A'(a+1+t)-A'(a-1+t)\bigr]\dd t$
satisfies $|\Theta_A|\le\sup_{[a-1,a+2]}|A''|\le\ln(4+|a|)+4$ by
\eqref{eq:tq-A2all}, and $|\Theta_r|\le2$ by
\eqref{eq:tq-Eharmless}.  Since $|a|\le B\ln c$ we have
$\ln(4+|a|)\le\ln c$ for $c$ large, and \eqref{eq:tq-density} follows
with any $C_B>3\bigl(\tfrac4{\pi^2}+1\bigr)$ after enlarging $c_B$.
\end{proof}

Proposition~\ref{prop:tq-density} is recorded for completeness and as an
independent check; it is not used in the proof of
Theorem~\ref{thm:tail}, which controls the negative-side tails directly
through Lemma~\ref{lem:tq-tails} below.

\subsection{Negative-side exponential tails}\label{sec:tq-tails}

Recall $T_+(v)=\sum_{w_n>v}e^{-2(w_n-v)}$ and
$T_-(v)=\sum_{w_n\le v}e^{2(w_n-v)}$ from Lemma~\ref{lem:tails}.
The sum $T_+(v)$ is finite because $w_n>v$ is equivalent to
$\lambda_n>\theta(v)>0$, and $\sum_n\lambda_n=c$; moreover
$T_-(v)\le e^{-2v}\sum_n\frac{\lambda_n}{1-\lambda_n}<\infty$.
The threshold $w_n=v$ belongs
to $T_-$, matching the strict convention in
$\Ntil(v)=\#\{n:w_n>v\}$.

\begin{lemma}[negative-side tail bound]\label{lem:tq-tails}
Fix $B>0$.  There are $c_B,C_B<\infty$ such that
\begin{equation}\label{eq:tq-tailbound}
 T_+(v)+T_-(v)\;\le\;C_B\ln c
 \qquad\text{for }-B\ln c\le v\le-2\text{ and }c\ge c_B .
\end{equation}
\end{lemma}

\begin{proof}
Step~1 of the proof of Lemma~\ref{lem:tails} is free of any restriction
on $v$: applying \eqref{eq:sigdiff} with $h=1$ and $t=2(w_n-v)$ and
summing over all $n$ gives \eqref{eq:tailsdagger}, that is,
\[
 T_+(v)+T_-(v)\;\le\;\tfrac43\bigl[M_c(v-1)-M_c(v+1)\bigr]
 \qquad(v\in\R),
\]
and, by \eqref{eq:avg-bracket} applied at $x=v-2$ and at $x=v+1$,
\[
 M_c(v-1)-M_c(v+1)\;\le\;\overline M_c(v-2)-\overline M_c(v+1).
\]
Fix the window $B+1$ in \eqref{eq:tq-Mbarneg} before letting $c$ and
$v$ vary.  For $-B\ln c\le v\le-2$ and $c\ge e^2$, both $v-2$ and $v+1$
lie in $[-(B+1)\ln c,-1]$, so \eqref{eq:tq-Mbarneg} applies at both, the
terms $c$ cancel, and the coefficient of $\ln(2\pi c)$ is
$\pi^{-2}\bigl[(2(v+1)+1)-(2(v-2)+1)\bigr]=6\pi^{-2}$, independent of
$v$:
\[
 \overline M_c(v-2)-\overline M_c(v+1)
 =\frac6{\pi^2}\ln(2\pi c)
 -\tfrac12\bigl[\bigl(A(v-1)-A(v-2)\bigr)
 -\bigl(A(v+2)-A(v+1)\bigr)\bigr]+E_{v-2}-E_{v+1}.
\]
By the mean value theorem the bracket equals $A'(\xi_1)-A'(\xi_2)$ with
$\xi_1\in(v-2,v-1)$ and $\xi_2\in(v+1,v+2)$, so $|\xi_1-\xi_2|\le4$ and,
by \eqref{eq:tq-A2all},
$\tfrac12|A'(\xi_1)-A'(\xi_2)|\le2\sup_{[v-2,v+2]}|A''|
\le2\bigl(\ln(4+|v|)+4\bigr)$.
Since $|v|\le B\ln c$, $\ln(4+|v|)\le\ln c$ for $c$ large, and
$|E_{v-2}|+|E_{v+1}|\le2$ by \eqref{eq:tq-Eharmless}.  Hence
\[
 T_+(v)+T_-(v)\;\le\;\tfrac43\Bigl[\frac6{\pi^2}\ln(2\pi c)
 +2\ln c+8+2\Bigr]\;\le\;C_B\ln c
\]
for $c\ge c_B$, both summands being nonnegative.
\end{proof}

\subsection{A preliminary negative-side count estimate}
\label{sec:tq-preform}

\begin{proposition}[preliminary negative-side count estimate]
\label{prop:tq-preform}
Fix $A>0$.  There are $c_A,C_A<\infty$ such that, uniformly for
$c\ge c_A$ and $3\le\omega\le A\ln c$,
\begin{equation}\label{eq:tq-preform}
 \Ntil(-\omega)
 = c+\frac{2\omega}{\pi^2}\ln(2\pi c)+\frac{A'(\omega)}{2}
 +O_A(\ln c)
\end{equation}
and therefore
\begin{equation}\label{eq:tq-preform2}
 \Ntil(-\omega)
 = c+\frac{2\omega}{\pi^2}\ln(2\pi c)
 -\frac{\omega}{\pi^2}\ln\Bigl(1+\frac{\omega^2}{\pi^2}\Bigr)
 +O_A(\ln c+\omega).
\end{equation}
\end{proposition}

\begin{proof}
Fix the window $B=A+2$ in \eqref{eq:tq-Mbarneg} and in
Lemma~\ref{lem:tq-tails} before letting $c$ and $\omega$ vary.

\emph{Step 1: two-sided smoothed bounds.}  Lemma~\ref{lem:conversion}
holds for every real $v$ and every $h>0$; take $v=-\omega$ and $h=1$:
\[
 M_c(-\omega+1)-e^{-2}T_-(-\omega)\;\le\;\Ntil(-\omega)\;\le\;
 M_c(-\omega-1)+e^{-2}T_+(-\omega).
\]
By \eqref{eq:avg-bracket}, $\overline M_c(-\omega+1)\le M_c(-\omega+1)$
and $M_c(-\omega-1)\le\overline M_c(-\omega-2)$, so
\[
 \overline M_c(-\omega+1)-e^{-2}T_-(-\omega)\;\le\;\Ntil(-\omega)\;\le\;
 \overline M_c(-\omega-2)+e^{-2}T_+(-\omega).
\]
Since $3\le\omega\le A\ln c$, the depth $v=-\omega$ satisfies
$-A\ln c\le v\le-3\le-2$, so Lemma~\ref{lem:tq-tails} bounds both tail
corrections by $C_A\ln c$.

\emph{Step 2: evaluating the two averages.}  Both $x=-\omega+1$ and
$x=-\omega-2$ satisfy $-(A+2)\ln c\le x\le-2$, so
\eqref{eq:tq-Mbarneg} applies to each, with $|E_x|\le1$ by
\eqref{eq:tq-Eharmless}.  The values of $-(2x+1)$ at the two points are
$2\omega-3$ and $2\omega+3$, so in both cases
\[
 -\frac{2x+1}{\pi^2}\ln(2\pi c)
 =\frac{2\omega}{\pi^2}\ln(2\pi c)+O(\ln c).
\]
For the Barnes increment, the mean value theorem gives
$-\tfrac12\bigl(A(x+1)-A(x)\bigr)=-\tfrac12A'(\xi)$ with
$\xi\in(x,x+1)\subset[-\omega-2,-\omega+2]$.  By \eqref{eq:tq-A2all},
$-A'(\xi)=A'(-\xi)$ with $-\xi\in[\omega-2,\omega+2]$, and
\[
 |A'(-\xi)-A'(\omega)|\;\le\;2\sup_{[\omega-2,\omega+2]}|A''|
 \;\le\;2\bigl(\ln(4+\omega)+4\bigr),
\]
so $-\tfrac12\bigl(A(x+1)-A(x)\bigr)
=\tfrac12A'(\omega)+O\bigl(\ln(4+\omega)\bigr)$.  Since
$\omega\le A\ln c$, $\ln(4+\omega)=O_A(\ln\ln c)=O(\ln c)$.  Hence both
averages equal
$c+\frac{2\omega}{\pi^2}\ln(2\pi c)+\frac{A'(\omega)}2+O_A(\ln c)$,
which with Step~1 proves \eqref{eq:tq-preform}.

\emph{Step 3: the Barnes term.}  By Lemma~\ref{lem:barnes1} and
Lemma~\ref{lem:barnes2}(i), with $x=\omega/\pi$ and
$\ell(x)=\tfrac12\ln(1+x^2)$,
\[
 \frac{A'(\omega)}2=\frac{2\omega}{\pi^2}\bigl[1+\gamma_E-S(x)\bigr]
 =-\frac{2\omega}{\pi^2}\,\ell(x)+O(\omega)
 =-\frac{\omega}{\pi^2}\ln\Bigl(1+\frac{\omega^2}{\pi^2}\Bigr)+O(\omega),
\]
since $0\le S(x)-\ell(x)\le1$ and $1+\gamma_E\le2$.  Substituting into
\eqref{eq:tq-preform} gives \eqref{eq:tq-preform2}.
\end{proof}

\begin{remark}[the exact mirror of \eqref{eq:NvModel}]\label{rem:tq-mirror}
Formula \eqref{eq:tq-preform2} is \eqref{eq:NvModel} evaluated at
$v=-\omega$: the head-side display
$\Ntil(v)=c-\frac{2v}{\pi^2}\ln(2\pi c)
+\frac v{\pi^2}\ln(1+\frac{v^2}{\pi^2})+O(\ln c+v)$ becomes
\eqref{eq:tq-preform2} under $v\mapsto-\omega$.  The two are separate
theorems, resting on Theorem~\ref{qt:BDIK} and
Theorem~\ref{thm:signedmain} respectively; the coincidence of the
displays is a consequence of Remark~\ref{rem:tq-onefla} and is not used
as a substitute for either proof.
\end{remark}

\subsection{Proof of Theorem~\ref{thm:tail} and
Corollary~\ref{cor:plunge}}\label{sec:tq-quantile}

\begin{proof}[Proof of Theorem~\ref{thm:tail}]
Fix $A>0$ and let $c$ and $\delta$ satisfy \eqref{eq:tailrange}.  Put
$\omega=\Lb/2$, so that $3\le\omega\le\tfrac A2\ln c$; the range
\eqref{eq:tailrange} is exactly this range of $\omega$.  By
\eqref{eq:tq-depth}, $\theta(-\omega)=(1+e^{2\omega})^{-1}
=\bigl(1+\tfrac{1-\delta}\delta\bigr)^{-1}=\delta$, so by
\eqref{eq:logodds} $N_\delta(c)=\Ntil(-\omega)$ exactly.

Apply Proposition~\ref{prop:tq-preform} with its fixed parameter equal
to $A/2$:
\[
 N_\delta(c)=c+\frac{2\omega}{\pi^2}\ln(2\pi c)
 -\frac{\omega}{\pi^2}\ln\Bigl(1+\frac{\omega^2}{\pi^2}\Bigr)
 +O_A(\ln c+\omega).
\]
With $\Lb=2\omega$, Step~4 of \S\ref{sec:thmA} is an identity in $\Lb$
and applies unchanged:
\[
 \frac{2\omega}{\pi^2}\ln(2\pi c)
 -\frac{\omega}{\pi^2}\ln\Bigl(1+\frac{\omega^2}{\pi^2}\Bigr)
 =\frac{\Lb}{2\pi^2}\ln\frac{16\pi^4c^2}{4\pi^2+\Lb^2}
 =\frac{\Lb}{\pi^2}\ln\frac{4\pi^2c}{\sqrt{4\pi^2+\Lb^2}} ,
\]
and for $\Lb\ge6$,
$0\le\ln\sqrt{4\pi^2+\Lb^2}-\ln\Lb\le\tfrac12\ln(1+4\pi^2/36)\le1$,
so the last display differs from $\Phi_c(\Lb)$ by at most
$\pi^{-2}\Lb$.  Since $\omega\le\Lb$, this is \eqref{eq:TQ}.
\end{proof}

\begin{proof}[Proof of Corollary~\ref{cor:plunge}]
By \eqref{eq:identities}, exactly and with no symmetry assumption,
\begin{equation}\label{eq:tq-countid}
 D(\delta,c)=N_\delta(c)-N_{1/2}(c),
 \qquad
 \Lambda_\delta(c)=\Lambda^+_\delta(c)+D(\delta,c)
 =N_\delta(c)-N_{1-\delta}(c).
\end{equation}
Theorem~\ref{thm:tail} and the second assertion of
Theorem~\ref{thm:A}, $|N_{1/2}(c)-c|\le C_{\mathrm h}\ln c$, give
\eqref{eq:Dformula}.  For \eqref{eq:Pformula}, note that for a fixed $A$
one has $A\ln c\le c^{1/3}$ for all $c$ large, so \eqref{eq:tailrange}
is contained in the range $6\le\Lb\le c^{1/3}$ of
Theorem~\ref{thm:A}; that theorem gives
$N_{1-\delta}(c)=c-\Phi_c(\Lb)+O(\ln c+\Lb)$, and subtracting from
\eqref{eq:TQ} yields \eqref{eq:Pformula}.  No relation between the
spectrum near $0$ and near $1$ is used: the two counts are established
independently, one from Theorem~\ref{thm:signedmain} and one from
Theorem~\ref{qt:BDIK}.  The relative statements are exactly those of
Remark~\ref{rem:relativelimit}.
\end{proof}

\subsection{Tensor selection and the one-sided surface
bound}\label{sec:tq-tensor}

\begin{lemma}[fixed-threshold cluster count]\label{lem:tq-fixedq}
For each fixed $q\in(\tfrac12,1)$ there are $c_q,C_q<\infty$ with
\begin{equation}\label{eq:tq-fixedq}
 |N_q(c)-c|\;\le\;C_q\ln c\qquad(c\ge c_q).
\end{equation}
\end{lemma}

\begin{proof}
Since $q>\tfrac12$, $N_{1/2}(c)-N_q(c)=\#\{n:\tfrac12<\lambda_n\le q\}$.
Put $\eps_q=\tfrac{1-q}2\in(0,\tfrac12)$.  Then $\eps_q<\tfrac12$ and
$q<1-\eps_q$, so
$\{n:\tfrac12<\lambda_n\le q\}\subseteq\{n:\eps_q<\lambda_n<1-\eps_q\}$,
and Theorem~\ref{qt:KRD} bounds the latter by
$\frac2{\pi^2}\ln(50c+25)\ln\frac5{\eps_q(1-\eps_q)}+7=O_q(\ln c)$.
Combining with $|N_{1/2}(c)-c|\le C_{\mathrm h}\ln c$ from
Theorem~\ref{thm:A} gives \eqref{eq:tq-fixedq}.
\end{proof}

\begin{proof}[Proof of Theorem~\ref{thm:tensorblock}]
Write
\[
 \mathcal D'=\bigl\{n:\delta'<\lambda_n(c)\le\tfrac12\bigr\},
 \qquad
 \mathcal Q=\bigl\{n:\lambda_n(c)>q\bigr\},
\]
so that $\#\mathcal D'=D(\delta',c)$ by \eqref{eq:counts} and
$\#\mathcal Q=N_q(c)$.  Since $q>\tfrac12$ these two index sets are
disjoint.  For $1\le j\le d$ let $\mathcal F_j$ be the set of tensor
slots $\mathbf n=(n_1,\dots,n_d)$ with $n_j\in\mathcal D'$ and
$n_r\in\mathcal Q$ for every $r\ne j$.

\emph{Membership.}  For $\mathbf n\in\mathcal F_j$,
$\mu_{\mathbf n}=\lambda_{n_j}\prod_{r\ne j}\lambda_{n_r}
>\delta'q^{\,d-1}=\delta$, and $\mu_{\mathbf n}\le\lambda_{n_j}\le\tfrac12
<1-\delta$ because $\lambda_n<1$ for every $n$ and $\delta<\tfrac12$.
Hence $\mathcal F_j\subseteq\{\mathbf n:\delta<\mu_{\mathbf n}\le1-\delta\}$
for every $j$.

\emph{Disjointness.}  For $d=1$ there is only the family
$\mathcal F_1$.  For $d\ge2$ and $j\ne j'$, a slot in $\mathcal F_j$ has
$n_{j'}\in\mathcal Q$ while a slot in $\mathcal F_{j'}$ has
$n_{j'}\in\mathcal D'$; since $\mathcal D'\cap\mathcal Q=\varnothing$,
$\mathcal F_j\cap\mathcal F_{j'}=\varnothing$.

\emph{Counting.}  By Lemma~\ref{lem:tensor} the slots are indexed by
multi-indices with multiplicity, so
$\#\mathcal F_j=\#\mathcal D'\cdot(\#\mathcal Q)^{d-1}
=D(\delta',c)N_q(c)^{d-1}$ for each $j$.  Summing the $d$ disjoint
families gives \eqref{eq:tensorblock}.  Nothing asymptotic was used, so
the inequality is exact for every $c>0$.
\end{proof}

\begin{proof}[Proof of Theorem~\ref{thm:surface}]
Fix $A>0$, $d\ge1$ and $q\in(\tfrac12,1)$, and let $c$ and $\delta$
satisfy \eqref{eq:surfrange}.  Since $\Lb=\ln\frac{1-\delta}\delta$ we
have exactly $\delta=(1+e^{\Lb})^{-1}$, so with $\beta=q^{-(d-1)}\ge1$
and $\delta'=\beta\delta$,
\begin{equation}\label{eq:tq-Lprime}
 \frac{1-\delta'}{\delta'}=\frac{1+e^{\Lb}-\beta}{\beta},
 \qquad
 \Lb'\;:=\;\ln\frac{1-\delta'}{\delta'}
 \;=\;\Lb-\ln\beta+\ln\bigl(1-(\beta-1)e^{-\Lb}\bigr).
\end{equation}

\emph{Step 1: the lower threshold is exactly the admissibility of
$\delta'$.}  From the first identity in \eqref{eq:tq-Lprime},
\[
 \Lb'\ge6
 \iff 1+e^{\Lb}-\beta\ge\beta e^6
 \iff e^{\Lb}\ge\beta(e^6+1)-1
 \iff \Lb\ge L_{d,q},
\]
with $L_{d,q}$ as in \eqref{eq:Ldq}; the quantity
$\beta(e^6+1)-1\ge e^6>0$, so $L_{d,q}$ is well defined.  Thus the lower
restriction in \eqref{eq:surfrange} is precisely the statement
$\Lb'\ge6$, and in particular $\delta'=(1+e^{\Lb'})^{-1}<\tfrac12$, so
Theorem~\ref{thm:tensorblock} applies.  Moreover $\Lb'\le\Lb$ because
$\beta\ge1$; and since
$e^{\Lb}\ge\beta(e^6+1)-1\ge\beta e^6$ we get
$(\beta-1)e^{-\Lb}\le\beta^{-1}(\beta-1)e^{-6}\le e^{-6}$, whence
$\ln\bigl(1-(\beta-1)e^{-\Lb}\bigr)\ge\ln(1-e^{-6})\ge-2e^{-6}$ and
\begin{equation}\label{eq:tq-Lprimeclose}
 \Lb-\ln\beta-2e^{-6}\;\le\;\Lb'\;\le\;\Lb,
 \qquad\text{so}\qquad
 |\Lb-\Lb'|\le\ln\beta+1=O_{d,q}(1).
\end{equation}

\emph{Step 2: transporting the lower-half formula.}  Both $\Lb'$ and
$\Lb$ lie in $[6,A\ln c]$ by Step~1 and \eqref{eq:surfrange}.  On that
interval $\Phi_c'(x)=\pi^{-2}\bigl(\ln\frac{4\pi^2c}x-1\bigr)$ satisfies
$|\Phi_c'(x)|\le\pi^{-2}\bigl(\ln\frac{4\pi^2c}6+1\bigr)\le C\ln c$ for
$c\ge3$, so by \eqref{eq:tq-Lprimeclose} and the mean value theorem
\begin{equation}\label{eq:tq-phiprime}
 \Phi_c(\Lb')=\Phi_c(\Lb)+O_{A,d,q}(\ln c).
\end{equation}
Applying \eqref{eq:Dformula} to $\delta'$, whose depth is $\Lb'$, and
using \eqref{eq:tq-phiprime} together with $\Lb'\le\Lb$,
\begin{equation}\label{eq:tq-Dprime}
 D(\delta',c)=\Phi_c(\Lb)+O_{A,d,q}(\ln c+\Lb).
\end{equation}
In particular $0\le D(\delta',c)\le C_{A,d,q}\ln^2c$ on the range, since
$\Phi_c(\Lb)\le\pi^{-2}A\ln c\cdot\ln(4\pi^2c)$.

\emph{Step 3: the cluster factor.}  By Lemma~\ref{lem:tq-fixedq},
$N_q(c)\ge c-C_q\ln c\ge0$ for $c$ large, so for $d\ge2$, by Bernoulli's
inequality,
\[
 N_q(c)^{d-1}\;\ge\;c^{d-1}\Bigl(1-\frac{C_q\ln c}{c}\Bigr)^{d-1}
 \;\ge\;c^{d-1}-(d-1)C_qc^{d-2}\ln c ,
\]
and for $d=1$ both sides equal $1$.  Combining with Step~2 and
$D(\delta',c)\ge0$,
\[
 d\,N_q(c)^{d-1}D(\delta',c)
 \;\ge\;d\,c^{d-1}D(\delta',c)
 -d(d-1)C_qc^{d-2}\ln c\cdot C_{A,d,q}\ln^2c ,
\]
and $c^{d-2}\ln^3c\le c^{d-1}$ for $c$ large.

\emph{Step 4: conclusion.}  By Theorem~\ref{thm:tensorblock},
\eqref{eq:tq-Dprime} and Step~3,
\[
 \Lambda_\delta(c;d)\;\ge\;d\,c^{d-1}\Phi_c(\Lb)
 -C_{A,d,q}\,c^{d-1}(\ln c+\Lb),
\]
which is \eqref{eq:surface} since
$d\,\Phi_c(\Lb)=\frac d{\pi^2}\Lb\ln\frac{4\pi^2c}\Lb$.  Finally,
dividing \eqref{eq:surface} by its main term and using the estimate of
Remark~\ref{rem:relativelimit},
\[
 \frac{c^{d-1}(\ln c+\Lb)}
 {c^{d-1}\Lb\ln(4\pi^2c/\Lb)}
 =O_A\Bigl(\frac1\Lb+\frac1{\ln c}\Bigr)
 \longrightarrow0
\]
uniformly on any subrange \eqref{eq:relrange} with $L_0(c)\ge L_{d,q}$
and $L_0(c)\to\infty$, which is \eqref{eq:surfacerel}.
\end{proof}

\subsection{Scalar inversion of the main
term}\label{sec:tq-lambert}

\begin{proof}[Proof of Lemma~\ref{lem:lambert}]
Put $y=x/(4\pi^2c)$, so that by \eqref{eq:Phiintro}
\[
 \Phi_c(x)=\frac{4\pi^2cy}{\pi^2}\ln\frac1y=4c\,y\ln\frac1y ,
 \qquad
 x\in(0,4\pi^2c/e)\iff y\in(0,e^{-1}).
\]
The function $y\mapsto y\ln(1/y)$ is strictly increasing on
$(0,e^{-1})$, with range $(0,e^{-1})$; hence $m=\Phi_c(x)$ has exactly
one solution in $(0,4\pi^2c/e)$ precisely when $0<m/(4c)<e^{-1}$, i.e.\
when $0<m<4c/e$, and that solution is determined by
\[
 y\ln\frac1y=\frac m{4c}.
\]
Write $w=\ln y\in(-\infty,-1)$.  Then $y\ln(1/y)=-we^{w}$, so the
equation reads $we^{w}=-m/(4c)$ with $-m/(4c)\in(-e^{-1},0)$ and
$w\le-1$; by definition of the branches of the Lambert $W$ function this
is $w=W_{-1}\bigl(-m/(4c)\bigr)$.  Finally $e^{w}=-m/(4cw)$, so
\[
 x=4\pi^2c\,y=4\pi^2c\,e^{w}=-\frac{\pi^2m}{w}
 =-\frac{\pi^2m}{W_{-1}\bigl(-m/(4c)\bigr)} ,
\]
which is \eqref{eq:lambert}.  Substituting back verifies the identity.
\end{proof}

\begin{remark}[no individual-eigenvalue statement is claimed]
\label{rem:nolambertthm}
Lemma~\ref{lem:lambert} is an exact identity about the continuous
function $\Phi_c$ and about nothing else.  It is deliberately
\emph{not} converted into a statement about $\lambda_n(c)$ for an
individual index $n$.  Such a conversion would require: a generalized
inverse of the integer-valued, piecewise constant map
$\delta\mapsto N_\delta(c)$; an explicit admissible range of indices $n$
in which that inverse is meaningful; a convention resolving the unit
jumps of $N_\delta(c)$; and the propagation of the additive counting
error $O_A(\ln c+\Lb)$ of \eqref{eq:TQ} through the inverse, which on the
present range is of the same order as the displacement being measured
unless $\Lb\to\infty$.  None of these is carried out here, and no
individual-eigenvalue theorem is asserted.
\end{remark}

\section{Summary of parameter ranges}\label{sec:map}

\begin{center}
\begin{tabularx}{\linewidth}{@{}>{\raggedright\arraybackslash}p{.22\linewidth}>{\raggedright\arraybackslash}p{.31\linewidth}>{\raggedright\arraybackslash}X@{}}
\toprule
Range & Result/count & Source or conclusion\\
\midrule
fixed $\delta$ & $\pi^{-2}\ln((1-\delta)/\delta)\ln c+o(\ln c)$ & Landau--Widom \cite{LandauWidom}\\
$\alpha_1^{-c}<\delta<c^{-\alpha_1}$ & $\asymp L\ln(\alpha_1c/L)$ & Kulikov--Dam Larsen \cite{KDL}\\
$c^{-\alpha_1}\le\delta\le\delta_0$ & lower bound \eqref{eq:bridgeclosed}; explicit upper bound in \eqref{eq:upperD-a} & both bounds proved\\
$6\le\Lb\le A\ln c$, $A$ fixed & $\Phi_c(\Lb)+O_A(\ln c+\Lb)$, an additive two-sided formula; the displayed error is uniformly relative on subranges with $L_0(c)\to\infty$, while fixed $\delta$ is covered separately by Landau--Widom & Corollary~\ref{cor:plunge}, from Theorem~\ref{thm:tail}\\
upper half, $L_0\le L\le c^{1/3}$ & relative bounds with coefficient $\pi^{-2}$; an asymptotic if $L\to\infty$ & proved in Theorem~\ref{thm:B}\\
full plunge, $6\le\Lb\le A\ln c$ & $\Lambda_\delta(c)=2\Phi_c(\Lb)+O_A(\ln c+\Lb)$ & Corollary~\ref{cor:plunge}\\
$d\ge2$, cube pair & Corollary~\ref{cor:tensor} gives a tensor-product block; fixed-window asymptotics plus monotonicity give $\gtrsim c^{d-1}\ln c$ uniformly for $\delta\le\delta_\ast$ & no priority or sharpness claim\\
$d\ge1$ (nontrivial tensor case $d\ge2$), cube pair, one tail coordinate & $\Lambda_\delta(c;d)\ge\frac d{\pi^2}c^{d-1}\Lb\ln\frac{4\pi^2c}{\Lb}-C_{A,d,q}c^{d-1}(\ln c+\Lb)$ on $L_{d,q}\le\Lb\le A\ln c$ & Theorem~\ref{thm:surface}; one-sided only, no matching upper coefficient\\
\bottomrule
\end{tabularx}
\end{center}

\section{Conclusion}\label{sec:verdict}

The head-side determinant asymptotic gives the uniform upper-half
quantile formula and its tensor-product and joint-range consequences.
The signed Riemann--Hilbert analysis of Section~\ref{sec:tail} provides
the corresponding logarithmic window on the side $\gamma<0$.  A key
local step is Lemma~\ref{qmi:third}, which derives the third-ray jump,
the origin-sector multipliers, and the analytic prefactor from the
defining Kummer functions.  Combined with the two-way reduction, the
signed determinant theorem yields the lower-half estimate with constant
$1/(32\pi^2)$.  The deep-range theorem of \cite{KDL} then completes the
range in Problem~\ref{op:bridge}.

Section~\ref{sec:tailquant} runs the counting argument of
Sections~\ref{sec:smoothed}--\ref{sec:thmA} on the negative half of the
determinant parameter, with Theorem~\ref{thm:signedmain} replacing
Theorem~\ref{qt:BDIK}.  This gives the tail-side quantile formula
$N_\delta(c)=c+\Phi_c(\Lb)+O_A(\ln c+\Lb)$ on $6\le\Lb\le A\ln c$, hence
two-sided additive formulas for the lower half and for the full plunge
on the same range, and --- through an exact one-tail-coordinate tensor
selection --- a one-sided surface-order lower bound
$\Lambda_\delta(c;d)\ge\frac d{\pi^2}c^{d-1}\Lb\ln\frac{4\pi^2c}\Lb
-C_{A,d,q}c^{d-1}(\ln c+\Lb)$ for fixed $A$, $d$ and $q$.

The error
$O_A((1+\omega)^4\ln^2s/s)$ of Theorem~\ref{thm:signedmain} is weaker
than the head-side $O((v+v^3)/s)$ of \cite[Thm.~1.2]{BDIK2}, and the
admissible range $\omega\le A\ln s$ is far shorter than $v<s^{1/3}$;
both are consequences of the shrinking endpoint disks
$r_s=(\ln s)^{-2}$, and neither affects the bridge application.  Also,
the constants $C_A$, and hence the threshold $c_T(\alpha)$, are not made
numerically explicit (Remark~\ref{rem:threshold}).  Finally, the
$\Theta(L)$ term in the joint-range upper bound
\eqref{eq:upperD-a} is not removable by the present method.  The
joint-range estimate therefore gives no uniform bounded-depth relative
conclusion.  Likewise, the tail-side error bound yields a uniform
relative form on subranges with $\inf\Lb\to\infty$; it does not rule out
the fixed-threshold relative asymptotic supplied separately by
Landau--Widom (Remark~\ref{rem:relativelimit}).  The
tensor bound \eqref{eq:surface} is one-sided; its coefficient $d/\pi^2$
is not matched by any upper bound proved or cited here.

\appendix

\section{Uniform control of exponential factors}
\label{app:ledger}

Every factor of the form $e^{\pm\omega}$, $e^{\pm2\omega}$ or
$e^{C\omega\arg z}$ used in Section~\ref{sec:tail} is estimated below.
Throughout,
$\nu=-i\omega/\pi$, so $|e^{i\pi\nu}|=e^{\omega}$,
$|e^{-i\pi\nu}|=e^{-\omega}$, $|e^{2\pi i\nu}|=e^{2\omega}$ and
$|e^{-3i\pi\nu}|=e^{-3\omega}$.

\begin{center}
\small
\begin{tabularx}{\linewidth}{@{}>{\raggedright\arraybackslash}p{.10\linewidth}
  >{\raggedright\arraybackslash}p{.26\linewidth}
  >{\raggedright\arraybackslash}p{.17\linewidth}
  >{\raggedright\arraybackslash}X@{}}
\toprule
Term & Where it occurs & Modulus & Estimate or cancellation\\
\midrule
E1 & $\gamma e^{2v}$, lens coefficient in
\eqref{eq:factorization} & $\le1$ & Bounded outright by
\eqref{eq:lenscoef}; the sign works in our favour.  No cancellation
needed.\\
E2 & $\bigl(\frac{\lambda-1}{\lambda+1}\bigr)^{\pm2\nu}$, conjugation by
$P^{(\infty)}$ on the lens & $\le e^{2\omega}$ & Dominated by
$e^{-sr_s}$; net $O_{A,N}(s^{-N})$
(Lemma~\ref{lem:lensnegligible}).\\
E3 & $\beta^{\pm2}=[2s(\lambda+1)]^{\pm2\nu}$ on
$\partial D_1$ & $e^{\pm\frac{2\omega}\pi\arg(\lambda+1)}$ &
Reduces to $e^{O(\omega r_s)}=e^{O_A(1/\ln s)}=O_A(1)$, because
$|\arg(\lambda+1)|\le r_s$ there (Lemma~\ref{lem:Delta}, Step 4).\\
E4 & $e^{i\pi\nu}\Gamma(1+\nu)/\Gamma(-\nu)$ in $\mathcal F_{12}$ &
$|\nu|e^{\omega}$ & Cancels exactly against $a^{2}$ of the diagonal
dressing (Lemma~\ref{lem:Delta}, Step 2): $a^2M_1^{12}
=i\nu r_\nu\beta^{-2}e^{2is}$ has modulus $|\nu|\,|\beta^{-2}|$.\\
E5 & $e^{i\pi\nu}\Gamma(1-\nu)/\Gamma(\nu)$ in $\mathcal F_{21}$ &
$|\nu|e^{\omega}$ & Same, with $a^{-2}$; see E4 mirrored.\\
E6 & $a^{\pm2}$, where
$a=\beta^{-1}e^{-\frac{i\pi}2(\frac12-\nu)}e^{is}$ &
$e^{\pm\omega}|\beta^{\mp2}|$ & Never estimated alone; always paired
with $M_k^{12}$ or $M_k^{21}$ as in E4/E5.  This pairing is forced by
the identity $\mathcal A\Lambda=P^{(\infty)}$ of
Lemma~\ref{lem:AL}.\\
E7 & $\tau_\pm$, the Stokes multipliers of \eqref{eq:sectors} &
$\le2e^{2\omega}$ & Cancels identically against the exponential term of
the across-the-cut formula \cite[eq.~13.2.44]{NIST}
(Lemma~\ref{lem:stokes}).  Never estimated.\\
E8 & $|z^{-a}|=e^{(\omega/\pi)\arg z}$, up to $e^{3\omega/2}$, in
DLMF \eqref{eq:dlmfexp} & $e^{O(\omega)}$ & Divides out exactly:
the DLMF error bound in \eqref{eq:dlmfexp} carries the same $z^{-a}$, so
only the relative error $z^{a}\varepsilon_n(z)$ of
\eqref{eq:relerr} is used.\\
E9 & $r_\nu^{\pm1}=(\Gamma(\nu)/\Gamma(-\nu))^{\pm1}$ &
$=1$ exactly & Modulus one for purely imaginary $\nu$
(Lemma~\ref{lem:gammaratio}).\\
E10 & $e^{\pm2is}$, $e^{\pm is}$, $e^{\pm i\zeta/2}$ on real contours &
$=1$ & Unimodular; carried through symbolically.\\
\bottomrule
\end{tabularx}
\end{center}

\noindent
\textbf{Propagation of the error estimates.}  The chain of error sizes is
\begin{align*}
 \|E_\pm\|=O_A\Bigl(\frac{(1+\omega)^4}{s^2r_s^2}\Bigr)
 &\ \Longrightarrow\
 R_1\text{ error}=O_A\Bigl(\frac{(1+\omega)^4}{s^2r_s}\Bigr)\\
 &\ \Longrightarrow\
 \partial_s\ln\mathcal D\text{ error}
 =O_A\Bigl(\frac{(1+\omega)^4\ln^2s}{s^2}\Bigr)\\
 &\ \Longrightarrow\
 \ln\mathcal D\text{ error}
 =O_A\Bigl(\frac{(1+\omega)^4\ln^2s}{s}\Bigr),
\end{align*}
using $r_s=(\ln s)^{-2}$, the circle length $4\pi r_s$ at the first
implication (Lemma~\ref{lem:R1}), the factor $2$ from
\eqref{eq:diffid} at the second, and Step 4 of the proof of
Theorem~\ref{thm:signedmain} at the third.  The $L^2$ contribution
$\|\mu-I\|_{L^2}\|W\|_{L^2}\le2C_\ast\|W\|_{L^2}^2
=O_A((1+\omega)^4/(s^2r_s))$ enters at the same order as the $E_\pm$
contribution and no worse.

\section*{Acknowledgments}

AI-assisted tools were used during manuscript preparation for language
editing, bibliographic cross-checking, and limited symbolic and numerical
consistency checks.  Such checks were used only as supporting verification
and not as substitutes for mathematical proof.  The author reviewed the
resulting manuscript and accepts full responsibility for all statements,
derivations, citations, and conclusions.

\bibliographystyle{plain}
\bibliography{references}

\par\bigskip
\noindent\textsc{Ahmadreza Azimifard}\\
Harmonic Research \& Technologies, LLC\\
New York, NY, USA\\
\textit{Email:} \href{mailto:afard@harmonicrt.com}{afard@harmonicrt.com}

\end{document}